\documentclass[12pt, reqno]{amsart}
\usepackage{amssymb}
\usepackage{graphicx}
\usepackage{xcolor} 

\usepackage{fullpage} 

\definecolor{green}{rgb}{0,0.8,0} 

\newtheorem{theorem}{Theorem}[section]
\newtheorem{corollary}[theorem]{Corollary}
\newtheorem{lemma}[theorem]{Lemma}
\newtheorem{proposition}[theorem]{Proposition}
\theoremstyle{definition}
\newtheorem{definition}[theorem]{Definition}

\newtheorem{claim}{Claim}
\theoremstyle{remark}
\newtheorem{remark}[theorem]{Remark}
\numberwithin{equation}{section}
\newcommand{\nrm}[1]{\Vert#1\Vert}
\newcommand{\abs}[1]{\vert#1\vert}

\newcommand{\set}[1]{\{#1\}}

\newcommand{\supp}{{\mathrm{supp}}}

\renewcommand{\Im}{\mathrm{Im}}

\newcommand{\aeq}{\approx}
\newcommand{\aleq}{\lesssim}

\newcommand{\lap}{\Dlt}

\newcommand{\ud}{\mathrm{d}}
\newcommand{\rd}{\partial}
\newcommand{\nb}{\nabla}

\newcommand{\bb}{\Big}

\newcommand{\0}{\emptyset}

\newcommand{\alp}{\alpha}
\newcommand{\bt}{\beta}
\newcommand{\gmm}{\gamma}
\newcommand{\Gmm}{\Gamma}
\newcommand{\dlt}{\delta}
\newcommand{\Dlt}{\Delta}
\newcommand{\eps}{\epsilon}

\newcommand{\lmb}{\lambda}

\newcommand{\sgm}{\sigma}

\newcommand{\omg}{\omega}
\newcommand{\Omg}{\Omega}

\newcommand{\zt}{\zeta}

\newcommand{\bfm}{{\bf m}}
\newcommand{\bfn}{{\bf n}}

\newcommand{\bfD}{{\bf D}}

\newcommand{\bbC}{\mathbb C}

\newcommand{\bbR}{\mathbb R}
\newcommand{\bbS}{\mathbb S}

\newcommand{\bbZ}{\mathbb Z}

\newcommand{\calA}{\mathcal A}

\newcommand{\calC}{\mathcal C}
\newcommand{\calD}{\mathcal D}
\newcommand{\calE}{\mathcal E}
\newcommand{\calF}{\mathcal F}
\newcommand{\calG}{\mathcal G}
\newcommand{\calH}{\mathcal H}

\newcommand{\calO}{\mathcal O}
\newcommand{\calP}{\mathcal P}
\newcommand{\calQ}{\mathcal Q}
\newcommand{\calR}{\mathcal R}
\newcommand{\calS}{\mathcal S}
\newcommand{\calT}{\mathcal T}

\newcommand{\calV}{\mathcal V}

\newcommand{\calY}{\mathcal Y}

\newcommand{\pfstep}[1]{\vspace{.5em} \noindent {\bf #1.} }
\DeclareMathOperator*{\esssup}{ess\,sup\,}

\newcommand{\extr}{\mathrm{ext}}
\newcommand{\intr}{\mathrm{int}}

\newcommand{\covD}{\bfD}

\newcommand{\diam}{\mathrm{diam} \,}
\newcommand{\met}{\bfm}

\newcommand{\bnrm}[1]{\bb\Vert #1 \bb\Vert}
\newcommand{\BMO}{\mathrm{BMO}}
\newcommand{\CG}{\calY}				
\newcommand{\wCG}{\widehat{\calY}}		
\newcommand{\firstDR}{\mathrm{H}^{1}_{\mathrm{deRham}}} 
\newcommand{\rst}{\!\upharpoonright}		
\newcommand{\ecs}{r_{\mathrm{c}}}			
\newcommand{\En}{E}					
\newcommand{\thE}{\eps_{\ast}}		
\newcommand{\rglue}{\sgm_{0}}				
\newcommand{\rgExt}{\sgm_{1}}				
\newcommand{\rgInt}{\sgm_{2}}				

\newcommand{\angnb}{\hskip-.3em \not \hskip-.25em \nb}						

\begin{document}

\title[]{Local well-posedness of the $(4+1)$-dimensional Maxwell-Klein-Gordon equation at energy regularity}
\author{Sung-Jin Oh}%
\address{Department of Mathematics, UC Berkeley, Berkeley, CA, 94720}%
\email{sjoh@math.berkeley.edu}%

\author{Daniel Tataru}%
\address{Department of Mathematics, UC Berkeley, Berkeley, CA, 94720}%
\email{tataru@math.berkeley.edu}%


\begin{abstract}
  This paper is the first part of a trilogy \cite{OT2, OT3} dedicated
  to a proof of global well-posedness and scattering of the
  $(4+1)$-dimensional mass-less Maxwell-Klein-Gordon equation (MKG)
  for any finite energy initial data. The main result of the present
  paper is a large energy local well-posedness theorem for MKG in the
  global Coulomb gauge, where the lifespan is bounded from below by
  the energy concentration scale of the data. Hence the proof of
  global well-posedness is reduced to establishing non-concentration
  of energy. To deal with non-local features of MKG we develop initial
  data excision and gluing techniques at critical regularity, which
  might be of independent interest.

\end{abstract} 
\maketitle

\tableofcontents

\section{Introduction}
Let $\bbR^{1+4}$ be the $(4+1)$-dimensional Minkowski space with the metric 
\begin{equation*}
\met_{\mu \nu} := \mathrm{diag}\,(-1,+1,+1,+1,+1)
\end{equation*}
in the standard rectilinear coordinates $(t=x^{0}, x^{1}, \cdots, x^{4})$. 
Let $L = \bbR^{1+4} \times \bbC$ be the trivial $\mathrm{U}(1)$ complex line bundle over $\bbR^{1+4}$.
The \emph{Maxwell-Klein-Gordon system} is a relativistic gauge field theory that describes the evolution of a pair $(A, \phi)$ of a connection on $L$ and a section of $L$. In Section~\ref{subsec:background}, we present the necessary background material concerning the Maxwell-Klein-Gordon system on $\bbR^{1+4}$. Readers already familiar with this equation may skip ahead to Section~\ref{subsec:main-results}, where the main results and ideas of the paper are presented.

\subsection{The Maxwell-Klein-Gordon system on $\bbR^{1+4}$} \label{subsec:background}
Let $L = \bbR^{1+4} \times \bbC$ be the trivial complex line bundle with structure group $\mathrm{U}(1) = \set{e^{i \chi} \in \bbC}$. 
Global sections of $L$ are precisely $\bbC$-valued functions on $\bbR^{1+4}$. 
Using the trivial connection on $\bbR^{1+4}$ as a reference and employing the identification $\mathrm{u}(1) \equiv i \bbR$, any connection $\covD_{\mu}$ on $L$ can be written as
\begin{equation*}
	\covD_{\mu} = \rd_{\mu} + i A_{\mu}
\end{equation*}
where $A_{\mu}$ is a real-valued 1-form on $\bbR^{1+4}$. 

The (mass-less) \emph{Maxwell-Klein-Gordon system} for a pair $(A, \phi)$ of a connection on $L$ and a section of $L$ takes the form
\begin{equation} \label{eq:MKG} \tag{MKG}
\left\{
\begin{aligned}
	\rd^{\mu} F_{\nu \mu} =& \Im(\phi \overline{\covD_{\nu} \phi})  \\
	\Box_{A} \phi =& 0,
\end{aligned}
\right.
\end{equation}
where $F_{\mu \nu} := (\ud A)_{\mu \nu} = \rd_{\mu} A_{\nu} -\rd_{\nu} A_{\mu}$ is the \emph{curvature 2-form} associated to $\covD_{\mu}$ and $\Box_{A} := \covD^{\mu} \covD_{\mu}$ is the covariant d'Alembertian.  We are using the usual convention of raising and lowering indices using the Minkowski metric, and also of summing over repeated upper and lower indices.

We consider the initial value problem for \eqref{eq:MKG}.
An \emph{initial data set} for \eqref{eq:MKG} consists of two pairs of 1-forms $(a_{j}, e_{j})$ and $\bbC$-valued functions $(f, g)$ on $\bbR^{4}$. We say that $(a_{j}, e_{j}, f, g)$ is the initial data for a solution $(A, \phi)$ if
\begin{equation*}
	(A_{j}, F_{0 j}, \phi, \covD_{t} \phi) \rst_{\set{t=0}} = (a_{j}, e_{j}, f, g).
\end{equation*}
Note that \eqref{eq:MKG} imposes the condition that the following equation be true for any initial data for \eqref{eq:MKG}:
\begin{equation} \label{eq:MKGconstraint}
	\rd^{j} e_{j} = \Im (f \overline{g}).
\end{equation}
This equation is the \emph{Gauss} (or the \emph{constraint}) \emph{equation} for \eqref{eq:MKG}. 

A basic geometric feature of the Maxwell-Klein-Gordon system is \emph{gauge invariance}. Let $\chi$ be a \emph{gauge transformation} for \eqref{eq:MKG}, i.e., a real-valued function on $\bbR^{1+4}$, so that $e^{i \chi} \in \mathrm{U}(1)$. Then \eqref{eq:MKG} is invariant under the associated gauge transform $(A, \phi) \mapsto (A - \ud \chi, e^{i \chi} \phi)$. Geometrically, a gauge transform corresponds to a change of basis in the fiber $\bbC$ of the complex line bundle $L$ over each point in $\bbR^{1+4}$. To establish any sort of well-posedness of the initial value problem and also to reveal the hyperbolicity\footnote{Observe that without any choice of gauge, the the principal part of $\rd^{\mu} F_{\nu \mu}$ is $-\Box A_{\nu} + \rd_{\nu} \rd^{\mu} A_{\mu}$, which does not have a well-defined character.} of \eqref{eq:MKG}, the ambiguity arising from this invariance must be fixed. For this purpose we rely on the \emph{global Coulomb gauge} condition $\sum_{j=1}^{4} \rd_{j} A_{j} = 0$ in this paper. 

The Maxwell-Klein-Gordon system on $\bbR^{1+4}$ obeys the law of \emph{conservation of energy}. The conserved energy of a solution $(A, \phi)$ at time $t$ is defined as
\begin{equation} \label{eq:energy-def}
	\calE_{\set{t} \times \bbR^{4}} [A, \phi] := \frac{1}{2} \int_{\set{t} \times \bbR^{4}} \sum_{0 \leq \mu < \nu \leq 4} \abs{F_{\mu \nu}}^{2} + \sum_{0 \leq \mu \leq 4} \abs{\covD_{\mu} \phi}^{2} \, \ud x.
\end{equation}
For any sufficiently regular solution to \eqref{eq:MKG} on $I \times \bbR^{4}$, where $I \subseteq \bbR$ is a connected interval, $\calE_{\set{t_{1}} \times \bbR^{4}}[A, \phi] = \calE_{\set{t_{2}} \times \bbR^{4}}[A, \phi]$ for every $t_{1}, t_{2} \in I$. For a \eqref{eq:MKG} initial data set $(a, e, f, g)$, the conserved energy takes the form
\begin{equation} \label{eq:energy-id-def}
	\calE_{\bbR^{4}}[a, e, f, g] = \frac{1}{2} \int_{\bbR^{4}} \sum_{1 \leq j < k \leq 4}^{4} \abs{\rd_{j} a_{k} - \rd_{k} a_{j}}^{2} + \sum_{j = 1}^{4} \abs{e_{j}}^{2} + \sum_{j=1}^{4} \abs{\covD_{j} f}^{2} + \abs{g}^{2} \, \ud x,
\end{equation}
where $\covD_{j} := \rd_{j} + i a_{j}$. Furthermore, given any (measurable) subset $O' \subseteq \bbR^{4}$, we define the local energy $\calE_{O'}[a, e, f, g]$ by replacing the domain of integral above by $O'$.

The Maxwell-Klein-Gordon system can in fact be formulated on any
$\bbR^{1+d}$ ($d \geq 1$). However, the $(4+1)$-dimensional case is
distinguished by the fact that the system becomes \emph{energy
  critical}. That is, in $\bbR^{1+4}$ both the conserved energy
\eqref{eq:energy-def} and the equations \eqref{eq:MKG} are invariant
under the scaling
\begin{equation*}
	(A, \phi) \mapsto (\widetilde{A}, \widetilde{\phi}) (t,x) := (\lmb^{-1} A, \lmb^{-1} \phi)(\lmb^{-1} t, \lmb^{-1} x) \quad \hbox{ for any } \lmb > 0.
\end{equation*}

\subsection{Main results and ideas} \label{subsec:main-results} The
present paper is the first of a sequence of three papers \cite{OT2,
  OT3}, in which we give a complete proof of global well-posedness and
scattering of \eqref{eq:MKG} on $\bbR^{1+4}$ for any finite energy
data. This theorem is analogous to the \emph{threshold theorem} for
energy critical wave maps \cite{Krieger:2009uy, MR2657817, MR2657818,
  Tao:2008wn, Tao:2008tz, Tao:2008wo, Tao:2009ta, Tao:2009ua}. The
main result of this paper is the following local well-posedness
theorem for \eqref{eq:MKG} in the global Coulomb gauge at the energy
regularity.

\begin{theorem}[Local well-posedness of \eqref{eq:MKG} at energy
  regularity, simple version] \label{thm:lwp4MKG:simple} Let $E$
  be any positive number and let $(a, e, f, g)$ be a smooth initial data set
  with energy $\leq E$ satisfying the global Coulomb condition
  $\sum_{j=1}^{4} \rd_{j} a_{j} = 0$.
  \begin{enumerate}
  \item Then there exists an open time interval $I \ni 0$ and a unique
    smooth solution $(A, \phi)$ to the initial value problem on $I
    \times \bbR^{4}$ satisfying the global Coulomb gauge condition
    $\sum_{j=1}^{4} \rd_{j} A_{j} = 0$.
  \item Define the \emph{energy concentration scale} of $(a, e, f, g)$
    by
    \begin{equation*}
      \ecs = \ecs(E)[a, e, f, g] := \sup\set{r > 0: \forall x \in \bbR^{4}, \ 
\calE_{B_{r}(x)} [a, e, f, g] < \dlt_{0}(E, \thE^{2})},
    \end{equation*}
    where $B_{r}(x)$ denotes the open ball of radius $r$ centered at
    $x$, $\thE$ is a universal constant (see Theorem~\ref{thm:KST}
    below) and $\dlt_{0}(E, \thE^{2})$ is some positive function (to be
    specified in Section~\ref{sec:lwp}). Then $I$ contains the
    interval $[-\ecs, \ecs]$.
  \item Finally, the solution map extends continuously on compact time intervals
 to general  finite energy initial data, with the same lifespan properties as in (2)
above. 
  \end{enumerate}
\end{theorem}
For a more precise version, see Theorem~\ref{thm:lwp4MKG}. We remark
that we do not lose any generality by restricting to initial data sets
in the global Coulomb gauge, as any finite energy initial data sets
can be gauge transformed into this gauge; see Section~\ref{sec:id}.
We formulate our local well-posedness theorem specifically in the
global Coulomb gauge in view of the rest of the series \cite{OT2,
  OT3}, where we show global well-posedness and scattering in this
gauge.

An important feature of Theorem~\ref{thm:lwp4MKG:simple} is that it
provides a lower bound on the lifespan in terms of the \emph{energy
  concentration scale} $\ecs$ of the data. Taking the contrapositive,
we see that any finite time blow up of a solution to \eqref{eq:MKG}
must be accompanied by energy concentration at a point. In
\cite{OT2, OT3}, following the scheme successfully developed by one of
the authors (D.~Tataru) and J.~Sterbenz in the context of energy
critical wave maps \cite{MR2657817, MR2657818}, we establish global
well-posedness of \eqref{eq:MKG} for finite energy data by showing
that such a phenomenon cannot occur. We refer to the last and the main
paper of the sequence \cite{OT3} for an overview of the entire series.

To prove Theorem~\ref{thm:lwp4MKG:simple}, we rely on the following
small energy global well-posedness theorem for the
Maxwell-Klein-Gordon equations in the global Coulomb gauge, which was
established recently by one of the authors (D. Tataru) jointly with J.~Krieger
and J.~Sterbenz.
\begin{theorem} [Small energy global well-posedness in Coulomb
  gauge \cite{Krieger:2012vj}] \label{thm:KST} There exists an $\thE > 0$
  such that the following holds. Let $(a, e, f, g)$ be a smooth
  initial data on $\bbR^{4}$ satisfying the global Coulomb gauge
  condition $\sum_{\ell=1}^{4} \rd_{\ell} a_{\ell} = 0$ and
  \begin{equation*}
    \calE_{\bbR^{4}} [a, e, f, g] \leq \thE^{2}.
  \end{equation*}
  \begin{enumerate}
  \item Then there exists a unique smooth global solution $(A, \phi)$
    to the initial value problem for \eqref{eq:MKG} on $\bbR^{1+4}$
    satisfying
    \begin{equation}
      \nrm{A_{0}}_{Y^{1}(\bbR^{1+4})} + \nrm{A_{x}}_{S^{1}(\bbR^{1+4})} + \nrm{\phi}_{S^{1}(\bbR^{1+4})} \aleq \sqrt{\calE_{\bbR^{4}} [a, e, f, g]},
    \end{equation}
    where $A_{x} = (A_{1}, \ldots, A_{4})$.
  \item For every compact time interval $I \subseteq \bbR$, the
    solution map extends continuously to general finite energy initial
    data after restriction\footnote{Although this continuity statement
      is not explicitly stated in \cite[Theorem 1]{Krieger:2012vj},
      its proof can be read off from
      \cite[Section~5.5]{Krieger:2012vj}. We remark that continuous
      dependence on the data in $\calH^{1}$ does not seem to hold in
      the global space $S^{1}(\bbR^{1+4})$, due to the strong dependence of
      the linear magnetic flow for $\Box_A$  on the low frequency part of
      $A_{x}$.} to $I \times \bbR^{4}$. More precisely, if
    $(a^{(n)}, e^{(n)}, f^{(n)}, g^{(n)})$ is a sequence of finite
    energy initial data sets in global Coulomb gauge whose limit is
    $(a, e, f, g)$ in $\calH^{1}$ (defined in
    Section~\ref{subsec:MKGid:1}), then
    \begin{equation}
      \nrm{A_{0}^{(n)} - A_{0}}_{Y^{1}(I \times \bbR^{4})} 
      + \nrm{A^{(n)}_{x} - A_{x}}_{S^{1}(I \times \bbR^{4})} 
      + \nrm{\phi^{(n)} - \phi}_{S^{1}(I \times \bbR^{4})} \to 0 \quad \hbox{ as } n \to \infty,
    \end{equation}
    where $(A^{(n)}, \phi^{(n)})$ is the global solution to
    \eqref{eq:MKG} with data $(a^{(n)}, e^{(n)}, f^{(n)}, g^{(n)})$.
  \end{enumerate}
\end{theorem}
More detailed descriptions of the function spaces $S^{1}$ and $Y^{1}$
will be given in Sections~\ref{sec:lwp} and \ref{sec:gtCutoff}. In
particular, $S^{1}$ is a delicate function space consisting of a
number of pieces, including the energy norm, a frequency localized
Strichartz norm, an $\dot{X}^{s, b}$-type norm and a null frame norm
as in the energy critical wave maps problem \cite{MR1827277,
  Tao:2001gb}. The precise version of the main local well-posedness
theorem (Theorem~\ref{thm:lwp4MKG}) also involves these spaces. At
this point we simply remark that for any interval $I \times \bbR^{4}$,
we have
\begin{equation*}
  \nrm{(\varphi, \rd_{t} \varphi)}_{C_{t}(I; \dot{H}^{1}_{x} \times L^{2}_{x})} \aleq \nrm{\varphi}_{S^{1}(I \times \bbR^{4})}, \quad
  \nrm{(\varphi, \rd_{t} \varphi)}_{C_{t}(I; \dot{H}^{1}_{x} \times L^{2}_{x})} \aleq \nrm{\varphi}_{Y^{1}(I \times \bbR^{4})}.
\end{equation*}

For a simpler energy critical semilinear wave equation, such as $\Box
u = \pm u^{\frac{d+2}{d-2}}$ on $\bbR^{1+d}$, a statement analogous to
Theorem~\ref{thm:lwp4MKG:simple} is an immediate consequence of the
small energy global well-posedness theorem (Theorem~\ref{thm:KST} in
our context) and the finite speed of propagation of the
system. Roughly speaking, the proof of local well-posedness (in
particular, local existence) proceeds in the following three steps
(see, for instance \cite[Section 5.1]{MR2233925}):
\begin{itemize}
\item [{\it Step 1.}] Truncation of the initial data set locally in
  space so that the energy becomes small;
\item [{\it Step 2.}] Application of small energy global
  well-posedness to produce the corresponding set of global solutions;
  and
\item [{\it Step 3.}] Patching together the resulting solutions via
  finite speed of propagation\footnote{More precisely, in Step $3$, by
    finite speed of propagation, note that the global solutions in
    Step $2$ restricted to the domain of dependence of the truncated
    regions in Step $1$ give rise to a family of local-in-space-time
    solutions, which agree with each other on the intersection of the
    domains.  }.
\end{itemize}

However, implementation of this strategy in our context is not as
straightforward due to \emph{non-local} features of the
Maxwell-Klein-Gordon system in the global Coulomb gauge. One source of
non-locality is the Gauss equation for initial data sets, which
forbids us from naively truncating initial data to reduce to the small
energy case. Another source is the global Coulomb gauge condition,
which imposes a Poisson (hence non-local) equation for the component
$A_{0}$ of the connection 1-form. In particular, finite speed of
propagation \emph{fails} in the global Coulomb gauge.

In this paper we develop techniques for overcoming such issues
concerning non-locality of the Maxwell-Klein-Gordon equations, and
employ them to prove Theorem~\ref{thm:lwp4MKG:simple} from
Theorem~\ref{thm:KST} by essentially carrying out Steps~$1$--$3$
above.  These techniques (in addition to
Theorem~\ref{thm:lwp4MKG:simple} itself) are also crucially used in
the last paper of the sequence \cite{OT3}, where we carry out a
blow-up analysis of \eqref{eq:MKG} to preclude concentration of energy
and non-scattering.

To deal with the non-locality of the Gauss equation, we introduce the
method of \emph{initial data excision and gluing} at critical
regularity for \eqref{eq:MKG}; see Propositions~\ref{prop:gluing} and
\ref{prop:intGluing} for the precise formulation. Instead of naively
truncating an initial data set $(a, e, f, g)$, which would violate the
Gauss equation, the idea is to \emph{excise} the unwanted part and
then \emph{glue} another solution to the Gauss equation with the
appropriate behavior. Similar techniques have been developed for the
initial data sets of the Einstein equations in general relativity
\cite{Corvino:2000du, Corvino:2006wf, MR2971203, MR2031583}. In our
context, we need to develop a sharp version that works at the critical
regularity. Our key tool is an explicit solution operator to the
divergence equation \cite{MR553920, MR631691, Isett:2014vk} which
preserves the compact support property; see
Proposition~\ref{prop:sol4div}.

The initial data excision and gluing technique allows us to carry out
an analogue of Step~$1$. Then applying suitable gauge transformations
to the resulting initial data sets to impose the global Coulomb gauge
condition, we are in position to use Theorem~\ref{thm:KST} to produce
the corresponding global solutions. This procedure is analogous to
Step~$2$. However, we face difficulty in patching these solutions in
the global Coulomb gauge (which corresponds to Step $3$), since finite
speed of propagation does not hold in this gauge.

We use two ideas for addressing this issue. The first is the
observation that even though finite speed of propagation may fail in a
particular gauge (e.g., the global Coulomb gauge), it remains true up
to a gauge transformation. We refer to this fact as the \emph{local
  geometric uniqueness} of \eqref{eq:MKG}; see
Proposition~\ref{prop:locGeom}. Hence we obtain from the global
solutions produced in Step $2$ a family of local-in-space-time
solutions $(A_{[\alp]}, \phi_{[\alp]})$ to \eqref{eq:MKG}, which agree
with each other on the intersection of the domains up to gauge
transformations. We call such solutions \emph{compatible pairs} (see
Definition~\ref{def:compatiblePairs}). Geometrically, these are
nothing but a description of a global pair of a connection 1-form 
and a section of $L$ in local trivializations.

The second idea is to patch these local descriptions together to form
a single solution in the global Coulomb gauge. We begin by adapting an
argument of Uhlenbeck \cite[Section 3]{Uhlenbeck:1982vna} to produce a
single global-in-space solution in the desired function spaces
$S^{1}$, $Y^{1}$; see Proposition~\ref{prop:patch}. For this purpose,
we develop a functional space framework for performing gauge
transforms between local-in-spacetime solutions in $S^{1}$ and
$Y^{1}$; see Section~\ref{subsec:ftnspace4patching} and
Section~\ref{sec:gtCutoff}. A key point in this argument is that a
gauge transformation $\chi$ between two Coulomb gauges obeys the
Laplace equation $\lap \chi = 0$, and hence enjoys improved
regularity. The solution resulting from this patching argument does
not necessarily satisfy the exact global Coulomb
condition. Nevertheless this solution is \emph{approximately Coulomb},
since it arose by patching together Coulomb solutions. Hence there
exists a nicely behaved gauge transformation into the global Coulomb
gauge, which completes the analogue of Step $3$ and hence the sketch
of our proof of Theorem~\ref{thm:lwp4MKG:simple}.

\begin{remark} 
  The main result and the techniques developed in this paper are
  perturbative in nature, and hence can be easily generalized to
  higher dimensions, i.e., $\bbR^{1+d}$ for any $d \geq 4$. In what
  follows we focus on the most interesting case $\bbR^{1+4}$ for
  concreteness.
\end{remark}
%
%
%

\subsection{Other works on the Maxwell-Klein-Gordon equations}
Here we give a brief review of the literature on the
Maxwell-Klein-Gordon problem.  
In dimensions $2+1$ and $3+1$ the Maxwell-Klein-Gordon system is \emph{energy subcritical}, so global
regularity follows from local well-posedness at the energy regularity; see Klainerman-Machedon
\cite{Klainerman:1994jb} and Selberg-Tesfahun \cite{Selberg:2010ig}. 
We also mention the works of Moncrief \cite{MR579231} and Eardley-Moncrief
\cite{MR649158, MR649159}, where global regularity of sufficiently smooth solutions in $\bbR^{1+2}$ and $\bbR^{1+3}$ was established by a
different argument; the latter two also handled the more general Yang-Mills-Higgs system on $\bbR^{1+3}$. 
The problem of low regularity well-posedness in $\bbR^{1+3}$ was further studied by Cuccagna \cite{Cu}
and then more recently by Machedon-Sterbenz \cite{Machedon:2004cu}, who reached the essentially optimal regularity $A(0), \phi(0) \in H^{\frac{1}{2}+}$.
In \cite{KRT}, global well-posedness was established below the energy norm, more precisely for $A(0), \phi(0) \in H^{\frac{\sqrt{3}}{2}+}$

In dimension $4+1$, Klainerman-Tataru \cite{Klainerman:1999do} established an essentially optimal local well-posedness result for a model equation
closely related to Maxwell-Klein-Gordon and Yang-Mills. This result was further refined by Selberg \cite{MR1916561}, who considered the full Maxwell-Klein-Gordon system on $\bbR^{1+4}$, and Sterbenz \cite{Ste}.

For the critical regularity problem, Rodnianski-Tao \cite{MR2100060}
made an initial breakthrough and proved global regularity for small
scaling critical Sobolev data in dimensions $6+1$ and higher. This
result was greatly improved in the aforementioned work of
Krieger-Sterbenz-Tataru \cite{Krieger:2012vj} to include the energy
critical dimension $(4+1)$, which provides the starting point of the
present paper.

Finally, we note that an independent proof of global well-posedness
and scattering of \eqref{eq:MKG} has recently been announced by
Krieger-L\"uhrmann, following a version of the Bahouri-G\'erard nonlinear profile decomposition \cite{MR1705001} 
and the Kenig-Merle concentration compactness/rigidity scheme \cite{MR2257393, MR2461508}, 
developed by Krieger-Schlag \cite{Krieger:2009uy} for the energy critical wave maps problem.

\subsection{The structure of the paper}
After some preliminaries in Section~\ref{sec:prelim}, we begin with a
systematic study of finite energy initial data sets for \eqref{eq:MKG}
in Section~\ref{sec:id}. We show, in particular, that every such
initial data set can be gauge transformed to the global Coulomb gauge
(Lemma~\ref{lem:gt2CoulombId}), and also that it can be approximated
by smooth data (Lemma~\ref{lem:idApprox}).  In
Section~\ref{sec:gluing}, we develop the theory of excision and gluing
of Maxwell-Klein-Gordon initial data sets at the energy regularity
(Propositions~\ref{prop:gluing}, \ref{prop:intGluing}).  In
Section~\ref{sec:locGeom}, we formulate a notion of solutions to
\eqref{eq:MKG} arising from general finite energy initial data
(admissible $C_{t} \calH^{1}$ solutions) and prove local geometric
uniqueness of \eqref{eq:MKG} in this class
(Proposition~\ref{prop:locGeom}).  In Section~\ref{sec:lwp}, we give a
precise statement of the main local well-posedness theorem
(Theorem~\ref{thm:lwp4MKG}) and prove it up to some estimates
concerning the functions spaces $S^{1}, Y^{1}$.  Finally, in
Section~\ref{sec:gtCutoff} we delve further into the structure of the
spaces $S^{1}, Y^{1}$ and establish the function space estimates used
in Section~\ref{sec:lwp}, thereby completing the proof of
Theorem~\ref{thm:lwp4MKG}.

\subsection*{Acknowledgements}
The authors thank Phil Isett for helpful discussions regarding the
divergence equation, and in particular for communicating the elegant
construction in Proposition~\ref{prop:sol4div}. Part of this work was carried out during the trimester program `Harmonic Analysis and PDEs' at the Hausdorff Institute of Mathematics in Bonn. S.-J. Oh is a Miller Research Fellow, and thanks the Miller Institute for support. D. Tataru was partially supported by the NSF grant
DMS-1266182 as well as by the Simons Investigator grant from the
Simons Foundation.

\section{Preliminaries} \label{sec:prelim}
\subsection{Notation and conventions}

We write $A \aleq B$ when there exists a constant $C > 0$ such that $A
\leq C B$. The dependence of the constant is specified by a subscript,
e.g., $A \aleq_{r} B$ means that there exists $C = C(r) > 0$ such that
$A \leq C B$. We write $A \aeq B$ when both $A \aleq B$ and $B \aleq
A$ hold.

We employ the index notation in this paper. Unless otherwise
specified, we always use the rectilinear coordinates $(t = x^{0},
x^{1}, \ldots, x^{4})$. The greek indices (e.g., $\mu, \nu, \ldots$)
run over $0, 1, \ldots, 4$, whereas the roman indices only run over
$1, \ldots, 4$. As already mentioned in the introduction, we raise and
lower indices using the Minkowski metric $\met_{\mu \nu}$, and use the
convention of summing up repeated upper and lower indices.

We denote the open ball in $\bbR^{4}$ of radius $r$ and center $x$ by
$B_{r}(x)$.  Given a cube $R \subseteq \bbR^{4}$, we refer to its side length by
$\ell(R)$. For a convex subset $K$ of $\bbR^{4}$ (or $\bbR^{1+4}$)
and $c \in (0, \infty)$, we define $c K$ to be the dilation of $K$ by
$c$ about the center of mass of $K$. For example, if $B_{r}(x)$ is an
open ball in $\bbR^{4}$, then $c B_{r}(x)$ is the open ball with the
same center and the radius $c$ times that of $B$, i.e., $c B_{r}(x) =
B_{cr}(x)$.

\subsection{Dyadic frequency projections}
Let $m_{\leq 0}(r)$ be a smooth cutoff which equals 1 on $\set{r \leq
  1}$ and vanishes outside $\set{r \geq 2}$. For every $k \in \bbZ$,
define $m_{\leq k}(r) := m_{\leq 0}(r / 2^{k})$ and $m_{k}(r) :=
m_{\leq k}(r) - m_{\leq k-1}(r)$. Then $m_{k}$ is supported in the set
$\set{2^{k-1} \leq r \leq 2^{k+1}}$ and forms a partition of unity,
i.e.,
\begin{equation*}
	\sum_{k} m_{k} (r) 
	= 1.
\end{equation*}

The following dyadic frequency (or \emph{Littlewood-Paley})
projections are used in this paper:
\begin{align*}
	P_{k} \varphi =& \calF^{-1} [m_{k}(\abs{\xi}) \calF[\varphi]],	&\quad
	Q_{j} \varphi =& \calF^{-1} [m_{j}(\abs{\abs{\tau} - \abs{\xi}}) \calF[\varphi]], \\
	S_{\ell} \varphi =& \calF^{-1} [m_{\ell}(\abs{(\tau, \xi)}) \calF[\varphi]],	&\quad
	T_{j} \varphi =& \calF^{-1} [ m_{j}(\abs{\tau}) \calF[\varphi]] .
\end{align*}
We also use the notation $P_{\leq k} := \sum_{k' \leq k} P_{k}$, 
$P_{(k_{1}, k_{2}]} := \sum_{k' \in (k_{1}, k_{2}]} P_{k'}$ etc.

\subsection{Standard functions spaces on $\bbR^{d}$ and domains}
Unless otherwise specified, we define function spaces on a subset $\calO \subseteq \bbR^{d}$ by restricting the $\bbR^{d}$ version, i.e.,
\begin{equation*}
	\nrm{\varphi}_{X(\calO)} := \inf_{\psi = \varphi \hbox{ on } \calO} \nrm{\psi}_{X(\bbR^{d})}.
\end{equation*}

The homogeneous Sobolev and Besov semi-norms $\nrm{\cdot}_{\dot{W}^{s, p}(\bbR^{d})}$, $\nrm{\cdot}_{\dot{B}^{s, p}_{r}(\bbR^{d})}$ on $\bbR^{d}$ are characterized using the Littlewood-Paley projections as follows:
\begin{equation*}
	\nrm{\varphi}_{\dot{W}^{s, p}(\bbR^{d})} \aeq \nrm{(\sum_{k} 2^{2sk} \abs{P_{k} \varphi}^{2})^{\frac{1}{2}}}_{L^{p}(\bbR^{d})}, \quad
	\nrm{\varphi}_{\dot{B}^{s, p}_{r}(\bbR^{d})} \aeq \bb( \sum_{k} 2^{rsk} \nrm{P_{k} \varphi}_{L^{p}(\bbR^{d})}^{r} \bb)^{\frac{1}{r}}. \quad
\end{equation*}
We define the corresponding spaces $\dot{W}^{s, p}(\bbR^{d})$,
$\dot{B}^{s, p}_{r} (\bbR^{d})$ to consist of tempered distributions
that are regular at zero frequency (i.e., $\nrm{P_{\leq k}
  \varphi}_{L^{\infty}(\bbR^{d})} \to 0 \hbox{ as } k \to -\infty$)
and have finite corresponding semi-norms. We use the standard notation
$\dot{W}^{s, 2} = \dot{H}^{s}$.  When $s < \frac{d}{p}$ or $s =
\frac{d}{p}$ with $r = 1$ (in the Besov case) the above semi-norms are
in fact norms when restricted to the space $\calS(\bbR^{d})$ of
Schwartz functions on $\bbR^{d}$, and the corresponding spaces are
obtained as the completion of $\calS(\bbR^{d})$ with respect to these
norms.

\section{Finite energy initial data for
  Maxwell-Klein-Gordon} \label{sec:id} In this section we
systematically develop the basic theory of finite energy initial data
sets for \eqref{eq:MKG}. In Section~\ref{subsec:MKGid:1} we define
the spaces of finite energy and classical initial data sets for
\eqref{eq:MKG}, and also the corresponding spaces of gauge
transformations. In Section~\ref{subsec:MKGid:2}, we prove a few elementary
facts about finite energy initial data sets, such as approximation by
classical initial data and gauge transformation to a globally Coulomb
initial data. We also show that any globally Coulomb finite energy
initial data set can be approximated by classical initial data sets in
the global Coulomb gauge.

\subsection{Finite energy initial data sets and gauge
  transformations} \label{subsec:MKGid:1} Let $O \subseteq \bbR^{4}$
be a non-empty open set. Given 1-forms $a, e$ and $\bbC$-valued
functions $f, g$ on $O$, we say that the quadruple $(a, e, f, g)$ is a
\emph{\eqref{eq:MKG} initial data set} if the following \emph{Gauss}
(or the \emph{constraint}) \emph{equation} holds:
\begin{equation} \label{eq:id:gauss} \rd^{\ell} e_{\ell} = \Im [f
  \overline{g}].
\end{equation}
We define the space $\calH^{1}(O)$, which consists of \eqref{eq:MKG}
initial data sets $(a,e,f,g)$ for which the following norm is finite:
\begin{equation*}
  \nrm{(a, e, f, g)}_{\calH^{1}(O)} 
  := \sup_{j=1, \ldots, 4} \nrm{(a_{j}, e_{j})}_{(\dot{H}^{1}_{x} \cap L^{4}_{x}) \times L^{2}_{x}(O)} 
+ \nrm{(f, g)}_{(\dot{H}^{1}_{x} \cap L^{4}_{x}) \times L^{2}_{x}(O)}.
\end{equation*}
A Coulomb (gauge) initial data set is a data  set  $(a,e,f,g)$  which in addition satisfies
the divergence condition 
\begin{equation*}
\nabla \cdot a = \rd^{\ell} a_{\ell} = 0.
\end{equation*}

Given an $\calH^{1}(O)$ initial data set $(a,e,f,g)$, we define its
\emph{energy} on $O' \subseteq O$ by
\begin{equation} \label{eq:id:energy} \calE_{O'}[a,e,f,g] :=
  \frac{1}{2} \int_{O'} \sum_{1 \leq j < k \leq 1}\abs{(\ud
    a)_{jk}}^{2} + \sum_{1 \leq j \leq 4}\abs{e_{j}}^{2} + \sum_{1
    \leq j \leq 4} \abs{\covD_{j} f}^{2} + \abs{g}^{2} \, \ud x,
\end{equation}
where $(\ud a)_{jk} = \rd_{j} a_{k} - \rd_{k} a_{j}$ and $\covD_{j} f
= \rd_{j} f + i a_{j} f$. The space $\calH^{1}(O)$ is a natural domain
on which the energy functional is always finite, and for this reason
$\calH^{1}(O)$ will also be referred to as the space of \emph{finite
  energy} initial data.  In general the energy does not control the
$\calH^{1}$ norm, and for this reason we view the $\calH^{1}$ bounds
as qualitative, whereas the energy related bounds are
quantitative. However, in the case of global Coulomb data sets the
situation improves and we can estimate the $\calH^{1}$ norm in terms
of the energy, see Lemma~\ref{lem:gt2CoulombId}.

We also remark that the energy $\calE_{O'}[a,e,f,g]$ is invariant
under gauge transformations, which will be rigorously defined below.

For $N \geq 1$, we define the higher regularity space $\calH^{N}(O)$
in a similar fashion with the norm
\begin{equation*}
  \nrm{(a, e, f, g)}_{\calH^{N}(O)} 
  := 	\sum_{n=1}^{N} \nrm{(\rd_{x}^{(n-1)} a, \rd_{x}^{(n-1)} e, \rd_{x}^{(n-1)} f, \rd_{x}^{(n-1)} g)}_{\calH^{1}(O)} \, .
\end{equation*}
To define the space $\calH^{\infty}(O)$ of \emph{classical initial
  data sets}, we first define the space $\calH^{0}(O)$ to consist of
\eqref{eq:MKG} initial data sets with finite $\calH^{0}(O)$ semi-norm,
which is given by
\begin{equation*}
  \nrm{(a, e, f, g)}_{\calH^{0}(O)} 
  := 	\nrm{a}_{L^{2}_{x}} + \nrm{f}_{L^{2}_{x}}
\end{equation*}
Then we take $\calH^{\infty}(O) := \cap_{N = 0}^{\infty} \calH^{N}(O)$
and topologize it using $\set{\nrm{\cdot}_{\calH^{N}(O)}}_{N \geq
  0}$. We remark that $\calH^{\infty}(O)$ initial data have not only
better regularity (it is in fact smooth), but also better
integrability than $\calH^{N}(O)$.

Next, we define spaces of gauge transformations between initial data
sets. A gauge transformation $\chi$, which is simply a $\bbR$-valued
function on $O$, acts on an initial data set $(a,e,f,g)$ as follows:
\begin{equation*}
  \Gmm_{\chi} [a,e,f,g] := [a - \ud \chi, e, e^{i \chi} f, e^{i \chi} g].
\end{equation*}
We define the space $\calG^{2}(O)$ to consist of locally integrable
gauge transformations such that the following semi-norm is finite:
\begin{equation*}
  \nrm{\chi}_{\calG^{2}(O)} := \nrm{\rd_{x} \chi}_{L^{4}_{x}(O)} + \nrm{\rd_{x}^{2} \chi}_{L^{2}_{x}(O)}.
\end{equation*}
Given an integer $N \geq 1$, we define the $\calG^{N+1}(O)$ semi-norm
as
\begin{equation*}
  \nrm{\chi}_{\calG^{N+1}(O)} := \sum_{n=1}^{N} \bb( \nrm{\rd_{x}^{(n)} \chi}_{L^{4}_{x}(O)} + \nrm{\rd_{x}^{(n+1)} \chi}_{L^{2}_{x}(O)} \bb),
\end{equation*}
and the space $\calG^{N+1}(O)$ to consist of locally integrable gauge
transformations with finite $\calG^{N+1}(O)$ semi-norm.  Observe that
$\nrm{\chi}_{\calG^{N+1}(O)} = 0$ if and only if $\chi$ is a
constant. Accordingly, $\calG^{N+1}(O)$ becomes a Banach space once we
mod out by constants, but we shall \emph{not} do so in this paper.
Finally, we also define
\begin{align*}
  \calG^{\infty}(O) := & \bigcap_{n=1}^{N} \dot{H}^{n}_{x} \cap
  \dot{W}^{n-1, 4}_{x}(O).
\end{align*}

The space of gauge transformations between initial data sets in the
class $\calH^{N}(O)$ is precisely $\calG^{N+1}(O)$. Indeed, given
$\chi \in \calG^{N+1}(O)$, it follows from the chain rule and the fact
that $\sgm \mapsto e^{i \sgm}$ is a bounded smooth function that $e^{i
  \chi} \in \calG^{N+1}(O)$ and
\begin{equation*}
  \nrm{e^{i \chi}}_{\calG^{N+1}(O)} \aleq \nrm{\chi}_{\calG^{N+1}(O)} (1 + \nrm{\chi}_{\calG^{N+1}(O)}^{N}).
\end{equation*}
From this fact, we see that if $(a,e,f,g) \in \calH^{N}(O)$ and $\chi
\in \calG^{N+1}(O)$, then $\Gmm_{\chi}(a,e,f,g) \in
\calH^{N}(O)$. Conversely, if $\chi$ is a locally integrable gauge
transformation on $O$ such that we have $(a', e', f', g') =
\Gmm_{\chi} (a, e, f, g)$ for some $(a, e, f, g), (a', e', f' ,g') \in
\calH^{N}(O)$, then it easily follows that $\chi \in \calG^{N+1}(O)$
from the relation $\ud \chi = a - a'$.

The map $\Gmm_{\chi}[a,e,f,g]$ furthermore enjoys a nice continuity
property. We state a version of this property for the case $N = 1$,
i.e., $(a,e,f,g) \in \calH^{1}(O)$ and $\chi \in \calG^{2}(O)$.

\begin{lemma} \label{lem:continuity4gt} Let $O$ be an open connected subset of
  $\bbR^{4}$.   Let $(a^{(n)}, e^{(n)}, f^{(n)}, g^{(n)})$
  [resp. $\chi^{(n)}]$ be a sequence of $\calH^{1}(O)$ initial data
  sets [resp. $\calG^{2}(O)$ gauge transformations] such that
  \begin{equation*}
    \nrm{(a - a^{(n)}, e - e^{(n)}, f - f^{(n)}, g - g^{(n)})}_{\calH^{1}(O)} \to 0, \quad
    \nrm{\chi - \chi^{(n)}}_{\calG^{2}(O)} \to 0,
  \end{equation*}
  for some $(a,e,f,g) \in \calH^{1}(O)$ and $\chi \in \calG^{2}(O)$ as
  $n \to \infty$. Then there exists a sequence $\chi^{(n)}_{0} \in
  \bbR$ of constant gauge transformations such that
  \begin{equation} \label{eq:continuity4gt} \nrm{\Gmm_{\chi} [a,e,f,g]
      - \Gmm_{\chi^{(n)} +
        \chi^{(n)}_{0}}[a^{(n)},e^{(n)},f^{(n)},g^{(n)}]}_{\calH^{1}(O)}
    \to 0 \hbox{ as } n \to \infty.
  \end{equation}
\end{lemma}
\begin{proof}
 We shall write
  \begin{equation*}
    (\widetilde{a}, \widetilde{e}, \widetilde{f}, \widetilde{g}) = \Gmm_{\chi}[a,e,f,g], \quad
    (\widetilde{a}^{(n)}, \widetilde{e}^{(n)}, \widetilde{f}^{(n)}, \widetilde{g}^{(n)}) = \Gmm_{\chi^{(n)}}[a^{(n)}, e^{(n)}, f^{(n)}, g^{(n)}].
  \end{equation*} 
  Before we begin the proof, we first make a few reductions.  We first
  remark that the constants $\chi_0^{(n)}$ above are needed because
  they are not seen by the $\calG^2(O)$ norm. We can eliminate them if
  we normalize $\chi^{(n)}$, e.g. by requiring that they have zero
  averages on some ball $B \subset O$:
 \begin{equation} \label{eq:continuity4gt:zeroint4chi} \int_{B}
    \chi^{(n)} \, \ud x= 0
  \end{equation}
  We will make this assumption from here on.

  Observe further that \eqref{eq:continuity4gt} is easy when all
  $\chi^{(n)}$'s are the same. Then applying $-\chi$ to every term in
  the sequence, it suffices to consider the case $\chi = 0$. Finally,
  the convergence of $(\widetilde{a}^{(n)}, \widetilde{e}^{(n)})$ in
  $\dot{H}^{1}_{x} \cap L^{4}_{x}(O) \times L^{2}_{x}(O)$ is obvious,
  so we will focus on $(\widetilde{f}^{(n)}, \widetilde{g}^{(n)})$.

  We claim that
  \begin{equation*}
    \nrm{\widetilde{\covD}_{j} \widetilde{f} - \widetilde{\covD}_{j}^{(n)} 
\widetilde{f}^{(n)} }_{L^{2}_{x}} 
    + \nrm{\widetilde{f} - \widetilde{f}^{(n)} }_{L^{4}_{x}(O)}
    + \nrm{\widetilde{g} - \widetilde{g}^{(n)}}_{L^{2}_{x}(O)} \to 0 \quad \hbox{ as } n \to \infty,
  \end{equation*}
  where $\widetilde{\covD}_{j} = \rd_{j} + i \widetilde{a}_{j}$,
  $\widetilde{\covD}_{j}^{(n)} = \rd_{j} + i
  \widetilde{a}_{j}^{(n)}$. Then the desired conclusion
  \eqref{eq:continuity4gt} would follow, using the claim and the
  $L^{4}_{x}(O)$ convergence of $\widetilde{a}^{(n)} \to
  \widetilde{a}$ to deduce that $\rd_{x} \widetilde{f}^{(n)} \to
  \rd_{x} \widetilde{f}$ in $L^{2}_{x}(O)$.

  We now prove $\widetilde{g}^{(n)} \to \widetilde{g}$ in
  $L^{2}_{x}(O)$; a similar argument works for $\widetilde{f}^{(n)}$
  and $\widetilde{\covD}_{j}^{(n)} \widetilde{f}^{(n)}$ as well. We
  write
  \begin{equation*}
    \nrm{\widetilde{g} - \widetilde{g}^{(n)}}_{L^{2}_{x}(O)}
    = \nrm{g - e^{i \chi^{(n)}} g^{(n)}}_{L^{2}_{x}(O)}
    \leq \nrm{(1-e^{i \chi^{(n)}}) g}_{L^{2}_{x}(O)} + \nrm{e^{i \chi^{(n)}} (g - g^{(n)})}_{L^{2}_{x}(O)}.
  \end{equation*}
  Since $\nrm{e^{i \chi^{(n)}}}_{L^{\infty}_{x}} \leq 1$, it follows
  that the last term vanishes as $n \to \infty$. It remains to prove
  \begin{equation} \label{eq:continuity4gt:key} \nrm{(1-e^{i
        \chi^{(n)}}) g}_{L^{2}_{x}(O)} \to 0.
  \end{equation}
By Lebesgue's dominated convergence theorem, it suffices to show
that each subsequence $n_k$ has a further subsequence $n_{k_j}$ 
so that $\chi^{(n_{k_j})} \to 0$ almost everywhere in $O$. To see this we 
use Poincare's inequality. In view of the normalization \eqref{eq:continuity4gt:zeroint4chi},
this shows that from the convergence $\| \chi^{(n)}\|_{\calG^2(O)} \to 0$ we obtain 
\[
 \chi^{(n)}  \to 0 \qquad \text{ in $L^4_{loc}(O)$.}
\]
Then the a.e. convergence on a subsequence immediately follows.
 \qedhere
\end{proof}

%
%
%
%
%

\subsection{Approximation and gauge transformation
  lemmas} \label{subsec:MKGid:2} In this subsection, we record a few
useful facts concerning $\calH^{1}$ initial data sets on
$\bbR^{4}$. The first result says that any $\calH^{1}(\bbR^{4})$
initial data set can be approximated by classical initial data sets.
\begin{lemma} \label{lem:idApprox} Let $(a, e, f, g)$ be an initial
  data set for \eqref{eq:MKG} in the class $\calH^{1}(\bbR^{4})$. Then
  there exists a sequence $(a^{(n)}, e^{(n)}, f^{(n)}, g^{(n)})$ of
  initial data sets in $\calH^{\infty}(\bbR^{4})$ which approximates
  $(a,e,f,g)$ in $\calH^{1}(\bbR^{4})$.
\end{lemma}
\begin{proof}
  Take any $C^{\infty}_{0}(\bbR^{4})$ sequence $(\widetilde{a}^{(n)},
  \widetilde{e}^{(n)}, \widetilde{f}^{(n)}, \widetilde{g}^{(n)})$
  which converges to $(a,e,f,g)$ in the $\calH^{1}(\bbR^{4})$ norm, and take
  \begin{equation*}
    a^{(n)} = \widetilde{a}^{(n)}, \quad
    f^{(n)} = \widetilde{f}^{(n)}, \quad
    g^{(n)} = \widetilde{g}^{(n)}.
  \end{equation*}
  To satisfy the Gauss equation, we take
  \begin{equation*}
    e^{(n)}_{j} = \widetilde{e}^{(n)}_{j} + (-\lap)^{-1} \rd_{j} (\rd^{\ell} \widetilde{e}^{(n)}_{\ell} - \Im[f^{(n)} \overline{g^{(n)}}]).
  \end{equation*}
  It can be readily verified that $e^{(n)} \in
  H^{\infty}_{x}(\bbR^{4})$. Moreover, since $\rd^{\ell}
  \widetilde{e}^{(n)}_{\ell} - \Im[f^{(n)} \overline{g^{(n)}}] \to 0$
  in $\dot{H}^{-1}_{x}(\bbR^{4})$, it follows that $e^{(n)}_{j} \to
  e_{j}$ in $L^{2}_{x}(\bbR^{4})$, as desired. \qedhere
\end{proof}

The second result shows that any $\calH^{1}(\bbR^{4})$ initial data
set can be gauge transformed to a globally Coulomb initial data set.
\begin{lemma} \label{lem:gt2CoulombId} Let $(\widetilde{a},
  \widetilde{e}, \widetilde{f}, \widetilde{g})$ be an initial data set
  for \eqref{eq:MKG} in the class $\calH^{1}(\bbR^{4})$. Then there
  exists a gauge transform $\chi \in \calG^{2}(\bbR^{4})$,  unique up constants, such that
  \begin{equation*}
    (a, e, f, g) = (\widetilde{a} - \ud \chi, \widetilde{e} ,e^{i \chi} \widetilde{f}, e^{i \chi} \widetilde{g})
  \end{equation*}
  satisfies the global Coulomb gauge condition $\rd^{\ell} a_{\ell} =
  0$ [resp. $\rd^{\ell} a'_{\ell} = 0$] on $\bbR^{4}$. Moreover, we have 
  the estimate
  \begin{equation}
    \nrm{\chi}_{\calG^{2}(\bbR^{4})} \aleq  \nrm{\widetilde{a}}_{\dot{H}^{1}_{x}(\bbR^{4})} .  \label{eq:gt2CoulombId} 
\end{equation}
\end{lemma}


\begin{proof} 
  Let
  \begin{equation} \label{eq:gt2CoulombId:omg} \omg_{j} = (-\lap)^{-1}
    \rd_{j} \rd^{\ell} \widetilde{a}_{\ell}.
  \end{equation}
  Since $\widetilde{a} \in \dot{H}^{1}_{x}(\bbR^{4})$, it follows that
  $\omg_{j} \in \dot{H}^{1}_{x}(\bbR^{4})$. Note moreover that
  \begin{equation*}
    \rd_{i} \omg_{j} - \rd_{j} \omg_{i} = 0
  \end{equation*}
  for every $i, j = 1,2,3,4$. Thus there exists\footnote{This is
    obvious when $\widetilde{a} \in \calS(\bbR^{4})$; the full
    statement follows by approximation of $a$ by Schwartz 1-forms,
    using the fact that $\BMO$ is a Banach space modulo constant
    functions.} a real-valued function $\chi$ such that
  \begin{equation*}
    \ud \chi = \omg,
  \end{equation*}
  which furthermore satisfies $\chi \in \calG^{2}(\bbR^{4})$ and
  \eqref{eq:gt2CoulombId}. Note that $(a, e, f, g)$ defined as above
  satisfies the global Coulomb condition, since $\rd^{\ell} a_{\ell} =
  \rd^{\ell} \widetilde{a}_{\ell} + \lap \chi = 0$. The uniqueness
  statement follows from the fact that the solution to $\lap \rd_{j}
  \chi = \rd_{j} \rd^{\ell} \widetilde{a}_{\ell}$ in $L^{4}_{x}(O)$ is
  uniquely given by \eqref{eq:gt2CoulombId:omg}. \qedhere
\end{proof}

An immediate consequence of Lemma \ref{lem:continuity4gt} and the
preceding two lemmas is that any Coulomb initial data set in
$\calH^{1}(\bbR^{4})$ can be approximated in $\calH^{1}(\bbR^{4})$ by
classical Coulomb initial data sets. We record this statement as a
corollary.

\begin{corollary} \label{cor:idApproxByCoulomb} Let $(a, e, f, g)$ be
  a globally Coulomb initial data set for \eqref{eq:MKG} in the class
  $\calH^{1}(\bbR^{4})$. Then there exists a sequence
  $(\check{a}^{(n)}, \check{e}^{(n)}, \check{f}^{(n)},
  \check{g}^{(n)})$ of globally Coulomb initial data sets in
  $\calH^{\infty}(\bbR^{4})$ which approximates $(a,e,f,g)$ in
  $\calH^{1}(\bbR^{4})$.
\end{corollary}

\section{Excision and gluing of initial data sets} \label{sec:gluing}
A recurrent nuisance in gauge theory is the presence of a non-trivial
constraint equation for the initial data sets.  More concretely,
consider the problem of localizing a \eqref{eq:MKG} initial data
set. The most naive way to proceed would be to apply a smooth cutoff;
however, integrating the constraint equation (also called the
\emph{Gauss equation})
\begin{equation*}
  \rd^{\ell} e_{\ell} = \Im[f \overline{g}]
\end{equation*}
by parts over balls of large radius, we see that $e_{\ell}$ must in
general be non-trivial on the boundary spheres even if $f, g$ are
compactly supported. This simple argument precludes the naive approach
of simply cutting off $(a, e, f, g)$.

The purpose of this section is to introduce a set of techniques for
addressing this difficulty, namely \emph{excision and gluing} of
\eqref{eq:MKG} initial data sets. In the context of localization of
initial data sets, the basic idea is as follows: Instead of simply
\emph{excising} the unwanted part of the initial data set, we
\emph{glue} it to another initial data set, which has an explicit
description in the excised region. For instance, in the exterior of a
ball (see Proposition~\ref{prop:gluing} below) we glue with a data set
of the form $(e_{(q) j} := \frac{q}{2 \pi^{2}}
\frac{x_{j}}{\abs{x}^{4}}, 0, 0, 0)$, which is precisely the
electro-magnetic field of an electric monopole of charge $q$ situated
at the origin.


Key to our approach is a simple solution operator $\calV$ for the
divergence equation that preserves the support and obeys a sharp
$L^{p}$-$L^{q}$ bound. This solution operator was first used by
Bogovski{\u\i} \cite{MR553920, MR631691}. We remark that a similar
solution operator was used in \cite{Isett:2014vk} in the context of
the incompressible Euler equations.

The main results are stated in the next two propositions. The first
one concerns excision and gluing of initial data sets to the exterior
of a ball.
\begin{proposition}[Excision and gluing of initial data sets to the
  exterior] \label{prop:gluing} Let $B$ be 
  a ball of radius $r_0$ in $\bbR^{4}$, and  $1 < \rgExt < \rglue \leq  2$. 
  Then there exists an operator $E^{\extr}$ from the class
  $\calH^{1}(\rglue B \setminus \overline{B})$ to the class
  $\calH^{1}(\bbR^{4} \setminus \overline{B})$ satisfying the
  following properties:
  \begin{enumerate}
 \item \label{item:gluing:2}  $(\widetilde{a}, \widetilde{e},
    \widetilde{f}, \widetilde{g}) :=  E^{\extr}[a, e, f, g]$ is an extension of $ (a, e, f, g)$, 
    \begin{equation*}
  (\widetilde{a}, \widetilde{e},
    \widetilde{f}, \widetilde{g})   = (a, e, f, g) \quad \hbox{ on the annulus }
 \rgExt B \setminus \overline{B}.
    \end{equation*}
    
  \item \label{item:gluing:1} We have $(\widetilde{a}, \widetilde{f},
    \widetilde{g}) = (0, 0, 0)$ on $\bbR^{4} \setminus \rglue
    \overline{B}$. On the other hand, there exists a real number $q =
    q(e)$, depending continuously on $e \in L^{2}(\rglue B \setminus
    \overline{B})$, such that
    \begin{equation*}
      \widetilde{e}_{j}(x) = q \frac{x^{j}}{r^{4}} \quad \hbox{ on $\bbR^{4} \setminus \rglue \overline{B}$}.
    \end{equation*}

  \item \label{item:gluing:3} The following bounds hold, with implicit constants 
depending on $\rgExt,\rglue$:
    \begin{align}
      \nrm{E^{\extr}[a,e,f,g]}_{\calH^{1}(\bbR^{4} \setminus
        \overline{B})}
      \lesssim & \  \nrm{(a,e,f,g)}_{\calH^{1}(\rglue B \setminus \overline{B})}, \label{eq:gluing:est} \\
      \calE_{\bbR^{4} \setminus \overline{B}}[E^{\extr}[a, e, f, g]]
      \lesssim & \ r_0^{-2}
      \nrm{ f}_{L^{2}_{x}(\rglue B \setminus
        \overline{B})}^{2} +  \calE_{\rglue B
        \setminus \overline{B}}[a, e, f, g].  \label{eq:gluing:energy}
    \end{align}

  \item \label{item:gluing:4} The operator $E^{\extr}$ is continuous
    from $\calH^{1}(\rglue B \setminus \overline{B})$ to
    $\calH^{1}(\bbR^{4} \setminus \overline{B})$. Moreover,
    $E^{\extr}$ enjoys \emph{persistence of higher regularity}, i.e.,
    for every $N \geq 1$, we have $E^{\extr}[\calH^{N}(\rglue B
    \setminus \overline{B})] \subseteq \calH^{N}(\bbR^{4} \setminus
    \overline{B})$.
  \end{enumerate}
\end{proposition}

The second proposition concerns excision and gluing of initial data in
the interior of a ball.
\begin{proposition}[Excision and gluing of initial data sets to the
  interior] \label{prop:intGluing} Let $B$ be 
  a ball of radius $r_0$ in $\bbR^{4}$, and  $1 < \rgInt < \rglue \leq  2$. 
 Then there exists an operator $E^{\intr}$ from the class
  $\calH^{1}(\rglue B \setminus \overline{B})$ to the class
  $\calH^{1}(\rglue B)$ satisfying the following properties:
  \begin{enumerate}
 \item \label{item:intGluing:2} $E^{\intr}[a, e, f, g]$ is an interior extension of $ (a, e, f, g)$, 
    \begin{equation*}
      E^{\intr}[a, e, f, g] = (a, e, f, g) \quad \hbox{ on the annulus } 
\rglue B \setminus \rgInt \overline{B}.
    \end{equation*}

  \item \label{item:intGluing:1} The following bounds hold, with implicit constants 
depending on $\rgInt,\rglue$:
    \begin{align}
      \nrm{E^{\intr}[a,e,f,g]}_{\calH^{1}(\rglue B)} \lesssim & \
      \nrm{(a,e,f,g)}_{\calH^{1}(\rglue B \setminus
        \overline{B})} + \nrm{e}_{L^{2}_{x}(\rglue B \setminus
        \overline{B})}^{\frac{1}{2}},
      \label{eq:intGluing:est} \\
      \calE_{\rglue B}[E^{\intr}[a, e, f, g]]
      \lesssim &  \ \calE_{\rglue B \setminus \overline{B}}[a, e, f, g] 
 + r_0^{-2}\nrm{ f      }_{L^{2}_{x}(\rglue B \setminus \overline{B})}^{2}. \label{eq:intGluing:energy} 
    \end{align}

  \item \label{item:intGluing:3} The operator $E^{\intr}$ is
    continuous from $\calH^{1}(\rglue B \setminus \overline{B})$ to
    $\calH^{1}(\rglue B)$. Moreover, $E^{\intr}$ enjoys
    \emph{persistence of higher regularity}, i.e., for every $N \geq
    1$, we have $E^{\intr}[\calH^{N}(\rglue B \setminus \overline{B})]
    \subseteq \calH^{N}(\rglue B)$.
  \end{enumerate}
\end{proposition}

\begin{remark} 
  There are two main difficulties in the proving these
  propositions. The first one is the presence of the Gauss equation,
  which has been discussed at the beginning of this section.  The
  second difficulty stems from the \emph{local energy inequalities}
  \eqref{eq:gluing:energy} and \eqref{eq:intGluing:energy}, which
  require, in particular, choosing a `good gauge' before excising the
  initial data. To resolve this difficulty, we rely on the solvability
  in $L^{2}$-Sobolev spaces of the one-form Hodge system under
  suitable boundary conditions (see
  Section~\ref{subsec:hodge4oneform}). This statement can be thought
  of as an easier abelian variant of Uhlenbeck's lemma
  \cite{Uhlenbeck:1982vna} concerning existence of a gauge
  transformation to the Coulomb gauge.
\end{remark}


The rest of this section is structured as follows: In
Section~\ref{subsec:sol4div}, we introduce a solution operator $\calV
= \calV_{j} [h]$ to the divergence equation $\rd^{j} \calV_{j}[h] = h$
that, in particular, is compactly supported if $h$ is. In
Section~\ref{subsec:hodge4oneform}, we briefly recall a standard
result for the 1-form Hodge system on domains with smooth boundary,
which will be needed later. Then in Section~\ref{subsec:gluing:pf}, we
present proofs of Propositions \ref{prop:gluing} and
\ref{prop:intGluing}.

\subsection{Support-preserving solution operator for 
the divergence equation} \label{subsec:sol4div}

In this subsection, we define a solution operator to the divergence
equation which preserves the support property of the source.  This
solution operator was first introduced by Bogovski{\u\i}
\cite{MR553920, MR631691}. Our construction below follows the approach
of \cite{Isett:2014vk}, in which a similar solution operator was
constructed for the symmetric divergence equation $\rd_{j} R^{j \ell}
= U^{\ell}$.  We sharpen the estimates for $\calV$ compared to
\cite{Isett:2014vk} (where non-sharp estimates sufficed), which turns
out to be necessary due to the criticality of our problem.  The class
of domains we work with is that of  star-shaped domains, and
unions thereof. We call a domain \emph{strongly star-shaped} with respect to a 
set $B$ if it is star-shaped with respect to any point in  $B$.

\begin{proposition} \label{prop:sol4div} Let $B$ be a ball in
  $\bbR^d$, $d \geq 2$.  Then there exists a pseudodifferential
  operator $\calV \in OPS^{-1}_{loc}(\bbR^d)$, taking functions to $1$-forms,
which has the following properties:  
  \begin{enumerate}
  \item \label{eq:sol4div:supph} 
 For any compact domain $D$ which is star-shaped with respect
  to $B$, if $h \in \calD'$ is supported in $D$ then 
 $\calV[h]$ is also supported in $D$.
  \item Suppose that $h$ has compact support and 
    \begin{equation*}
      \int_{\bbR^{d}} h \, \ud x = 0.
    \end{equation*}
    Then $\calV[h]$ satisfies the divergence equation
    \begin{equation} \label{eq:sol4div:diveq} \rd^{\ell}
      \calV_{\ell}[h] = h.
    \end{equation}
  \end{enumerate}
\end{proposition}

 \begin{remark}   
The fact that $D$ is star-shaped with respect to $B$ requires that $B \subseteq D$.
Thus by scaling all bounds for the operator $\calV$ in $D$ depend only on the ratio
$\diam(D)/\diam(B)$. In particular, since  $\calV \in OPS^{-1}_{loc}(\bbR^d)$, we obtain
\begin{equation}\label{eq:sob}
\| \calV [h] \|_{W^{1,p}(\bbR^{d})} \lesssim \|h\|_{L^p(\bbR^{d})}, \qquad 1 < p < \infty
\end{equation}
Thus, by the Gagliardo-Nirenberg-Sobolev inequality we obtain the inequality
\begin{equation} \label{eq:sol4div:est}
    \nrm{\calV[h]}_{L^{q}_{x}(\bbR^{d})} \aleq_{p, q} (\diam
    B)^{\frac{d}{q} - \frac{d}{p} + 1} \nrm{h}_{L^{p}_{x}(\bbR^{d})}
    \, .
  \end{equation}	
whenever $1 <  p \leq q < \infty$ and 
\begin{equation*}
    \frac{d}{p} - 1 \leq  \frac{d}{q} \leq \frac{d}{p}
\end{equation*}
\end{remark}

Before we begin the proof of Proposition \ref{prop:sol4div} in
earnest, we give a short argument that provides a solution operator
$\calV$ with the required support properties but with less regularity.
We will use this to motivate the actual construction.  Let $h \in
C^{\infty}_{0}(\bbR^{d})$ satisfy \eqref{eq:sol4div:supph}. Assume
furthermore that $\int_{\bbR^{d}} h = 0$. Our goal is to find a
solution $v^{\ell}$ to \eqref{eq:sol4div:diveq} which satisfies the
support property $\supp \, v^{\ell} \subseteq B$. Note that
$\int_{\bbR^{d}} h = 0$ is a necessary condition for such a solution
to exist by the divergence theorem.

Taking the Fourier transform of $h$ and Taylor expanding at $\xi =0$,
we get
\begin{equation*}
  \widehat{h}(\xi) = \widehat{h}(0) + \xi^{\ell} \int_{0}^{1} \rd_{\ell} \widehat{h}(\sgm \xi) \, \ud \sgm.
\end{equation*}
Since $\widehat{h}(0) = \int h \, \ud x= 0$, we see that $\widehat{h}$
has the form of a divergence. Indeed, defining
\begin{equation*}
  v_{j}[h] := \calF^{-1}_{x} [\int_{0}^{1} \rd_{j} \widehat{h}(\sgm \xi) \, \sgm],
\end{equation*}
we see that $\rd^{\ell} v_{\ell}[h] = h$, as desired.  More generally, we remark that if 
$\int_{\bbR^d}  h \neq 0$, then 
\[
\nabla \cdot v = h - c \delta_0, \qquad  c = \int_{\bbR^d} h \, \ud x. 
\]

Carrying out the inverse Fourier transform, we obtain the following
physical space formula for $v_{j}[h]$:
\begin{equation*}
  v_{j}[h](x) =  \int_{0}^{1}  \frac{x^{j}}{\sgm} h \bb( \frac{x}{\sgm} \bb) \, \frac{\ud \sgm}{\sgm^{d}}.
\end{equation*}
Note that the value of $v_{j}[h]$ at $x$ is determined by a weighted
integral of $h$ on the radial ray $\set{s x : s \geq 1}$. In
particular, the desired support property $\supp \, v \subseteq B$
immediately follows. In terms of regularity, however, integration
along rays only yields radial regularity. No angular regularity at all is 
gained by doing this.

One can also view the above construction as arising from a mass transportation 
problem. The above $v$ corresponds to transporting all the mass of $h$ along rays
to zero.   

In order to produce a better solution operator, all we need to do is
to expand the above Dirac mass at zero into a smooth bump function,
i.e. some smooth averaging of the above construction. This idea is
carried out in the following proof.

\begin{proof} [Proof of Proposition \ref{prop:sol4div}]
By translation and scaling we assume that $B$ is the unit ball.  
Given $y \in \bbR^{d}$, define
  \begin{align*}
    v_{(y)j}[h](x) = \int_{0}^{1} \frac{(x-y)^{j}}{\sgm} h \bb(
    \frac{x-y}{\sgm} + y \bb) \, \frac{\ud \sgm}{\sgm^{d}}.
  \end{align*}
  Let $r_{0}$ be the radius of the ball $B$. Let $\zt$ be a smooth
  normalized bump function in $B$, i.e. 
  \begin{equation} \label{eq:sol4div:zt} \supp \, \zt \subseteq
     B, \quad \int_{\bbR^{d}} \zt = 1.
  \end{equation}
  We now define $\calV_{j}[h] := \int \zt(y) v_{(y)j} [h] \, \ud y$,
  i.e.,
  \begin{equation} \label{eq:sol4div:V} \calV_{j}[h] (x) = \int
    \int_{0}^{1} \zt(y) \frac{(x-y)^{j}}{\sgm} h \bb( \frac{x-y}{\sgm}
    + y \bb) \, \frac{\ud \sgm}{\sgm^{d}} \, \ud y.
  \end{equation}

  As before, $\calV[h]$ is a solution to the divergence
  equation 
\[
\partial^l \calV_l[h] = h - c \zt, \qquad c = \int_{\bbR^d} h \, \ud x 
\]
where the second term in the right-hand side vanishes
provided that $h$ has integral zero. Moreover, from 
 the construction it follows that we have the support  
property
\[
\supp \calV[h] \subseteq  \bigcup_{x \in \supp \, h} \mathrm{Conv}( \{x\} \cup B),
\]
where $\mathrm{Conv}(X)$ refers to the convex hull of $X$. This is exactly what we need.  

It remains to prove that $\calV$ is a regular pseudodifferential operator of order $-1$.  For
that we look at the kernel $K(x-y,y)$ of $\calV$, which after a change
of variable is written as
\[
K(z,y) = \int \frac{z}{1-\sigma} \zt( \frac{1}{1-\sigma} z + y ) \frac{d\sigma}{(1-\sigma)^d}
\]
The  support condition on $\zt$ restricts the integral to the range $|1-\sigma| \gtrsim |z|$.
Then a direct integration yields the bound
\[
|K(z, y)| \lesssim |z|^{1-d}
\]
Similarly, we have the differentiated bounds 
\[
| \partial_z^{(k)} \partial_y^{(j)} K(z, y)| \leq c_{kj}  |z|^{1-d-k}.
\]
The symbol $a(\xi, y)$ of $\calV$ in the right calculus\footnote{We prefer the right calculus,
  because there the symbol  is only needed for $y \in D$.}  is obtained by taking the
Fourier transform of $K$ with respect to $z$. Then the preceding bound implies the homogeneous symbol bound
\begin{equation*}
	\abs{\rd_{z}^{(k)} \rd_{y}^{(j)} a(\xi, y)} \aleq_{k, j} \abs{\xi}^{-1-k}.	
\end{equation*}
On the other hand, taking into account the support properties of $K$, it follows that $\abs{\rd_{\xi}^{(k)} \rd_{y}^{(j)} a}$ is bounded for every $k, j$ as well.
Hence the assertion $\calV \in OPS^{-1}$ follows.
\end{proof}

In the sequel we apply the above proposition in two situations.
The first is for a ball:
\begin{corollary}\label{cor:sol4div:ball}
 Let $B$ be a ball in $\bbR^d$. Then there exists a pseudodifferential  operator
  $\calV^B \in OPS^{-1}$, mapping distributions $h$ supported in $B$ to
  distributions $\calV^B [h]$ supported in $B$, and which satisfies property (2)
in Proposition~\ref{prop:sol4div}.
\end{corollary}
For this we only need to observe that $B$ is star-shaped with
respect to $B$.

Our second application is for an annulus:
 
\begin{corollary}\label{cor:sol4div:annulus}
 Let $A = \sigma B \setminus B$, with $\sigma > 1$, be an annulus.
Then there exists a pseudodifferential  operator
  $\calV^A \in OPS^{-1}$, mapping distributions $h$ supported in $A$ to
  distributions $\calV^{A} [h]$ supported in $A$, and which satisfies property (2)
in Proposition~\ref{prop:sol4div}. Further, all bounds are uniform for $ \sigma $ 
away from $1$. 
\end{corollary}

In particular we note the following bound
 \begin{equation} \label{eq:sol4div:annulus:est}
    \nrm{\calV^{A}[h]}_{L^{2}_{x}(\bbR^{4})} \aleq
    \nrm{h}_{L^{\frac{4}{3}}_{x}(\bbR^{4})}
  \end{equation}
with an implicit constant that is uniform for $\sigma$ away from $1$. 

To show that this follows from Proposition~\ref{prop:sol4div}, we
cover $A$ with three or more overlapping round sectors of identical
angle $\theta$, $A = \bigcup_{k=1}^K A_k$, so that the double-angle sectors $2A_k \subseteq A$
are star-shaped. The number of such sectors depends only on the
dimension $d$ if $\sigma $ is large, but increases as $ \sigma \to
1$.  The closer $\sigma$ gets to $1$, the worse our bounds will get.

In each such sector we can apply  Proposition~\ref{prop:sol4div}. However, to
conclude the proof of the corollary we need to also be able to
distribute the zero integral condition to the sectors. This is
achieved in the next lemma:  

\begin{lemma} \label{lem:divEqOnannulus} Consider a covering of the
  annulus $A = \sigma B \setminus B$ with round sectors $A = \bigcup
  A_k$ of angle $\theta$. Let $\eta_k$ be an associated partition of
  unity in $A$ with $\supp \, \eta_{k} \subseteq 2A_{k}$. Then for each distribution $h$ which satisfies
\begin{equation*}
    \supp \, h \subseteq A
    \quad \hbox{and} \quad
    \int h \, \ud x= 0.
  \end{equation*}
 there exists a linear decomposition $h = \sum_{k=1}^K
  h_{k}$ so that  
 \begin{equation*}
    \supp \, h_{k} \subseteq 2A_{k}, \quad
    \int h_{k} \, \ud x= 0.
  \end{equation*}
  and the maps $h \to h_k - \eta_k h$ are finite rank $\leq 2$ from
  $\calD'$ to $\calD$.
\end{lemma}

The previous corollary is then proved by applying  Proposition~\ref{prop:sol4div}
to each $h_k$ in the sectors $2A_k$.

\begin{proof} 
We label the sectors  $A_k$ so that $A_{k} \cap A_{k+1} \neq \emptyset$.
For each $k$, let   $\zt_{k}$ be a smooth function with unit
  mass supported in $2A_{k} \cap 2A_{k+1}  $. For convenience, we define
  $\zt_{0} = 0$.  The idea is to
  write 
  \begin{align*}
    h_{k} := & \eta_{k} h - \zt_{k} \int  \sum_{j \leq k} \eta_{j}  h + \zt_{k-1} \int \sum_{j < k} \eta_{j} h \quad \hbox{ for } 1 \leq k \leq K-1, \\
    h_{K} := & \eta_{K} h + \zt_{K-1} \int \eta_{K-1} h.
  \end{align*}
  By construction, we have $\int h_{k} =0$ for $1 \leq k \leq K-1$;
  then it follows that $\int h_{K} = 0$ since $\int h =0$.  \qedhere
\end{proof}


\subsection{$L^{2}$ Hodge theory for
  1-forms} \label{subsec:hodge4oneform} Another ingredient in our
proofs of Propositions \ref{prop:gluing} and \ref{prop:intGluing} is
the solvability of a boundary value problem for the 1-form Hodge
system. The result that we need is as follows:

\begin{proposition} \label{prop:hodge4oneforms} Let $O$ be a
  pre-compact connected open subset of $\bbR^{4}$ with a smooth
  boundary $\rd O$. Assume furthermore that the first de Rham
  cohomology group of $O$ vanishes, i.e., $\firstDR(O) = 0$. Then for
  any 2-form $F$ on $O$ such that $F \in H^{N}(O)$ ($N \geq 0$), there
  exists a unique 1-form $\omg \in H^{N+1}(O)$ which solves the
  following boundary value problem for the 1-form Hodge system:
  \begin{equation} \label{eq:hodge4oneform} \ud \omg = F, \quad
    \rd^{\ell} \omg_{\ell} = 0, \quad \omg \rst_{\rd O} (\bfn) = 0,
  \end{equation}
  where $\bfn$ is the outer-pointing normal vector field on $\rd
  O$. Moreover, $\omg$ obeys the estimate
  \begin{equation} \label{eq:hodge4oneform:est}
    \nrm{\omg}_{H^{N+1}_{x}(O)} \aleq \nrm{F}_{H^{N}_{x}(O)}.
  \end{equation}
\end{proposition}

This is a standard result; we refer the reader to \cite[Section
5.9]{MR2744150}. The cohomology condition ensures, by the Hodge
theorem, that the kernel of the Hodge system is trivial. Then the
latter fact allows us to conclude unique solvability of
\eqref{eq:hodge4oneform} by the Fredholm alternative theorem.

\subsection{Proof of Propositions \ref{prop:gluing} and
  \ref{prop:intGluing}} \label{subsec:gluing:pf} We are ready to prove
Propositions \ref{prop:gluing}--\ref{prop:intGluing}.
\begin{proof} [Proof of Proposition \ref{prop:gluing}]
  Without any loss of generality, we may assume that $B$ is centered
  at the origin of $\bbR^{4}$. For  $1 <  \rgExt < \rglue \leq 2$, we define  $\rgExt =
  \rglue^{(0)} < \rglue^{(1)} < \rglue^{(2)} < \rglue^{(3)} = \rglue$
  as
  \begin{equation*}
    \rglue^{(0)} = \rgExt, \quad
    \rglue^{(1)} = \frac{1}{2} \rgExt + \frac{1}{2} \rglue, \quad
    \rglue^{(2)} = \frac{1}{3} \rgExt + \frac{2}{3} \rglue, \quad
    \rglue^{(3)} = \rglue.
  \end{equation*}
  Below, we will write
  \begin{equation*}
    E^{\extr}[a,e,f,g] = (\widetilde{a}, \widetilde{e}, \widetilde{f}, \widetilde{g}).
  \end{equation*}

  \pfstep{Step 1. Excision of $a_{j}$} The purpose of this step is to
  cutoff $a_{j}$ to obtain $\widetilde{a}_{j}$ on $\bbR^{4} \setminus
  \overline{B}$ such that
  \begin{align}
    \widetilde{a} = a \hbox{ on } \rglue^{(1)} B \setminus
    \overline{B}, & \quad
    \widetilde{a} = 0 \hbox{ on } \bbR^{4} \setminus \rglue^{(3)} \overline{B}, \label{eq:gluing:basicprops4atilde} \\
    \nrm{\widetilde{a}}_{\dot{H}^{1}_{x} \cap L^{4}_{x}(\bbR^{4}
      \setminus \overline{B})}
    & \aleq_{\rglue}  \nrm{a}_{\dot{H}^{1}_{x} \cap L^{4}_{x}(\rglue B \setminus \overline{B})}, \label{eq:gluing:est4atilde} \\
    \nrm{\ud \widetilde{a}}_{L^{2}_{x}(\bbR^{4} \setminus
      \overline{B})}^{2} & \aleq_{\rglue} \calE_{\rglue B \setminus
      \overline{B}}[a,e,f,g], \label{eq:gluing:energy4atilde}
  \end{align}
  where $(\ud \widetilde{a})_{jk} = \rd_{j} \widetilde{a}_{k} -
  \rd_{k} \widetilde{a}_{j}$. If one drops the last condition, then
  the simple choice $\widetilde{a} = \eta a$ for a suitable cutoff
  $\eta$ will do the job; however, having the estimate
  \eqref{eq:gluing:energy4atilde} with only the energy of $(a, e, f,
  g)$ on the annular region $\rglue B \setminus \overline{B}$ on the
  right-hand side will be crucial for our later purposes, in particular for
  performing the blow-up analysis in \cite{OT3}. Our idea for achieving this goal is
  as follows: First, we will find a gauge equivalent connection 1-form
  $\check{a}$ on the annular region $\rglue B \setminus \overline{B}$
  such that
  \begin{equation} \label{eq:gluing:est4acheck}
    \nrm{\check{a}}_{\dot{H}^{1}_{x} \cap L^{4}_{x}(\rglue B \setminus
      \overline{B})} + (\diam B)^{-1}
    \nrm{\check{a}}_{L^{2}_{x}(\rglue B \setminus \overline{B})}
    \aleq_{\rglue} \nrm{\ud a}_{L^{2}_{x}(\rglue B \setminus
      \overline{B})}
  \end{equation}
  We remind the reader that $\nrm{\ud a}_{L^{2}_{x}(\rglue B \setminus
    \overline{B})}^{2} \leq \calE_{\rglue B \setminus
    \overline{B}}[a,e,f,g]$.  The connection 1-form $\check{a}$ can be
  safely excised outside $\rglue^{(2)} B \supset \rglue^{(1)}
  B$. Finally, we patch together $a_{j}$ and $\widetilde{a}_{j}$
  inside $\rglue^{(2)} B$ using a suitable gauge transformation to
  produce $\widetilde{a}$ satisfying
  \eqref{eq:gluing:basicprops4atilde}--\eqref{eq:gluing:energy4atilde}.

  We now proceed to the details. Let $O = O(\rglue, B)$ denote the
  annulus $\rglue B \setminus \overline{B}$. Applying Proposition
  \ref{prop:hodge4oneforms} with $F = \ud a$ on the region $O$ (which
  is possible since $H^{1}_{\bbR}(O) = 0$), we infer the existence of
  a unique 1-form $\check{a}$ which solves
  \begin{equation} \label{eq:gluing:acheck} \ud \check{a} = \ud a,
    \quad \rd^{\ell} \check{a}_{\ell} = 0, \quad \check{a} \rst_{\rd
      B} (\rd_{r}) = \check{a} \rst_{\rd (\rglue B)} (\rd_{r}) = 0.
  \end{equation}
  Moreover, $\check{a}$ obeys the estimate
  \eqref{eq:gluing:est4acheck}. To see this, first observe that this
  estimate follows from \eqref{eq:hodge4oneform:est} and Sobolev when
  $B$ is a ball of unit radius. The general case follows once we note
  that, for a fixed $\rglue > 1$, both sides of
  \eqref{eq:gluing:est4acheck} are invariant under scaling.

  Next, we prove that $\check{a}$ is gauge equivalent to $a$. This
  amounts to finding a function $\chi$ such that
  \begin{equation*}
    \check{a} = a - \ud \chi.
  \end{equation*}
  Since $\ud(\check{a} - a) = 0$, the existence of such a function
  $\chi$ on $O$ is guaranteed by the topological fact that
  $\firstDR(O) = 0$; it is moreover unique if we furthermore require
  that $\int_{O} \chi = 0$. By Poincar\'e's inequality and
  \eqref{eq:gluing:est4acheck}, it follows that $\chi$ satisfies the
  bound
  \begin{equation} \label{eq:gluing:est4chi} \nrm{\rd_{x}^{(2)}
      \chi}_{L^{2}_{x}(O)} + \nrm{\rd_{x} \chi}_{L^{4}_{x}(O)} +
    (\diam B)^{-2} \nrm{\chi}_{L^{2}_{x}(O)} \aleq
    \nrm{a}_{\dot{H}^{1}_{x} \cap L^{4}_{x}(O)}.
  \end{equation}

  We now show that, thanks to \eqref{eq:gluing:est4acheck}, it is safe
  to cut off $\check{a}$. Let $\eta_{(2)}$ be a smooth function on
  $\bbR^{4}$ such that
  \begin{equation*}
    \eta_{(2)} = 1\hbox{ on } \rglue^{(2)} B,\quad
    \eta_{(2)} = 0 \hbox{ outside } \rglue^{(3)} B, \quad
    \abs{\rd_{x}^{(N)} \eta_{(2)}} \aleq_{N, \rglue} (\diam B)^{-N} \hbox{ for } N \geq 0.
  \end{equation*}
  Then by \eqref{eq:gluing:est4acheck}, it is immediate that for any
  open subset $O' \subseteq O$,
  \begin{equation} \label{eq:gluing:energy4acheck-cutoff}
    \nrm{\ud(\eta_{(2)} \check{a})}_{L^{2}_{x}(O')} \leq
    \nrm{\eta_{(2)} \ud \check{a}}_{L^{2}_{x}(O')} + \nrm{\rd_{x}
      \eta_{(2)}}_{L^{\infty}_{x}(O')} \nrm{\check{a}}_{L^{2}_{x}(O')}
    \aleq_{\rglue}\nrm{\ud a}_{L^{2}_{x}(O')}.
  \end{equation}

  To conclude the proof, we finally patch $a$ and $\check{a}$ by a
  suitable gauge transformation to obtain $\widetilde{a}$ with the
  desired properties. Let $\eta_{(1)}$ be a smooth function on
  $\bbR^{4}$ such that
  \begin{equation*}
    \eta_{(1)} = 1 \hbox{ on } \rglue^{(1)} B,\quad
    \eta_{(1)} = 0 \hbox{ outside } \rglue^{(2)} B, \quad
    \abs{\rd_{x}^{(N)} \eta_{(1)}} \aleq_{N, \rglue} (\diam B)^{-N} \hbox{ for } N \geq 0.
  \end{equation*}
  We now define $\widetilde{a}$ by the following formula:
  \begin{equation} \label{eq:gluing:atilde} \widetilde{a} :=
    \eta_{(2)} (a - \ud \widetilde{\chi}), \quad \hbox{ where }
    \widetilde{\chi} := (1-\eta_{(1)}) \chi.
  \end{equation}
  From \eqref{eq:gluing:est4chi}, it follows that $\widetilde{\chi}$
  obeys
  \begin{equation} \label{eq:gluing:est4chitilde:0}
    \nrm{\rd_{x}^{(2)}\widetilde{\chi}}_{L^{2}_{x}(\rglue B \setminus
      \overline{B})} + \nrm{\rd_{x}
      \widetilde{\chi}}_{L^{4}_{x}(\rglue B \setminus \overline{B})} +
    (\diam B)^{-2} \nrm{\widetilde{\chi}}_{L^{2}_{x}(\rglue B
      \setminus \overline{B})} \aleq \nrm{a}_{\dot{H}^{1}_{x} \cap
      L^{4}_{x}(O)}
  \end{equation}
  It remains to verify the properties
  \eqref{eq:gluing:basicprops4atilde}--\eqref{eq:gluing:energy4atilde}. The
  first property \eqref{eq:gluing:basicprops4atilde} follows easily
  from the construction. The second property
  \eqref{eq:gluing:est4atilde} follows from
  \begin{equation} \label{eq:gluing:est4chitilde} \nrm{\eta_{(2)} \ud
      \widetilde{\chi}}_{\dot{H}^{1}_{x} \cap L^{4}_{x}(\bbR^{4}
      \setminus \overline{B})} \aleq_{\rglue} \nrm{\ud
      \widetilde{\chi}}_{\dot{H}^{1}_{x} \cap L^{4}_{x}(\rglue B
      \setminus \overline{B})} \aleq_{\rglue} \nrm{a}_{\dot{H}^{1}_{x}
      \cap L^{4}_{x} (\rglue B \setminus \overline{B})},
  \end{equation}
  which in turn follows from \eqref{eq:gluing:est4chitilde:0}.
  Finally, the third property \eqref{eq:gluing:energy4atilde} is a
  consequence of \eqref{eq:gluing:basicprops4atilde} and
  \eqref{eq:gluing:energy4acheck-cutoff} with $O' = \rglue B \setminus
  \rglue^{(1)} \overline{B}$.

  \pfstep{Step 2. Excision of $f, g$} In this step, we excise $(f, g)$
  to construct $(\widetilde{f}, \widetilde{g})$ on $\bbR^{4} \setminus
  \overline{B}$ that satisfies the following properties:
  \begin{align}
    (\widetilde{f}, \widetilde{g}) = (f, g) \hbox{ on } \rgExt B
    \setminus \overline{B}, & \quad
    (\widetilde{f}, \widetilde{g}) = 0 \hbox{ on } \bbR^{4} \setminus \rglue^{(1)} B, \label{eq:gluing:basicprops4fgtilde} \\
    \nrm{\widetilde{f}}_{\dot{H}^{1}_{x} \cap L^{4}_{x} (\bbR^{4}
      \setminus \overline{B})}
    & \aleq  \nrm{f}_{\dot{H}^{1}_{x} \cap L^{4}_{x}(\rglue B \setminus \overline{B})}, \label{eq:gluing:est4ftilde}\\
    \nrm{\widetilde{g}}_{L^{2}_{x} (\bbR^{4} \setminus \overline{B})}
    & \aleq  \nrm{g}_{L^{2}_{x}(\rglue B \setminus \overline{B})} , \label{eq:gluing:est4gtilde} \\
    \sum_{j=1, \ldots 4} \nrm{\widetilde{\covD}_{j}
      \widetilde{f}}_{L^{2}_{x}(\bbR^{4} \setminus \overline{B})}^{2}
    & \aleq (\diam B)^{-2} \nrm{    f}_{L^{2}_{x}(\rglue B \setminus \overline{B})}^{2} + \sum_{j =
      1, \ldots, 4} \nrm{\covD_{j} f}^{2}_{L^{2}_{x}(\rglue B
      \setminus \overline{B})}, \label{eq:gluing:energy4ftilde}
  \end{align}
  where $\widetilde{\covD}_{j} = \rd_{j} + i \widetilde{a}_{j}$. 
These conditions are easily achieved by naively choosing $\rgExt =
  \rglue^{(0)}$ and cutting off $f, g$ by a smooth function $\eta_{(0)}$ that is
  supported in $\rglue^{(1)} B$ and equals $1$ on $\rglue^{(0)} B$. 

  \pfstep{Step 3. Excision and gluing of $e_{j}$} In this step, we
  construct $\widetilde{e}_{j}$ that, together with
  $\widetilde{a}_{j}$, $\widetilde{f}$ and $\widetilde{g}$ constructed
  in the preceding steps, would satisfy the properties in Proposition
  \ref{prop:gluing}.  The problem of localizing of $\widetilde{e}_{j}$
  is subtle, as it must satisfy the Gauss equation
  \begin{equation} \label{eq:gluing:diveq4etilde} \rd^{\ell}
    \widetilde{e}_{\ell} = \Im [\widetilde{f} \,
    \overline{\widetilde{g}}].
  \end{equation}
  In particular, integrating \eqref{eq:gluing:diveq4etilde} over a
  ball $B_{r}$ of radius $r \gg 1$, the divergence theorem implies
  \begin{equation*}
    \int_{\rd B_{r}} \widetilde{e}_{\ell} \bfn^{\ell} = \int_{\bbR^{4}} \Im[\widetilde{f} \, \overline{\widetilde{g}}] \, \ud x, \quad \hbox{ where } \bfn^{\ell} = \frac{x^{\ell}}{\abs{x}},
  \end{equation*}
  which precludes the possibility of having a compactly supported
  $\widetilde{e}$ in general. Instead, we will \emph{glue} the 1-form
  $e$ to another solution $e_{(q)}$ (see \eqref{eq:gluing:eQ}) to the
  Gauss equation with a well-understood behavior at infinity, while
  keeping $e$ unchanged in the region $\rgExt B$. The key to carrying
  out this procedure is Proposition \ref{prop:sol4div}, which allows
  us to solve away certain errors in the Gauss equation in a bounded
  region of space.

  We define $\widetilde{e}$ to be
  \begin{equation} \label{eq:gluing:etilde} \widetilde{e} =
    \eta_{(0)}^{2} e + (1 - \eta_{(0)}^{2}) e_{(q)} + e_{(G)},
  \end{equation}
  where $\set{e_{(q)}}_{q \in \bbR}$ is an explicit 1-parameter family
  of solutions to $\rd^{\ell} e_{(q) \ell} = 0$ on $\bbR^{4} \setminus
  \set{0}$, to be introduced below, and $e_{(G)}$ will be constructed
  to satisfy the equation
  \begin{equation} \label{eq:gluing:diveq4eG} \rd^{\ell} e_{(G) \ell}
    = - \rd^{\ell} \eta_{(0)}^{2} (e_{\ell} - e_{(q) \ell}) \quad
    \hbox{ with } \quad \supp \, e_{(G)} \subseteq \rglue^{(1)} B
    \setminus \rgExt \overline{B}.
  \end{equation}
  For $e_{(q)}$ and $e_{(G)}$ as above, we can readily verify that
  \eqref{eq:gluing:diveq4etilde} holds as follows:
  \begin{align*}
    \rd^{\ell} \widetilde{e}_{\ell} - \Im[\widetilde{f} \,
    \overline{\widetilde{g}}]
    =& \rd^{\ell} (\eta_{(0)}^{2} e_{\ell} + (1 - \eta_{(0)}^{2}) e_{(q) \ell} + e_{(G) \ell}) - \eta_{(0)}^{2} \Im[f \overline{g}] \\
    =& \eta_{(0)}^{2} (\rd^{\ell} e_{\ell} - \Im[f \overline{g}]) +
    (\rd^{\ell} \eta_{(0)}^{2}) (e_{\ell} - e_{(q) \ell}) + \rd^{\ell}
    e_{(G) \ell} = 0.
  \end{align*}

  The 1-form $e_{(q)}$ is defined on $\bbR^{4} \setminus \set{0}$
  component-wisely as follows:
  \begin{equation} \label{eq:gluing:eQ} e_{(q) j} = \frac{q}{2
      \pi^{2}} \frac{x^{j}}{\abs{x}^{4}}.
  \end{equation}
  Note that $e_{(q)}$ is precisely the electric field of a point
  charge at the origin given by the 4-dimensional version of Coulomb's
  law.  Indeed, $e_{(q)}$ satisfies the free divergence equation
  \begin{equation} \label{eq:gluing:divEq4eQ} \rd^{\ell} e_{(q) \ell}
    = 0,
  \end{equation}
  and the charge of $e_{(q)}$ measured on any sphere $\rd B_{r}$ of
  radius $r$ centered at the origin (in fact, any hypersurface
  enclosing the origin) equals $q$, i.e.,
  \begin{equation} \label{eq:gluing:charge4eQ} \int_{\rd B_{r}} e_{(q)
      \ell} \bfn^{\ell} = q \quad \hbox{ where } \bfn^{\ell} =
    \frac{x^{\ell}}{\abs{x}}.
  \end{equation}

  We now turn to the construction of $e_{(G)}$. We wish to apply
  Corollary \ref{cor:sol4div:annulus}; thus we must ensure that
  \begin{equation} \label{eq:gluing:intCond} 0 = \int \rd^{\ell}
    \eta_{(0)}^{2} (e_{\ell} - e_{(q) \ell}) \, \ud x.
  \end{equation}
  By \eqref{eq:gluing:charge4eQ} and the divergence theorem, we
  compute
  \begin{equation*}
    \int \rd^{\ell} \eta_{(0)}^{2} \, e_{(q) \ell} \, \ud x= q.
  \end{equation*}
  Thus, \eqref{eq:gluing:intCond} dictates the following choice of $q$
  as a function of $e$ for a fixed $\rglue$:
  \begin{equation} \label{eq:gluing:q} q[e] := \int \rd^{\ell}
    \eta_{(0)}^{2} \, e_{\ell} \, \ud x.
  \end{equation}
  Since $\rd^{\ell} \eta_{(0)}$ is supported in $\rglue^{(1)} B
  \setminus (\rgExt + \dlt \rgExt) \overline{B} \subseteq \rglue B
  \setminus \overline{B}$, we have
  \begin{equation*}
    \abs{q} \aleq \int_{\rglue B \setminus \overline{B}} \frac{1}{\abs{x}} \abs{e} \, \ud x \aleq_{\rglue} (\diam B) \nrm{e}_{L^{2}_{x}(\rglue B \setminus \overline{B})}.
  \end{equation*}
  Therefore, the $L^{2}$ norm of $e_{(q)}$ obeys the bound
  \begin{equation} \label{eq:gluing:est4eQ}
    \nrm{e_{(q)}}_{L^{2}_{x}(\bbR^{4} \setminus \overline{B})} \aleq
    \abs{q} \nrm{\frac{1}{\abs{x}^{3}}}_{L^{2}_{x}(\bbR^{4} \setminus
      \overline{B})} \aleq \nrm{e}_{L^{2}_{x}(\rglue B
      \setminus \overline{B})}.
  \end{equation}
  Similarly, we also have
  \begin{align*}
    \nrm{\rd^{\ell} \eta_{(1)}^{2} (e_{\ell} - e_{(q)
        \ell})}_{L^{\frac{4}{3}}_{x}(\bbR^{4})} \aleq \nrm{e
      - e_{(q)}}_{L^{2}_{x}(\rglue B \setminus \overline{B})}
    \aleq \nrm{e}_{L^{2}_{x}(\rglue B \setminus
      \overline{B})} \, .
  \end{align*}
  Applying Corollary \ref{cor:sol4div:annulus} with $A = \rglue^{(1)}
  B \setminus \rgExt \overline{B}$  we obtain a solution $e_{(G)}$ to the
  problem \eqref{eq:gluing:diveq4eG} that satisfies
  \begin{equation} \label{eq:gluing:est4eG}
    \nrm{e_{(G)}}_{L^{2}_{x}(\bbR^{4})} \aleq_{\rglue}
    \nrm{e}_{L^{2}_{x}(\rglue B \setminus \overline{B})}.
  \end{equation}
  Combined with \eqref{eq:gluing:est4atilde},
  \eqref{eq:gluing:energy4atilde}, \eqref{eq:gluing:est4ftilde},
  \eqref{eq:gluing:est4gtilde}, \eqref{eq:gluing:energy4ftilde},
  \eqref{eq:gluing:etilde} and \eqref{eq:gluing:est4eQ}, estimates
  \eqref{eq:gluing:est} and \eqref{eq:gluing:energy} follow. The proof
  of Statements (\ref{item:gluing:1})--(\ref{item:gluing:3}) of
  Proposition \ref{prop:gluing} is therefore complete.

  \pfstep{Step 4. Continuity and persistence of regularity} It remains
  to verify Statement (\ref{item:gluing:4}) of Proposition
  \ref{prop:gluing}.  Inspection of our proof so far (using also the
  linearity statement in Corollary \ref{cor:sol4div:annulus}) shows
  that $\widetilde{a}$, $\widetilde{e}$, $\widetilde{f}$ and
  $\widetilde{g}$ are in fact \emph{linear} in $a$, $e$, $f$ and $g$,
  respectively; thus the continuity statement is a triviality.
  Checking the persistence of regularity property is a routine
  exercise using the corresponding statements in Corollary
  \ref{cor:sol4div:annulus} and Proposition \ref{prop:hodge4oneforms};
  we omit the details.
\end{proof}

Next, we prove Proposition \ref{prop:intGluing}. The main idea is the
same as for the preceding proof of Proposition \ref{prop:gluing}; the
key difference is the choice of an 1-parameter family of solutions
$e_{(p)}$ to the Gauss equation in Step 3, which now must be regular
at the origin.

\begin{proof} [Proof of Proposition \ref{prop:intGluing}]
  As before, we may assume that $B$ is centered at the origin of
  $\bbR^{4}$. For any given $1 < \rgInt < \rglue \leq 2$, we define $1 = \rglue^{(-3)} <
  \rglue^{(-2)} < \rglue^{(-1)} < \rglue^{(0)} = \rgInt
  < \rglue$ as
  \begin{equation*}
    \rglue^{(-3)} = 1, \quad
    \rglue^{(-2)} = \frac{2}{3} + \frac{1}{3} \rglue, \quad
    \rglue^{(-1)} = \frac{1}{2} + \frac{1}{2} \rglue, \quad
    \rglue^{(0)} = \rgInt.
  \end{equation*}
  In what follows, we will write $E^{\intr}[a,e,f,g] = (\widetilde{a},
  \widetilde{e}, \widetilde{f}, \widetilde{g})$.


  \pfstep{Step 1. Excision of $a_{j}$} This step is very similar to
  Step 1 in the proof of Proposition \ref{prop:gluing}, except that we
  now excise the data in the inner part of the annulus. The goal is to
  construct $\widetilde{a}$ on $\rglue B$ such that the following
  properties hold:
  \begin{align}
    \widetilde{a} = a \hbox{ on } \rglue B \setminus \rglue^{(-1)}
    \overline{B}, & \quad
    \widetilde{a} = 0 \hbox{ on } \rglue^{(-3)} B, \label{eq:intGluing:basicprops4atilde}\\
    \nrm{\widetilde{a}}_{\dot{H}^{1}_{x} \cap L^{4}_{x}(\rglue B)}
    & \aleq_{\rglue}  \nrm{a}_{\dot{H}^{1}_{x} \cap L^{4}_{x}(\rglue B \setminus \overline{B})}, \label{eq:intGluing:est4atilde} \\
    \nrm{\ud \widetilde{a}}_{L^{2}_{x}(\rglue B)}^{2} & \aleq_{\rglue}
    \calE_{\rglue B \setminus
      \overline{B}}[a,e,f,g]. \label{eq:intGluing:energy4atilde}
  \end{align}

  Let $O = O(\rglue, B)$ denote the annulus $\rglue B \setminus
  \overline{B}$. Applying Proposition \ref{prop:hodge4oneforms} with
  $F = \ud a$ on $O$, we obtain a unique 1-form $\check{a}$ that
  satisfies \eqref{eq:gluing:est4acheck}--\eqref{eq:gluing:acheck},
  and also a function $\chi$ satisfying $\check{a} = a - \ud \chi$,
  $\int_{O} \chi = 0$ and \eqref{eq:gluing:est4chi}. Let
  $\eta_{(-3)}$, $\eta_{(-2)}$ be smooth function on $\bbR^{4}$ such
  that
  \begin{gather*}
    \eta_{(-3)} = 0\hbox{ on } \rglue^{(-3)}B,\quad \eta_{(-3)} = 1
    \hbox{ outside } \rglue^{(-2)} B, \quad
    \abs{\rd_{x}^{(N)} \eta_{(-3)}} \aleq_{N, \rglue} (\diam B)^{-N} \hbox{ for } N \geq 0, \\
    \eta_{(-2)} = 0\hbox{ on } \rglue^{(-2)} B,\quad \eta_{(-2)} = 1
    \hbox{ outside } \rglue^{(-1)} B, \quad \abs{\rd_{x}^{(N)}
      \eta_{(-2)}} \aleq_{N, \rglue} (\diam B)^{-N} \hbox{ for } N
    \geq 0.
  \end{gather*}
  We define
  \begin{equation} \label{eq:intGluing:atilde} \widetilde{a} :=
    \eta_{(-3)} (a - \ud \widetilde{\chi}), \quad \hbox{ where }
    \widetilde{\chi} := (1-\eta_{(-2)}) \chi.
  \end{equation}
  Then proceeding as before, it can be checked that $\widetilde{a}$
  satisfies
  \eqref{eq:intGluing:basicprops4atilde}--\eqref{eq:intGluing:energy4atilde}.

  \pfstep{Step 2. Excision of $f$, $g$} We seek to construct $f', g'$
  on $\rglue B$ such that
  \begin{align}
    (f', g') = (f, g) \hbox{ on } \rglue B \setminus \rgInt
    \overline{B}, & \quad
    (f', g') = 0 \hbox{ on } \bbR^{4} \setminus \rglue^{(-1)} B, \label{eq:intGluing:basicprops4fgprime} \\
    \nrm{f'}_{\dot{H}^{1}_{x} \cap L^{4}_{x} (\rglue B)}
    & \aleq_{\rglue}  \nrm{f}_{\dot{H}^{1}_{x} \cap L^{4}_{x}(\rglue B \setminus \overline{B})}, \label{eq:intGluing:est4fprime}\\
    \nrm{g'}_{L^{2}_{x} (\rglue B)}
    & \aleq  \nrm{g}_{L^{2}_{x}(\rglue B \setminus \overline{B})} , \label{eq:intGluing:est4gprime} \\
    \sum_{j=1, \ldots 4} \nrm{\widetilde{\covD}_{j}
      f'}_{L^{2}_{x}(\rglue B)}^{2} & \aleq_{\rglue}
    \nrm{\frac{1}{\abs{x}} f}_{L^{2}_{x}(\rglue B \setminus
      \overline{B})}^{2} + \sum_{j = 1, \ldots, 4} \nrm{\covD_{j}
      f}^{2}_{L^{2}_{x}(\rglue B \setminus
      \overline{B})}, \label{eq:intGluing:energy4fprime}
  \end{align}
  where $\widetilde{\covD}_{j} = \rd_{j} + i \widetilde{a}_{j}$ and
  $\rgInt = \rglue^{(0)} = \frac{1 + \rglue}{2}$.

  Let $\eta_{(-1)}$ be a smooth function on $\bbR^{4}$ such that
  \begin{equation*}
    \eta_{(-1)} = 0\hbox{ on } \rglue^{(-1)} B,\quad
    \eta_{(-1)} = 1 \hbox{ outside } \rglue^{(0)} B, \quad
    \abs{\rd_{x}^{(N)} \eta_{(-1)}} \aleq_{N, \rglue} (\diam B)^{-N} \hbox{ for } N \geq 0.
  \end{equation*}
  We simply define
  \begin{equation} \label{eq:intGluing:fgprime} f' = \eta_{(-1)} f,
    \quad g' = \eta_{(-1)} g.
  \end{equation}
  Then
  \eqref{eq:intGluing:basicprops4fgprime}--\eqref{eq:intGluing:energy4fprime}
  can be easily verified.

  \pfstep{Step 3. Excision and gluing of $f$, $g$ and $e_{j}$} In this
  step, we finally define $\widetilde{e}, \widetilde{f}$ and
  $\widetilde{g}$ on $\rglue B$. As remarked above, the basic idea is
  similar to that in Step 3 of the proof of Proposition
  \ref{prop:gluing}. However, the 1-forms $\set{e_{(q)}}_{q \in \bbR}$
  are not suitable for gluing to the cutoff of $e$ outside a ball
  centered at the origin, since each $e_{(q)}$ (with $q \neq 0$) is
  singular at $0$. Thus we need to devise a different one parameter
  family of initial data sets. To have a solution to the Gauss
  equation with a nontrivial electric charge while being regular, we
  need to introduce a non-trivial charge density $\Im[f_{(p)}
  \overline{g_{(p)}}]$ as well as $e_{(p)}$, where $p$ is the charge
  parameter.

  Let $\zt$ be a smooth function on $\bbR^{4}$ such that
  \begin{equation*}
    \zt \geq 0, \quad
    \zt = 0 \hbox{ outside } B, \quad
    \int_{\bbR^{4}} \zt^{2} \, \ud x = 1, \quad
    \abs{\rd_{x}^{(N)} \zt} \aleq_{N} (\diam B)^{-N-2} .
  \end{equation*}
  Then for $p \in \bbR$, we define
  \begin{align}
    e_{(p)j}
    &= - p (- \lap)^{-1} \rd_{j} \zt^{2}, \label{eq:intGluing:eP} \\
    f_{(p)} &= \sqrt{p} (\diam B)^{\frac{1}{2}} \, \zt, \label{eq:intGluing:fP} \\
    g_{(p)} &= - i (\diam B)^{-1} f_{(p)} = - i \sqrt{p} (\diam
    B)^{-\frac{1}{2}} \zt. \label{eq:intGluing:gP}
  \end{align}
  Note that $(e_{(p)}, f_{(p)}, g_{(p)})$ solves the Gauss equation
  \begin{equation} \label{eq:intGluing:diveq4efgP} \rd^{\ell} e_{(p)
      \ell} = \Im[f_{(p)} \overline{g_{(p)}}],
  \end{equation}
  and obeys the following properties:
  \begin{gather}
    \int_{r B} \rd^{\ell} e_{(p) j} = p \quad \hbox{ for any } r > 1, \\
    \nrm{e_{(p)}}_{L^{2}_{x}} \aleq p (\diam B)^{-1}, \quad
    \nrm{f_{(p)}}_{\dot{H}^{1}_{x} \cap L^{4}_{x}(\bbR^{4})} +
    \nrm{g_{(p)}}_{L^{2}_{x}(\bbR^{4})} \aleq \sqrt{p} (\diam
    B)^{-\frac{1}{2}}. \label{eq:intGluing:est4efgP}
  \end{gather}

  Recall the definitions of $\eta_{(-1)}$, $f'$, $g'$ from the
  previous step. We define $(\widetilde{e}, \widetilde{f},
  \widetilde{g})$ as follows:
  \begin{align*}
    \widetilde{e} =& \eta_{(-1)}^{2} e + (1- \eta_{(-1)}^{2}) e_{(p)} + e_{(G)}\\
    \widetilde{f} =& f' + f_{(p)}  \\
    \widetilde{g} =& g' + g_{(p)}
  \end{align*}
  where $e_{(G)}$ will be constructed so that
  \begin{equation} \label{eq:intGluing:diveq4eG} \rd^{\ell} e_{(G)
      \ell} = - \rd^{\ell} \eta_{(-1)}^{2} (e_{\ell} - e_{(p) \ell})
    \quad \hbox{ with } \quad \supp \, e_{(G)} \subseteq \rglue^{(0)}
    B \setminus \overline{B}.
  \end{equation}
  Note that
  \begin{equation*}
    \supp\, f' \cup \supp \, g' \subseteq \supp \, \eta_{(-1)}
    \quad \hbox{ and } \quad
    \supp \, \eta_{(-1)} \cap (\supp \, f_{(p)} \cup \supp \, g_{(p)}) = \0.
  \end{equation*}
  Using these properties, we can verify that $(\widetilde{e},
  \widetilde{f}, \widetilde{g})$ solves the Gauss equation as follows:
  \begin{align*}
    \rd^{\ell} \widetilde{e}_{\ell} - \Im[ \widetilde{f} \,
    \overline{\widetilde{g}} ]
    =& \rd^{\ell} (\eta_{(-1)}^{2} e_{\ell} + (1 - \eta_{(-1)}^{2}) e_{(p) \ell} + e_{(G) \ell}) \\
    &- \eta_{(-1)}^{2} \Im[f \overline{g}] - (1 - \eta_{(-1)}^{2}) \Im[f_{(p)} \overline{g_{(p)}}] \\
    =& \rd^{\ell} \eta_{(-1)}^{2} (e_{\ell} - e_{(p) \ell}) +
    \rd^{\ell} e_{(G) \ell} = 0.
  \end{align*}

  In order to apply Corollary \ref{cor:sol4div:annulus}, we need
  \begin{equation*}
    0 = \int \rd^{\ell} \eta_{(-1)}^{2} (e_{\ell} - e_{(p) \ell}) \, \ud x,
  \end{equation*}
  which enforces the following choice of $p$ as a function of $e$ for
  a fixed $\rglue$:
  \begin{equation} \label{eq:intGluing:p} p[e] := \int \rd^{\ell}
    \eta_{(-1)}^{2} e_{\ell} \, \ud x.
  \end{equation}
  As before, $p$ obeys the bound
  \begin{equation} \label{eq:intGluing:est4p} \abs{p} \aleq_{\rglue}
    (\diam B) \nrm{e}_{L^{2}_{x}(\rglue B \setminus \overline{B})},
  \end{equation}
  and therefore
  \begin{equation*}
    \nrm{\rd^{\ell} \eta_{(-1)}^{2} (e_{\ell} - e_{(p) \ell})}_{L^{\frac{4}{3}}_{x}(\bbR^{4})}
    \aleq_{\rglue} \nrm{e - e_{(p)}}_{L^{2}_{x}(\rglue B \setminus \overline{B})}
    \aleq_{\rglue} \nrm{e}_{L^{2}_{x}(\rglue B \setminus \overline{B})}.
  \end{equation*}
  Now applying Corollary \ref{cor:sol4div:annulus} with $A =
  \rglue^{(0)} B \setminus \rglue^{(-1)} \overline{B}$  we obtain a solution $e_{(G)}$
  to the problem \eqref{eq:intGluing:diveq4eG} such that
  \begin{equation} \label{eq:intGluing:est4eG}
    \nrm{e_{(G)}}_{L^{2}_{x}(\bbR^{4})} \aleq_{\rglue}
    \nrm{e}_{L^{2}_{x}(\rglue B \setminus \overline{B})}.
  \end{equation}

  From \eqref{eq:intGluing:est4atilde},
  \eqref{eq:intGluing:energy4atilde}, \eqref{eq:intGluing:est4fprime},
  \eqref{eq:intGluing:est4gprime}, \eqref{eq:intGluing:energy4fprime},
  \eqref{eq:intGluing:est4efgP}, \eqref{eq:intGluing:est4p} and
  \eqref{eq:intGluing:est4eG}, estimates \eqref{eq:intGluing:est} and
  \eqref{eq:intGluing:energy} follow. Thus the proof of Statements
  (\ref{item:intGluing:2})--(\ref{item:intGluing:1}) of Proposition
  \ref{prop:intGluing} is complete.

  \pfstep{Step 4. Continuity and persistence of regularity} To
  complete the proof, we need to establish Statement
  (\ref{item:intGluing:3}) of Proposition \ref{prop:intGluing}. As in
  Proposition \ref{prop:gluing}, this task is a routine exercise of
  inspecting the proofs so far; we omit the details.
\end{proof}

\section{Local geometric uniqueness of Maxwell-Klein-Gordon} \label{sec:locGeom}
In this section we formulate and prove local geometric uniqueness (i.e., uniqueness up to a gauge transformation) of Maxwell-Klein-Gordon equations at the energy regularity. In Section~\ref{subsec:admSol}, we formulate the notion of an admissible $C_{t} \calH^{1}$ solution and the associated class $C_{t} \calG^{2}$ of gauge transformations, which provides an adequate setting for local geometric uniqueness. Then in Section~\ref{subsec:locGeom}, we state and prove the local geometric uniqueness of \eqref{eq:MKG} in the class of admissible $C_{t} \calH^{1}$ solutions (Proposition~\ref{prop:locGeom}).

\subsection{Admissible $C_{t} \calH^{1}$ solutions and gauge transformations} \label{subsec:admSol}
Here we introduce the notions of \emph{classical} and \emph{admissible $C_{t} \calH^{1}$ solutions} to \eqref{eq:MKG}. Classical solutions refer to smooth solutions to \eqref{eq:MKG} with sufficient spatial decay, and admissible $C_{t} \calH^{1}$ solutions are defined as local-in-time limits of classical solutions in the \emph{energy topology} $C_{t} \calH^{1}$, to be defined below. We also define the associated classes of gauge transformations.

Given an open set $\calO \subseteq \bbR^{1+4}$ and a pair $(A_{\mu}, \phi)$, we define the $C_{t} \calH^{1}(\calO)$ norm of $(A_{\mu}, \phi)$ to be
\begin{equation*}
	\nrm{(A_{\mu}, \phi)}_{C_{t} \calH^{1}(\calO)} 
	:= \esssup_{t \in I(\calO)} (\nrm{(A_{\mu}, \phi)}_{\dot{H}^{1}_{x} \cap L^{4}_{x}(\calO_{t})} 	
						+ \nrm{(\rd_{t} A_{\mu}, \rd_{t} \phi)}_{L^{2}_{x}(\calO_{t})}),
\end{equation*}
where $\calO_{t} := \calO \cap (\set{t} \times \bbR^{4})$ and $I(\calO) := \set{t \in \bbR : \calO_{t} \neq \emptyset}$. Similarly, we define the $C_{t} \calG^{2}(\calO)$ norm to be
\begin{equation*}
	\nrm{\chi}_{C_{t} \calG^{2}(\calO)} 
	:= \esssup_{t \in I(\calO)} (\nrm{\chi}_{\dot{H}^{2}_{x} \cap \dot{W}^{1, 4}_{x} \cap \BMO (\calO_{t})}
						+ \nrm{\rd_{t} \chi}_{\dot{H}^{1}_{x} \cap L^{4}_{x}(\calO_{t})} 
						+ \nrm{\rd_{t}^{2} \chi}_{L^{2}_{x}(\calO_{t})} )
\end{equation*}
We will say that a smooth solution $(A, \phi)$ is \emph{classical}, and write $(A, \phi) \in C_{t}^{\infty} \calH^{\infty}(\calO)$, if
\begin{equation*}
	(\rd_{t, x}^{(N)} A_{\mu}, \rd_{t,x}^{(N)} \phi) \in C_{t} \calH^{1}(\calO) \hbox{ for all } N \geq 0 \hbox{ and }
	(A_{\mu}, \phi) \in C_{t} (I(\calO) ; L^{2}_{x}(\calO_{t})).
\end{equation*}
We similarly define the space $C_{t}^{\infty} \calG^{\infty}(\calO)$ of \emph{classical gauge transformations} by saying that $\chi \in C_{t}^{\infty} \calG^{\infty}(\calO)$ if and only if
\begin{equation*}
	\chi \in C_{t}(I(\calO) ; L^{4}_{x}(\calO_{t})) \hbox{ and } 
	\rd_{t,x}^{(N)} \chi \in C_{t} (I(\calO) ; L^{2}_{x}(\calO_{t})) \hbox{ for every } N \geq 1.
\end{equation*}
We define the notion of a \emph{admissible $C_{t} \calH^{1}$ solution} to \eqref{eq:MKG} and gauge equivalence between two such solutions as follows.
\begin{definition}[Admissible $C_{t} \calH^{1}$ solutions] \label{def:sol2MKG}
Let $\calO$ be an open subset of $\bbR^{1+4}$. 
\begin{enumerate}
\item We say that a pair $(A_{\mu}, \phi) \in C_{t} \calH^{1}(\calO)$ is an \emph{admissible $C_{t} \calH^{1}$ solution} to \eqref{eq:MKG} on $\calO$ (or \emph{admissible $C_{t} \calH^{1}(\calO)$ solution}) if it can be approximated by a sequence $(A^{(n)}_{\mu}, \phi^{(n)})$ of classical solutions to \eqref{eq:MKG} locally in time with respect to the $C_{t} \calH^{1}$ norm. More precisely, for every compact interval $J \subseteq I(\calO)$, we have as $n \to \infty$,
\begin{align*}
	\nrm{(A_{\mu}, \phi) - (A^{(n)}_{\mu}, \phi^{(n)})}_{C_{t} \calH^{1}(\calO \cap (J \times \bbR^{4}))} \to 0.
\end{align*}
\item We say that two admissible $C_{t} \calH^{1}(\calO)$ solutions $(A_{\mu}, \phi)$ and $(A'_{\mu}, \phi')$ are \emph{gauge equivalent} if there exists a gauge transform $\chi \in C_{t} \calG^{2}(\calO)$  such that $A_{\mu} = A'_{\mu} - \rd_{\mu} \chi$, $\phi = \phi' e^{i \chi}$. 
\end{enumerate}
\end{definition}

\subsection{Local geometric uniqueness of an admissible $C_{t}
  \calH^{1}$ solution} \label{subsec:locGeom}

In this subsection, we state and prove the geometric uniqueness of an
admissible $C_{t} \calH^{1}$ solution of \eqref{eq:MKG}. As discussed
earlier, this statement can be thought of as the gauge invariant
version of finite speed of propagation for \eqref{eq:MKG}.

Before stating the main result (Proposition~\ref{prop:locGeom}), we
need to make a few definitions. Given a point $(t_{0},x_{0}) \in
\bbR^{1+4}$, we define its \emph{causal past} $J^{-}(t_{0},x_{0})$ to
be past-directed light cone with $(t_{0}, x_{0})$ as the tip, i.e.,
\begin{equation*}
  J^{-}(t_{0}, x_{0}) := \set{(t,x) \in \bbR^{1+4} : t \leq t_{0}, \, \abs{x-x_{0}} \leq t_{0}-t}.
\end{equation*}
For an open subset $B \subseteq \set{t_{0}} \times \bbR^{4}$, we define
its \emph{future domain of dependence} $\calD^{+}(B)$ to be
\begin{equation*}
  \calD^{+}(B) = \set{(t,x) \in \calO : J^{-}(t,x) \cap (\set{t_{0}} \times \bbR^{4}) \subseteq B}.
\end{equation*}
For example, when $B$ is an open ball of radius $r_{0} >0$ in
$\set{t_{0}} \times \bbR^{4}$ centered at $x_{0}$, its future domain
of dependence is the cone given by $\calD^{+}(B) = \set{(t, x) : t_{0}
  \leq t < t_{0}+r_{0}, 0 \leq \abs{x - x_{0}} < r_{0} - (t -
  t_{0})}$.  The causal future $J^{+}(t_{0}, x_{0})$ and past domain
of dependence $\calD^{-}(B)$ can be defined analogously.

We now state our local geometric uniqueness result.
\begin{proposition}[Local geometric uniqueness at energy
  regularity] \label{prop:locGeom} Let $T_{0} > 0$ and let $B$ be an
  open ball in $\bbR^{4}$. Let $(A, \phi)$, $(A', \phi')$ be
  admissible $C_{t} \calH^{1}$ solutions on the region
  \begin{equation*}
    \calD := \calD^{+}(\set{0} \times B) \cap \big( [0, T_{0}) \times \bbR^{4} \big).
  \end{equation*}
  Suppose that the initial data $(a, e, f, g)$ and $(a', e', f', g')$
  for $(A, \phi)$ and $(A', \phi')$, respectively, are gauge
  equivalent on $B$, i.e., there exists a gauge transformation
  $\underline{\chi} \in \calG^{2}(B)$ such that
  \begin{equation*}
    (a, e, f, g) = (a' - \ud \underline{\chi}, e', e^{i \underline{\chi}} f', e^{i \underline{\chi}} g').
  \end{equation*}
  Then there exists a unique gauge transformation $\chi \in C_{t}
  \calG^{2}(\calD)$ such that $\chi \rst_{\set{0} \times B} =
  \underline{\chi}$ and
  \begin{align*}
    (A, \phi) =& (A' - \ud \chi, e^{i \chi} \phi') \hbox{ on } \calD.
  \end{align*}
\end{proposition}

When the energy is small, this proposition is a rather quick
consequence of Lemma \ref{lem:gt2CoulombId}, the small energy
well-posedness theorem (Theorem \ref{thm:KST}) and the following local
geometric uniqueness for classical solutions.

\begin{lemma}[Local geometric uniqueness of a classical
  solution] \label{lem:locGeom:classical} Let $T_{0} > 0$ and let $B$
  be an open ball in $\bbR^{4}$. Let $(A, \phi)$, $(A', \phi')$ be
  classical solutions on the region $\calD$ as in
  Proposition~\ref{prop:locGeom}. Suppose that the initial data $(a,
  e, f, g)$ and $(a', e', f', g')$ for $(A, \phi)$ and $(A', \phi')$,
  respectively, are gauge equivalent on $B$ by a gauge transformation
  $\underline{\chi} \in \calG^{\infty}(B)$. Then there exists a unique
  gauge transformation $\chi \in C_{t}^{\infty} \calG^{\infty}(\calD)$
  such that $\chi \rst_{\set{0} \times B} = \underline{\chi}$ and $(A,
  \phi) = (A' - \ud \chi, e^{i \chi} \phi')$ on $\calD$.
\end{lemma}
This lemma can be proved by applying a gauge transformation to both
solutions $(A, \phi)$, $(A', \phi')$ so that they have the same
initial data and lie in a gauge where some higher regularity local
well-posedness (hence uniqueness) and the finite speed of propagation
property holds. An example of such a gauge is the \emph{temporal
  gauge} $A_{0} = 0$ \cite{Segal:1979hg, MR649159, MR649158}. We omit
the straightforward details.

Our idea for proving Proposition \ref{prop:locGeom}, which foreshadows
the strategy behind establishing the local well-posedness theorem
(Theorem~\ref{thm:lwp4MKG}) in Section~\ref{sec:lwp}, is essentially
to piece together the aforementioned small energy uniqueness by
exploiting finite speed of propagation. An immediate obstacle is that
Theorem \ref{thm:KST} requires using a non-local gauge (i.e., the
global Coulomb gauge), with respect to which finite speed of
propagation breaks down. To get around this, we will rely on the
excision and gluing techniques developed in Section \ref{sec:gluing}.

\begin{proof}  [Proof of Proposition~\ref{prop:locGeom}]
  \setcounter{claim}{0} For simplicity of the exposition, we will
  assume that $T_{0} \geq 1$, so that $\calD = \calD^{+}(B)$. The
  general case $T_{0} > 0$ can be handled with a little modification
  of the argument below. Given a subset $O \subseteq \bbR^{4}$, we
  will abuse the notation for convenience and use $O$ and $\set{0}
  \times O$ interchangeably.

  \pfstep{Step 1} We begin the proof of Proposition \ref{prop:locGeom}
  by reducing it to the following claim:
  \begin{claim} \label{claim:geomUnique0} Let $\dlt > 0$ and let $B$
    be an open ball in $\bbR^{4}$. Let $(A, \phi)$ and $(A', \phi')$
    be admissible $C_{t} \calH^{1}$ solutions on $\calD$ with gauge
    equivalent initial data on $B$ as in the hypothesis of Proposition
    \ref{prop:locGeom}. Then there exists a unique gauge transform
    $\chi \in C_{t} \calG^{2}(\calD^{+}((1-\dlt)B))$ such that $(A,
    \phi) = (A' - \ud \chi, e^{i \chi} \phi')$ on
    $\calD^{+}((1-\dlt)B)$ and $\chi \rst_{\set{0} \times (1-\dlt) B}
    = \underline{\chi}\rst_{(1-\dlt) B}$.
  \end{claim}

  Indeed, once Claim \ref{claim:geomUnique0} is proved, Proposition
  \ref{prop:locGeom} would immediately follow by taking $\dlt \to
  0$. Note that we have an apriori bound on the gauge transformation
  $\chi$ between $(A, \phi)$ and $(A', \phi')$ in $C_{t}
  \calG^{2}(\calD)$ simply from the fact that $(A, \phi), (A', \phi')
  \in C_{t} \calH^{1} (\calD)$.

  The advantage of establishing Claim \ref{claim:geomUnique0} instead
  of directly proving the proposition is that we have gained an extra
  room $\calD^{+}(B) \setminus \calD^{+}((1-\dlt) B)$, which will
  serve as a `cushion' for performing the excision and gluing
  procedure developed in Section \ref{sec:gluing}.

  \pfstep{Step 2} In this step, we show that Claim
  \ref{claim:geomUnique0} follows from a more local statement, namely
  Claim \ref{claim:geomUnique1} to be stated below.  By translation
  and scaling symmetries, we may assume that $B$ is the unit ball
  $\set{\abs{x} < 1}$ in $\bbR^{4}$.  Let $(A, \phi)$ be an admissible
  $C_{t} \calH^{1}$ solution to \eqref{eq:MKG} on $\calD$, and let
  $(a, e, f, g)$ be its initial data on $\set{0} \times B$.

  We make the following claim:
  \begin{claim} \label{claim:geomUnique1} There exists $0 < \eps \leq
    \frac{1}{1+\dlt}$, which depends only on $(a, e, f, g)$ and $\dlt
    > 0$, such that the following holds: For every ball $B_{\eps}$ of
    radius $\eps$ such that $(1+\dlt) B_{ \eps} \subseteq B$, there
    exists an admissible $C_{t} \calH^{1}$ solution
    $(\check{A}_{[B_{\eps}]}, \check{\phi}_{[B_{\eps}]})$ to
    \eqref{eq:MKG} on $\calD^{+}(B_{\eps})$ such that $(A, \phi)
    \rst_{\calD^{+}(B_{\eps})}$ is gauge equivalent to
    $(\check{A}_{[B_{\eps}]}, \check{\phi}_{[B_{\eps}]})$. Moreover,
    for a fixed $\dlt > 0$, $(\check{A}_{[B_{\eps}]},
    \check{\phi}_{[B_{\eps}]})$ is uniquely determined by $(a, e, f,
    g)$.
  \end{claim}

  In the rest of this step, we give a proof of Claim
  \ref{claim:geomUnique0} assuming Claim \ref{claim:geomUnique1}.  In
  what follows, we will write $B_{\eps}$ to denote a ball of radius
  $\eps$ whose center may vary.

  Let $(a', e', f', g')$ be the initial data set for $(A', \phi')$ on
  $\set{0} \times B$. By hypothesis, there exists $\underline{\chi}
  \in \calG^{2}(B)$ such that
  \begin{equation*}
    (a_{j}, e_{j}, f, g) = (a'_{j} - \rd_{j} \underline{\chi}, e'_{j}, e^{i \underline{\chi}} f', e^{i \underline{\chi}} g').
  \end{equation*}

  We extend $\underline{\chi}$ to $\calD$ by imposing the condition
  $\rd_{t} \underline{\chi} = 0$; abusing the notation a bit, we will
  denote the extension still by $\underline{\chi}$. We then define
  $(A'', \phi'') := (A' - \ud \underline{\chi}, e^{i \underline{\chi}}
  \phi')$. Note that $\underline{\chi} \in C_{t} \calG^{2}(\calD)$,
  $(A'', \phi'') \in C_{t} \calH^{1} (\calD)$, and that the initial
  data for $(A, \phi)$ and $(A'', \phi'')$ coincide on $\set{0} \times
  B$.
  Applying Claim \ref{claim:geomUnique1} to $(A, \phi)$ and $(A'',
  \phi'')$ separately, observe that we obtain the same solution
  $(\check{A}_{[B_{\eps}]}, \check{\phi}_{[B_{\eps}]})$ for each
  $B_{\eps}$ such that $(1+\dlt) B_{\eps} \subseteq B$, because the
  initial data are identical.  Since gauge equivalence is a transitive
  relation, it follows that for every $(1+\dlt) B_{\eps} \subseteq B$,
  there exists $\chi_{[B_{\eps}]} \in C_{t}
  \calG^{2}(\calD^{+}(B_{\eps}))$ such that
  \begin{gather*}
    (A, \phi) = (A'' - \ud \chi_{[B_{\eps}]}, e^{i \chi_{[B_{\eps}]}}
    \phi'') \quad \hbox{ on } \calD^{+}(B_{\eps})
  \end{gather*}
  with $\chi_{[B_{\eps}]} = 0$ on $\set{0} \times B_{\eps}$. Note that
  \begin{equation*}
    \calD^{+}((1-\dlt \eps) B) \cap \big( [0, \eps) \times \bbR^{4} \big) = \bigcup_{B_{\eps} \subseteq B} \calD^{+}(B_{\eps}).
  \end{equation*}
  Since $\rd_{t} \chi_{[B_{\eps}]} = A''_{0} - A_{0}$ for each
  $B_{\eps}$, we deduce that there exists a gauge transform $\chi'$ on
  $\calD^{+}((1-\dlt \eps)B) \cap ([0, \eps) \times \bbR^{4})$ that
  coincides with each $\chi_{[B_{\eps}]}$ on $\calD^{+}(B_{\eps})$ and
  thus
  \begin{gather*}
    (A, \phi) = (A'' - \ud \chi', e^{i \chi'} \phi'') \quad \hbox{ on
    } \calD^{+}((1 - \dlt \eps) B) \cap \big( [0, \eps) \times
    \bbR^{4} \big).
  \end{gather*}
  Note also that $\chi' \in C_{t} \calG^{2}(\calD^{+}((1-\dlt \eps) B)
  \cap ([0, \eps) \times \bbR^{4}))$, since $(A, \phi)$ and $(A'',
  \phi'')$ are in $C_{t} \calH^{1}$. Moreover, we have $\chi'
  \rst_{\set{0} \times (1-\dlt \eps) B} = 0$, since each
  $\chi_{[B_{\eps}]}$ equals $0$ on $\set{0} \times
  B_{\eps}$. Defining $\chi = \chi' + \underline{\chi}$ on
  $\calD^{+}((1 - \dlt \eps) B) \cap ([0, \eps) \times \bbR^{4})$, it
  follows that
  \begin{equation} \label{eq:geomUnique:pf:step2} (A, \phi) = (A' -
    \ud \chi, e^{i \chi} \phi') \quad \hbox{ on } \calD^{+}((1 - \dlt
    \eps) B) \cap \big( [0, \eps) \times \bbR^{4} \big).
  \end{equation}
  and $\chi \rst_{\set{0} \times (1-\dlt \eps) B} = \underline{\chi}$.

  We now conclude with a continuity argument. Consider the set
  \begin{align*}
    \calT = \set{& T \in [0, 1] : 
       \, \exists \chi \in C_{t} \calG^{2} \hbox{ s.t. }(A, \phi) = (A' - \ud \chi, e^{i \chi} \phi') \hbox{ on } \calD^{+}((1-\dlt T) B) \cap \big( [0, T] \times \bbR^{4} \big) \\
      & \quad \hbox{ and } \chi \rst_{\set{0} \times (1-\dlt T) B} =
      \underline{\chi} \, }.
  \end{align*}
  Clearly $T$ is an interval containing $0$.  We claim that $\sup
  \calT = 1$. Indeed, by continuity of $(A, \phi)$ and $(A', \phi')$,
  we have $\sup \calT \in \calT$ so $\calT$ is closed.  On the other
  hand, if $T <1$ is in $\calT$, then by
  \eqref{eq:geomUnique:pf:step2} (suitably rescaled), we see that
  there exists some $\eps > 0$ such that $T+ \eps \in \calT$. Thus 
$\calT$ is open in $[0,1]$. As it is both open and closed, we must 
have $\calT = [0,1]$.   Claim
  \ref{claim:geomUnique0} now follows.

  \pfstep{Step 3. Proof of Claim \ref{claim:geomUnique1}} To finish
  the proof, it remains to establish Claim
  \ref{claim:geomUnique1}. The key ingredients are the local geometric
  uniqueness statement for classical solutions, Theorem \ref{thm:KST}
  and the excision and gluing techniques in Section \ref{sec:gluing}.

  Fix $\rglue := 1+\dlt$ and $\rgExt = 1+\dlt/2$. We select $\eps > 0$ so that for every
  $\rglue B_{\eps} \subseteq B$ we have
  \begin{equation}
    \nrm{(a, e, f, g)}_{\calH^{1}(\rglue B_{\eps} )}^{2} < \frac{1}{10 C_{1}^{2}} \thE^{2}.
  \end{equation}
  where $C_{1} = C_{1}(\rglue, \rgExt) \geq 1$ is the implicit constant from
  \eqref{eq:gluing:est} in Proposition \ref{prop:gluing}. Since $(a,
  e, f, g) \in \calH^{1}(B)$ and $\rglue B_{\eps} \subseteq B$, it is
  not difficult to see that a non-zero choice of $\eps$ is always
  possible, and it depends only on $\dlt > 0$ (through $\rglue =
  1+\dlt$) and $(a,e,f,g)$ on $B$.

  Next, by the definition of an admissible solution, there exists a
  sequence $(A^{(n)}, \phi^{(n)})$ of classical solutions on $\calD$
  which converges to $(A, \phi)$ in the $C_{t} \calH^{1}$
  norm. Denoting their initial data on $\set{0} \times B$ by
  $(a^{(n)}, e^{(n)}, f^{(n)}, g^{(n)})$, we may assume (by throwing
  away finitely many terms) that
  \begin{equation}
    \nrm{(a^{(n)}, e^{(n)}, f^{(n)}, g^{(n)})}_{\calH^{1}(\rglue B_{\eps})}^{2} 
< \frac{1}{10 C_{1}^{2}} \thE^{2} \quad \hbox{ for all } n \in \bbZ_{+}.
  \end{equation}

  Now we apply Proposition \ref{prop:gluing} to $(a^{(n)}, e^{(n)},
  f^{(n)}, g^{(n)})$ [resp. $(a, e, f, g)$] on $\rglue B_{\eps}
  \setminus \overline{B_{\eps}}$, from which we obtain an initial data
  set $(\widetilde{a}^{(n)}, \widetilde{e}^{(n)}, \widetilde{f}^{(n)},
  \widetilde{g}^{(n)})$ [resp. $(\widetilde{a}, \widetilde{e},
  \widetilde{f}, \widetilde{g})$] on $\bbR^{4}$ such that
  \begin{equation} \label{eq:geomUnique:properties4id}
    \calE[\widetilde{a}^{(n)}, \widetilde{e}^{(n)},
    \widetilde{f}^{(n)}, \widetilde{g}^{(n)}] < \frac{1}{2} \thE^{2},
    \quad (\widetilde{a}^{(n)}, \widetilde{e}^{(n)},
    \widetilde{f}^{(n)}, \widetilde{g}^{(n)}) \to (\widetilde{a},
    \widetilde{e}, \widetilde{f}, \widetilde{g}) \hbox{ in }
    \calH^{1}(\bbR^{4}).
  \end{equation}
  Applying Lemma \ref{lem:gt2CoulombId} and imposing some condition to
  fix the constant gauge transformation ambiguity (e.g., requiring the
  integral of the gauge transformation on $B_{\eps}$ to vanish), we arrive
  at a globally Coulomb initial data set $(\check{a}^{(n)},
  \check{e}^{(n)}, \check{f}^{(n)}, \check{g}^{(n)})$
  [resp. $(\check{a}, \check{e}, \check{f}, \check{g})$] which is
  gauge equivalent to $(\widetilde{a}^{(n)}, \widetilde{e}^{(n)},
  \widetilde{f}^{(n)}, \widetilde{g}^{(n)})$ [resp. $(\widetilde{a},
  \widetilde{e}, \widetilde{f}, \widetilde{g})$] and satisfies
  \eqref{eq:geomUnique:properties4id}. Then by Theorem \ref{thm:KST},
  there exists a sequence of global classical solutions
  $(\check{A}^{(n)}, \check{\phi}^{(n)})$ with initial data
  $(\check{a}^{(n)}, \check{e}^{(n)}, \check{f}^{(n)},
  \check{g}^{(n)})$ in the global Coulomb gauge, which converges in
  $S^{1} \subseteq C_{t} \calH^{1}$ locally in time to a solution
  $(\check{A}, \check{\phi})$ with initial data $(\check{a},
  \check{e}, \check{f}, \check{g})$.

  Observe that $(\check{a}^{(n)}, \check{e}^{(n)}, \check{f}^{(n)},
  \check{g}^{(n)}) \rst_{B_{\eps}}$ is gauge equivalent to $(a^{(n)},
  e^{(n)}, f^{(n)}, g^{(n)}) \rst_{B_{\eps}}$ by construction. By
  classical geometric well-posedness, it follows that
  $(\check{A}^{(n)}, \check{\phi}^{(n)}) \rst_{\calD^{+}(B_{\eps})}$
  is gauge equivalent to $(A^{(n)}, \phi^{(n)})
  \rst_{\calD^{+}(B_{\eps})}$ for each $n$. As $(A^{(n)}, \phi^{(n)})
  \to (A, \phi)$ and $(\check{A}^{(n)}, \check{\phi}^{(n)}) \to
  (\check{A}, \check{\phi})$ in $C_{t}
  \calH^{1}(\calD^{+}(B_{\eps}))$, we can take the limit of the gauge
  transformations and conclude that there exists $\chi \in C_{t}
  \calG^{2}(\calD^{+}(B_{\eps}))$ such that
  \begin{equation*}
    (A, \phi) = (\check{A} - \ud \chi, e^{i \chi} \check{\phi}) \quad \hbox{ on } \calD^{+}(B_{\eps}).
  \end{equation*}
  Defining $(\check{A}_{[B_{\eps}]}, \check{\phi}_{[B_{\eps}]}) :=
  (\check{A}, \check{\phi}) \rst_{\calD^{+}(B_{\eps})}$, Claim
  \ref{claim:geomUnique1} follows. \qedhere
\end{proof}


%
%
%

\section{Finite energy local well-posedness in global Coulomb
  gauge} \label{sec:lwp} The purpose of this section is to establish
local well-posedness of the $(4+1)$-dimensional \eqref{eq:MKG} for
finite energy Coulomb initial data in the class of admissible
solutions in the \emph{global Coulomb gauge} (to be defined precisely
below). As the energy regularity is critical respect to the scaling
property of \eqref{eq:MKG}, the lifespan of the solution \emph{cannot}
depend only on the size of the initial energy. However, given an
initial data $(a,e,f,g)$ with $\calE[a,e,f,g] \leq \En$, we shall
prove a lower bound on the lifespan that is proportional to the
\emph{energy concentration scale} $\ecs$ of the initial
data, defined as
\begin{equation} \label{eq:ecs}
  \begin{aligned}
    \ecs[a, e, f, g]
    :=& \sup \set{ r \geq 0 : \forall x \in \bbR^{4}, \
      \calE_{B_{r}(x)}[a, e, f, g] < \dlt_{0}(\En, \thE^{2})},
  \end{aligned}\end{equation}
where $\dlt_{0} (\En,\thE^{2}) > 0$ is a fixed function to be determined below (see Proposition~\ref{p:cubes})
and $\thE^{2}$ is the threshold energy for small data global
well-posedness (Theorem \ref{thm:KST}). Note that for any choice of
$\dlt_{0}$ and $(a, e, f, g) \in \calH^{1}(\bbR^{4})$, we always have
$\ecs[a, e, f, g] > 0$. 


We define the \emph{energy profile} $\rho$ of $(a, e, f, g)$ to be
\begin{equation} \label{eq:en-prof}
	\rho(x) = \rho[a, e, f, g](x) := \frac{1}{2} ( \abs{\ud a}^{2} + \abs{e}^{2} + \abs{\covD f}^{2} + \abs{g}^{2} )(x),
\end{equation}
so that $\int_{S} \rho \, \ud x = \calE_{S}[a, e, f, g]$ for any measurable set $S \subseteq \bbR^{4}$.
We say that an admissible
$C_{t} \calH^{1}$ solution $(A_{\mu}, \phi)$ on a time interval $I \times \bbR^{4}$
obeys the \emph{global Coulomb gauge condition} if
\begin{equation} \label{eq:globalCoulomb} \rd^{\ell} A_{\ell} = 0
  \quad \hbox{ on } I \times \bbR^{4}.
\end{equation}

The precise statement of our local well-posedness theorem in global
Coulomb gauge is as follows.

\begin{theorem}[Local well-posedness of \eqref{eq:MKG} at energy
  regularity, complete version] \label{thm:lwp4MKG}Let $(a, e, f, g)$
  be an $\calH^{1}(\bbR^{4})$ initial data set satisfying the global
  Coulomb condition $\rd^{\ell} a_{\ell} =0$ with energy
  $\calE[a,e,f,g] \leq \En$. Let $\ecs = \ecs[a, e, f, g]$ be
  defined as in \eqref{eq:ecs}. Then the following statements hold.
  \begin{enumerate}
  \item \label{item:lwp4MKG:exist} There exists a unique
    $C_{t} \calH^{1}$ admissible solution $(A, \phi)$ to
    \eqref{eq:MKG} on $[-\ecs, \ecs] \times \bbR^{4}$ with $(a, e,
    f, g)$ as its data at $t = 0$, which obeys the global Coulomb
    gauge condition \eqref{eq:globalCoulomb}. 

  \item \label{item:lwp4MKG:aprioriEst} We have the additional regularity 
properties
    \begin{equation} \label{eq:lwp4MKG:aprioriEst}
      A_{0} \in Y^{1}([-\ecs, \ecs] \times \bbR^{4}), \quad
      A_{x}, \phi \in S^{1}([-\ecs, \ecs] \times \bbR^{4}),
    \end{equation}  
    with bounds depending only on the energy profile $\rho$, where the spaces $Y^{1}$ and $S^{1}$ will be defined in
    Section~\ref{subsec:ftnspace4patching} below.
  \item \label{item:lwp4MKG:regPersists} The solution $(A, \phi)$ is more regular if the initial data set $(a, e, f, g)$ is. In particular, $(A, \phi)$ is
    classical if $(a, e, f, g)$ is a classical initial data set.
  \item \label{item:lwp4MKG:contdep} Consider a sequence $(a^{(n)},
    e^{(n)}, f^{(n)}, g^{(n)})$ of $\calH^{1}$ globally Coulomb
    initial data sets such that $(a^{(n)}, e^{(n)}, f^{(n)}, g^{(n)})
    \to (a, e, f, g)$ in $\calH^{1}(\bbR^{4})$ as $n \to
    \infty$. Denote the corresponding solutions to \eqref{eq:MKG}
    given by Statement (\ref{item:lwp4MKG:exist}) by $(A^{(n)},
    \phi^{(n)})$. Then the lifespan of $(A^{(n)}, \phi^{(n)})$
    eventually contains $[-\ecs, \ecs]$. Moreover, we have
    \begin{equation} \label{eq:lwp4MKG:contdep} \nrm{A_{0} -
        A^{(n)}_{0}}_{Y^{1}[-\ecs, \ecs]} +
      \nrm{A_{x} - A^{(n)}_{x}}_{S^{1}[-\ecs, \ecs]} 
      + \nrm{\phi - \phi^{(n)}}_{S^{1}[-\ecs, \ecs]} \to 0
    \end{equation}
    as $n \to \infty$.
  \end{enumerate}
\end{theorem}

\begin{remark} 
In fact our proof below yields an a-priori bound for the $S^{1}$ norm 
of $(A_{x}, \phi)$ and the $Y^{1}$ norm of $A_{0}$ that depends only on 
the energy $E$, the energy concentration scale $r_{c}$ and the tail of 
the energy profile $\rho$, i.e., the smallest radius $r_{0} > 0$ such that there 
exists $x_{0} \in \bbR^{4}$ satisfying
\begin{equation} \label{eq:energy-tail}
	\int_{\bbR^{4} \setminus B_{\frac{1}{54} r_{0}}(x_{0})} \rho \, \ud x < \dlt_{0}(E, \thE^{2}).
\end{equation}
We refer to Remark~\ref{rem:apriori-bnd} for a further discussion.
\end{remark}

As mentioned in the introduction, a theorem of this type is usually
proved by exploiting finite speed of propagation, patching together
local solutions with small initial data. However, while implementing
this strategy in our context, one is faced with difficulties due to
non-local features of \eqref{eq:MKG}. One source of non-locality is
the presence of the Gauss (or constraint) equation; another is the
elliptic nature of the global Coulomb gauge.  To address the first
issue, we use the technique of excision and gluing initial data sets
developed in Section \ref{sec:gluing}. To deal with the second issue,
we introduce a procedure for patching rough local solutions together
to produce a local-in-time but \emph{global-in-space} solution,
inspired by similar ideas in elliptic gauge theories.

The rest of this section is structured as follows. In
Section~\ref{subsec:unique}, the uniqueness statement of Theorem
\ref{thm:lwp4MKG} is established using the local geometric uniqueness
result proved in Section~\ref{sec:locGeom}.  In
Section~\ref{sec:split} we consider the question of partitioning the
initial surface $\bbR^4$ into regions which carry a small energy.
Section~\ref{subsec:ftnspace4patching}, we introduce the function
space framework for patching up local \eqref{eq:MKG} solutions. Using
this framework, we establish Proposition \ref{prop:patch} in
Section~\ref{subsec:patch}, which is an abstract statement that
contains the essence of our patching argument. Finally, in
Section~\ref{subsec:lwp4MKG:pf}, we put together the tools developed
in the previous subsections to prove Theorem \ref{thm:lwp4MKG}.

\subsection{Uniqueness in the global Coulomb
  gauge} \label{subsec:unique} In this brief subsection, we prove the
uniqueness statement in Theorem~\ref{thm:lwp4MKG} (i.e., uniqueness of
an admissible $C_{t} \calH^{1}(I \times \bbR^{4})$ solution in the
global Coulomb gauge) using Proposition~\ref{prop:locGeom}.

Patching together Proposition \ref{prop:locGeom} on balls covering
$\bbR^{4}$, it follows that two admissible $C_{t} \calH^{1}$ solutions
$(A, \phi)$ and $(A', \phi')$ on $[0, T_{0}) \times \bbR^{4}$ are
gauge equivalent if their initial data sets are gauge equivalent. We
then make the following observation:

\begin{lemma} \label{lem:coulombUnique} Let $I \subset \bbR$ be an
  open interval. Let $(A_{\mu}, \phi)$ and $(A'_{\mu}, \phi')$ be
  admissible $C_{t} \calH^{1}$ solutions on $I \times \bbR^{4}$, which
  are gauge equivalent and obey the global Coulomb gauge condition
  \eqref{eq:globalCoulomb}. Then there exists a constant $\chi_{0} \in
  \bbR$ such that $(A'_{\mu}, \phi') \equiv (A_{\mu}, \phi e^{ i
    \chi_{0}})$ on $I \times \bbR^{4}$.
\end{lemma}

\begin{proof} 
  Note that in the global Coulomb gauge, $A \in C_{t}^{0}
  \dot{H}^{1}_{x}$ is determined uniquely from $\rd^{\ell} A_{\ell} =
  0$, $\ud A = F$ and \eqref{eq:MKG}. This observation fixes the gauge
  transformation $\chi$ between $(A, \phi)$ and $(A', \phi')$ up to a
  constant, at which point we are done. \qedhere.
\end{proof}

Therefore, to complete the proof of Theorem~\ref{thm:lwp4MKG}, it
suffices to prove the local existence, persistence of regularity and
continuous dependence on the initial data.

\subsection{Energy concentrations scales}\label{sec:split}

Here we consider the energy distribution of initial data $(a,e,f,g)$
in the global Coulomb gauge, and show that we can cover $\bbR^4$ with
small energy cubes with side length bounded from below by $4 \ecs$.
We also ensure that the covering is slowly varying, in the sense that neighboring cubes 
have comparable side lengths. This condition is needed for an effective control of the
constants in the patching procedure in Section~\ref{subsec:patch}.
The number of such cubes, which we denote by $K$, can be trivially bounded by 
$(r_{0} / \ecs)^{4}$, where $r_{0}$ is defined by the condition \eqref{eq:energy-tail}; this number
will enter in the final a-priori $S^{1}$ regularity bound in \eqref{eq:lwp4MKG:aprioriEst}.
As a part of our analysis here, we also specify the constant $\delta_0(E, \thE^{2})$ in
\eqref{eq:ecs}. See Proposition~\ref{p:cubes} below for a more precise statement.

We begin with a preliminary result, which shows that for Coulomb data the 
energy controls the full $\calH^1$ norm:

\begin{proposition}
Let $(a,e,f,g) \in \calH^1(\bbR^4)$ be a Coulomb initial data set with energy $E$. Then
we have the bound
\begin{equation} \label{eq:id-en-H1}
\| (a,e,f,g)\|_{\calH^1(\bbR^4)}^2 \lesssim E+E^2.
\end{equation}
\end{proposition}
\begin{proof}
  We need to obtain bounds for $A$ and $f$ in $\dot H^1_{x}$.  We begin
  with $a$, where the Coulomb condition $\nabla \cdot a = 0$ allows us
  to estimate in linear elliptic fashion
\[
\| a\|_{\dot H^1_{x}} \lesssim \|\ud a\|_{L^2_{x}} \lesssim E^\frac12.
\]

For $f$ we first use the diamagnetic inequality and Sobolev embeddings to obtain
\[
\| f\|_{L^4_{x}} \lesssim \| \nabla |f|\|_{L^2_{x}} \lesssim \|\bfD f\|_{L^2_{x}}  \lesssim E^\frac12
\]
and then, splitting the covariant derivative, 
\[
\| \nabla f\|_{L^2_{x}} \lesssim \|\bfD f\|_{L^2_{x}} + \|f\|_{L^4_{x}} \|a\|_{L^4_{x}} \lesssim E^\frac12+E,
\]
which completes the proof. \qedhere
\end{proof}

Next, we give an improvement of Hardy's inequality 
\begin{equation}\label{eq:Hardy}
\| |x-x_0|^{-1} f\|_{L^2_{x}} \lesssim\| \nabla | f|\|_{L^2_{x}} \leq  \| \bfD f\|_{L^2_{x}}, 
\end{equation}
which is our tool for obtaining smallness of the weighted $L^{2}$ norm in \eqref{eq:gluing:energy}.
We state a general version on $\bbR^{d}$.
\begin{lemma}[Improved Hardy's inequality] \label{lem:simple-hardy+}
Let $a_{j}, f \in \dot{H}^{1}(\bbR^{d})$ where $d \geq 3$. Then for any ball $B = B_{r_{0}}(x_{0})$ and $\sgm_{0} \geq 2$, we have the bounds
\begin{equation}\label{eq:simple-hardy+out}
\| \frac{1}{\abs{x - x_{0}}} f \|_{L^2_{x}(2B \setminus \overline{B})} \lesssim \| \bfD f \|_{L^2_{x}(\sgm_{0} B \setminus \overline{B})}
+ \sgm_{0}^{-\frac{d-2}{2}} \|\bfD f \|_{L^2_{x}(\bbR^{d}\setminus \overline{\sgm_{0} B})},
\end{equation}
\begin{equation}\label{eq:simple-hardy+}
 \| \frac{1}{\abs{x - x_{0}}} f \|_{L^2_{x}(2B)} \lesssim \| \bfD f \|_{L^2_{x}(\sgm_{0} B)}
+ \sgm_{0}^{-\frac{d-2}{2}} \|\bfD f \|_{L^2_{x}(\bbR^{d}\setminus \overline{\sgm_{0} B})}.
\end{equation}
\end{lemma}
\begin{remark} 
In this paper, we only use the inequality \eqref{eq:simple-hardy+} on balls. The version \eqref{eq:simple-hardy+out} will be useful in the third paper \cite{OT3} of the series.
\end{remark}
\begin{proof} 
By translation and scaling, we may assume that $B = B_{1}(0)$. 
We begin by splitting $g := \abs{f}$ into spherical harmonics. In the case of non-spherically-symmetric modes, by Poincar\'e's inequality on spheres and the diamagnetic inequality, we have
\begin{equation*}
	\nrm{\frac{1}{\abs{x}} g}_{L^{2}_{x}(2 B \setminus \overline{B})}
	\aleq \nrm{\angnb \abs{f}}_{L^{2}_{x}(2 B \setminus \overline{B})}
	\aleq \nrm{\covD f}_{L^{2}_{x}(2 B \setminus \overline{B})},
\end{equation*}
where $\abs{\angnb \abs{f}}$ denotes the size of the angular derivatives under the induced metric on the sphere $\set{\abs{x} = const}$. Hence we are reduced to the case when $g$ is radial. 

By the one-dimensional Hardy inequality, we have
\begin{equation*}
	\sgm_{0}^{\frac{d-2}{2}} \abs{g(\sgm_{0})} \aleq \nrm{r^{\frac{d-1}{2}} g'}_{L^{2}(\sgm_{0}, \infty)} \aleq \nrm{\covD f}_{L^{2}_{x}(\bbR^{d} \setminus \overline{\sgm_{0} B})}.
\end{equation*}
Moreover, by the fundamental theorem of calculus and the diamagnetic inequality, we have the one-dimensional dyadic bounds
\begin{equation} \label{eq:H+:dyadic}
	\sgm^{-\frac{1}{2}} \sup_{\frac{1}{2} \sgm \leq r, r' \leq \sgm} \abs{g(r) - g(r')} \aleq \nrm{g'}_{L^{2}(\frac{1}{2} \sgm, \sgm)} \aleq \sgm^{- \frac{d-1}{2}} \nrm{\covD f}_{L^{2}_{x}(\sgm B \setminus \frac{1}{2} \sgm \overline{B})} \quad \hbox{ for all } \sgm \geq 2.
\end{equation}
Then by summing up the dyadic bounds for $2 \leq \sgm \aleq \sgm_{0}$, we then obtain the $L^{\infty}$ bound
\begin{equation} \label{eq:H+:Linfty}
	\nrm{g}_{L^{\infty}(1, 2)} \aleq \nrm{\covD f}_{L^{2}_{x}(\sgm B \setminus \overline{B})} + \sgm_{0}^{- \frac{d-2}{2}} \nrm{\covD f}_{L^{2}_{x}(\bbR^{d} \setminus \overline{\sgm_{0} B})}.
\end{equation}
Applying H\"older's inequality, the desired estimate \eqref{eq:simple-hardy+out} follows.

We now turn to the bound \eqref{eq:simple-hardy+} on the full ball $2 B$. Again splitting $g = \abs{f}$ into spherical harmonics, we are reduced to the case of a radial function $g$. But in this case we have the one-dimensional Hardy inequality
\begin{equation} \label{eq:H+:hardy-near-0}
	\nrm{r^{\frac{d-3}{2}} g}_{L^{2}(0, 1)}
	\aleq \nrm{r^{\frac{d-1}{2}} g'}_{L^{2}(0, 1)} + \abs{g(1)}
	\aleq \nrm{\covD f}_{L^{2}_{x}(B)} + \abs{g(1)}.
\end{equation}
Combined with \eqref{eq:H+:Linfty}, the desired inequality \eqref{eq:simple-hardy+} follows. \qedhere
\end{proof}

We are now ready to state and prove the main covering result of this section.
We also settle the choice of $\dlt_{0}(\En, \thE^{2})$.

\begin{proposition} \label{p:cubes}
Assume that $\dlt_{0}(\En, \thE^{2})$ is chosen so that
\begin{equation}\label{eq:delta-choose} 
\dlt_{0}(\En, \thE^{2}) = c^2  \thE^{2} \min \set{1, \thE^{4} E^{-2}},
\end{equation}
with a small universal constant $c$. Let  $r_{0}$ and $x_{0}$ be as in \eqref{eq:energy-tail}.
Then there exists a dyadic cube $R_0$ of side length $\aeq r_{0}$ and a partition of it into smaller dyadic cubes
\[
R_0 = \bigcup_{\alpha \in \calA} R_\alpha
\]
with the following properties:

\begin{enumerate}
\item Small energy: The following bound holds for $Q = 18 R_\alpha$
  and $Q = (\frac{1}{18} R_0)^c$:
\begin{equation} \label{eq:loc-enR}
\calE_{Q}[a,e,f,g] + \frac{1}{\ell(Q)^{2}} \|f\|_{L^2_{x}(Q)}^2 \ll  \thE^{2},
\end{equation}
where we use the convention that $\ell(Q) = \ell(\frac{1}{18} R_{0})$ when $Q = (\frac{1}{18} R_{0})^{c}$.
\item Size of cubes: The side length of the cubes $\set{R_{\alp}}$ is bounded from below by $4 \ecs$.
\item Number of cubes: The number of cubes $\set{R_{\alp}}$ is bounded by $ K := |\calA| \lesssim (r_{0} / \ecs)^{4}$.

\item Slow variance: The size of all pairs of neighboring cubes may
  differ at most by a factor of $2$, and all cubes adjacent to the boundary of $R_0$
have size at most $\ell(R_{0}) / 64$.
\end{enumerate}
\end{proposition}

\begin{proof}
Let $r_{0}$ and $x_{0}$ be as in \eqref{eq:energy-tail}. 
It suffices to consider the case $E > \thE^{2}$, since the proposition is trivial in the other case.
We may also assume that $r_{0} \geq 200 \, \ecs$, as otherwise we can simply choose $\set{R_{\alp}} = \set{R_{0}}$ where $R_{0}$ is the cube of side length $2 r_{0}$ centered at $x_{0}$. By translation and scaling, we may henceforth take $x_{0} = 0$ and $\ecs = 1$. 

We choose the large cube $R_{0}$ centered at $0$ so that $B_{r_{0}}(0) \subseteq R_{0} \subseteq 3 B_{r_{0}}(0)$ and $\ell(R_{0}) \in 2^{\bbZ}$. This cube will set the coordinates for our dyadic grid; more precisely, subsequent cubes will be obtained by repeatedly subdividing the sides of $R_{0}$ in half.
To ensure slow variance, we use the following procedure to construct the collection $\calR := \set{R_{\alp}}_{\alp \in \calA}$:
\begin{itemize}
\item In the first step, we add to the collection $\calR$ the cubes of side length $\frac{1}{64} \ell(R_{0})$ adjacent to $R_{0}$; 
\item Then we recursively add to the collection $\calR$ the cubes which are disjoint from but adjacent to the existing collection, with half the side length of the cubes added in the previous step;
\item We repeat this process until we arrive at cubes of side length $\frac{1}{4}$. Then we cover the rest of $R_{0}$ with dyadic cubes of side length $\frac{1}{4}$.
\end{itemize}
We call $R_{0}$ the \emph{initial cube}, the cubes of side length between $\frac{1}{2}$ and $\frac{1}{64} \ell(R_{0})$ the \emph{intermediate cubes}, 
and the cubes of side length $\frac{1}{4}$ the \emph{final cubes}. Note that all intermediate cubes are contained in $R_{0} \setminus (\frac{15}{16} R_{0})^{c}$.

From the construction, it is obvious to see that Properties (2), (3) and (4) hold. The condition \eqref{eq:loc-enR} clearly holds for the initial cube $Q = (\frac{1}{18} R_{0})^{c}$, by \eqref{eq:energy-tail} and the localized Hardy's inequality
\begin{equation} \label{eq:cubes:ext-hardy}
\nrm{\frac{1}{\abs{x}} f}_{L^{2}_{x}(\bbR^{4} \setminus \frac{1}{54}  B_{r_{0}})}^{2} \aleq \nrm{\covD f}_{L^{2}_{x}(\bbR^{4} \setminus \frac{1}{54}  B_{r_{0}})}^{2} < \dlt_{0}(E, \thE^{2}) = c \thE^{4} E^{-1}.
\end{equation}

Moreover, we claim that the final cubes also satisfy \eqref{eq:loc-enR}. 
Indeed, the energy term $\calE_{Q}$ in \eqref{eq:loc-enR} follows from the definition of $\ecs$. To control the weighted $L^{2}_{x}$ norm, we apply Lemma~\ref{lem:simple-hardy+} and use the fact that we scaled $\ecs = 1$ to obtain
\begin{equation*}
	\frac{1}{\ell(18 R_{\alp})^{2}} \nrm{f}_{L^{2}_{x}(18 R_{\alp})}^{2}
	\aleq \sgm_{0}^{4} \dlt_{0}(\En, \thE^{2}) + \sgm_{0}^{-2} \En
\end{equation*}
Then choosing $\dlt_{0}$ as in \eqref{eq:delta-choose} and $\sgm_{0}^{2} = c^{-2}_{0} \thE^{-2} E$ for some small universal constant $c_{0} >0$, the desired estimate \eqref{eq:loc-enR} follows.

Finally, for the intermediate cubes $R_{\alp} \in \calR$ of side length between $\frac{1}{4}$ and $\frac{1}{64}\ell(R_{0})$, the smallness for the energy $\calE_{Q}$ in \eqref{eq:loc-enR} follows immediately from \eqref{eq:energy-tail}. Hence it only remains to justify the weighted $L^{2}_{x}$ bound in \eqref{eq:energy-tail} for these intermediate cubes. We split our argument into two cases:

(a) When $\ell(R_{\alp}) \geq \frac{1}{100} \sgm_{0}^{-1} \ell(R_{0})$, we use \eqref{eq:cubes:ext-hardy} to estimate
\begin{equation*}
	\frac{1}{\ell(18 R_{\alp})^{2}} \nrm{f}_{L^{2}_{x}(18 R_{\alp})}^{2} 
	\aleq \frac{\sgm_{0}^{2}}{\ell(R_{0})^{2}} \nrm{f}_{L^{2}_{x}(18 R_{\alp})}^{2} 
	\aleq (c / c_{0})^{2} \thE^{4} E^{-1},
\end{equation*}
which is good.

(b) When $\ell(R_{\alp}) < \frac{1}{100} \sgm_{0}^{-1} \ell(R_{0})$, observe that we have 
\begin{equation*}
	18 R_{\alp} \subseteq B_{\alp} \subseteq \sgm_{0} B_{\alp} \subseteq \bbR^{4} \setminus \frac{1}{54} B_{r_{0}}
\end{equation*}
where $B_{\alp}$ is the ball of radius $\ell(18 R_{\alp})$ with the same center as $R_{\alp}$. This chain of inclusions is a consequence of the fact that all intermediate cubes belong to $R_{0} \setminus (\frac{15}{16} R_{0})^{c}$, which are all at distance at least $\frac{1}{4} \ell(R_{0})$ from $\frac{1}{54} B_{r_{0}}$. Therefore, by \eqref{eq:energy-tail} and application of Lemma~\ref{lem:simple-hardy+}, we obtain
\begin{equation*}
	\frac{1}{\ell(18 R_{\alp})^{2}} \nrm{f}_{L^{2}_{x}(18 R_{\alp})}^{2} 
	\aleq \frac{1}{\ell(18 R_{\alp})^{2}} \nrm{f}_{L^{2}_{x}(B_{\alp})}^{2} 
	\aleq \dlt_{0}(E, \thE^{2}) + \sgm_{0}^{-2} E,
\end{equation*}
which implies the desired bound. \qedhere
\end{proof}

\subsection{Functions spaces and gauge transformation
  estimates} \label{subsec:ftnspace4patching}  In this
subsection we introduce the function spaces that will be used in the
proof of existence of finite energy solutions to \eqref{eq:MKG} in the
Coulomb gauge.

The first two such spaces are the spaces $Y^1$ and $S^1$, which were
used in \cite{Krieger:2012vj} to control the elliptic component (i.e.,
$A_{0}$), respectively the hyperbolic components (i.e. $A_x$ and
$\phi$) of small energy solutions in the global Coulomb gauge.
These functions spaces are defined in \cite{Krieger:2012vj} in the 
 whole space-time  $\bbR^{1+4}$.


We start with the space $Y^{s}$, which was used in
\cite{Krieger:2012vj} to control the elliptic component (i.e.,
$A_{0}$) of a solution to \eqref{eq:MKG} in the global Coulomb
gauge. Let $s$ be a non-negative integer and $q \in [1,
\infty]$. Given a tempered distribution $\varphi$ on $\bbR^{1+4}$, we
define its $Y^{s,q}$ norm to be
\begin{equation*}
  \nrm{\varphi}_{Y^{s, q}(\bbR^{1+4})} := 
\nrm{\rd_{t,x}^{s} \varphi}_{L^{q}_{t} \dot{H}^{1/q}_{x}(\bbR^{1+4})} 
\end{equation*}
where we take $L^{q}_{t} \dot{H}^{1/q}_{x} = L^{\infty}_{t} L^{2}_{x}$
when $q = \infty$. Then the $Y^{s}$ space is defined as the space
of tempered distributions for which the following norm is finite:
\begin{equation*}
  \nrm{\varphi}_{Y^{s}(\bbR^{1+4})} := \nrm{\varphi}_{Y^{s, 2}(\bbR^{1+4})} + \nrm{\varphi}_{Y^{s, \infty}(\bbR^{1+4})}
\end{equation*}
Observe that $\nrm{\cdot}_{Y^{0,q}(\bbR^{1+4})}$ (and thus
$\nrm{\cdot}_{Y^{0}(\bbR^{1+4})}$) scales the same way as the
$L^{\infty}_{t} L^{2}_{x}$ norm. In particular,
$\nrm{\cdot}_{Y^{1}(\bbR^{1+4})}$ scales like the $L^{\infty}_{t}
\dot{H}^{1}_{x}$ norm.

Next, we introduce the $S^{1}$ norm on $\bbR^{1+4}$, which was used in
\cite{Krieger:2012vj} to measure the size of the hyperbolic components
(i.e., $A_{x}$ and $\phi$) of solutions to \eqref{eq:MKG} in the
global Coulomb gauge. The precise definition of this norm involves
\emph{null frame spaces} \cite{MR1827277, Tao:2001gb}, and is rather
technical to state. The fine structure of this norm, though crucial
for establishing the small data theory of \eqref{eq:MKG} at the energy
regularity, is not necessary for the purpose of the present
section. Hence, here we will be content with simply stating the
necessary properties of the $S^{1}$ norm; the rigorous definition of
$S^{1}$ will be recalled from \cite{Krieger:2012vj} in
Section~\ref{sec:gtCutoff}, where the proof of these properties will
be given.

We begin by introducing the norms $X^{s, b}_{r}$ and $\underline{X}$
(where $s, b\in \bbR$, $1 \leq r \leq \infty$), defined by
\begin{align}
  \nrm{\varphi}_{X^{s, b}_{k; r}(\bbR^{1+4})} :=& 2^{s k} \bb( \sum_{j} (2^{b j} \nrm{Q_{j} \varphi}_{L^{2}_{t,x}(\bbR^{1+4})})^{r} \bb)^{\frac{1}{r}}, \label{eq:Xsbkr-def} \\
  \nrm{\varphi}_{X^{s, b}_{r}(\bbR^{1+4})} := & \nrm{\varphi}_{\ell^{2} X^{s, b}_{r}(\bbR^{1+4})} = \bb( \sum_{k} \nrm{P_{k} \varphi}_{X^{s, b}_{k; r}(\bbR^{1+4})}^{2} \bb)^{\frac{1}{2}}, \label{eq:Xsbr-def}\\
  \nrm{\varphi}_{\underline{X}(\bbR^{1+4})} := & \nrm{\Box
    \varphi}_{L^{2}_{t}
    \dot{H}^{-\frac{1}{2}}_{x}(\bbR^{1+4})}, \label{eq:uX-def}
\end{align}
with the obvious modification in the case $r = \infty$. 

The $S^{1}$ norm, to first approximation, is an intermediate norm
between $C_{t}^{0} \dot{H}^{1}_{x} \cap X^{1, \frac{1}{2}}_{\infty}$
and $X^{1, \frac{1}{2}}_{1} \cap \underline{X}$. More
precisely, we have
\begin{equation} \label{eq:embedding4S1} \nrm{\rd_{t,x}
    \varphi}_{L^{\infty}_{t} L^{2}_{x}} + \nrm{\rd_{t,x}
    \varphi}_{X^{0, \frac{1}{2}}_{\infty}} \aleq \nrm{\varphi}_{S^{1}}
  \aleq \nrm{\rd_{t,x} \varphi}_{X^{0, \frac{1}{2}}_{1}} +
  \nrm{\varphi}_{\underline{X}},
\end{equation}
where all norms are defined on $\bbR^{1+4}$. Further properties of
$S^{1}$ will be stated in the course of this subsection.

The spaces $Y^1$ and $S^1$ have an $\ell^2$ dyadic structure in frequency.
However, it is  also useful to work with different dyadic summations.
Precisely, we introduce the notation $\ell^{r} X$ for any function space $X$ on
$\bbR^{1+4}$, where
\begin{equation*}
  \nrm{\varphi}_{\ell^{r} X} := \bb( \sum_{k} \nrm{P_{k} \varphi}_{X}^{r} \bb)^{\frac{1}{r}}.
\end{equation*}

\begin{remark}\label{rem:s1}
One motivation for this is the observation, heavily used in  in \cite{Krieger:2012vj}, that certain portions 
of small data MKG waves exhibit better dyadic summability properties, as follows:
\begin{itemize}
\item The elliptic portion $A_0$ of the solution is  in the smaller space $\ell^1 Y^1$.
\item The hyperbolic component  $A_x$, admits a decomposition 
$A_x = A_{x}^{free} + A_x^{nl}$, where $A_x^{free}$ represents the free wave 
matching the initial data, while the nonlinear portion $A^{nl}$ has the better regularity
$A^{nl} \in \ell^1 S^1$. 
\item The high modulation part of both $A_x$ and $\phi$ has better dyadic summability,
$(A-x,\phi) \in \ell^1 \underline{X}$.
\end{itemize}
We further remark that the $ \ell^1 \underline{X}$ norm was included in $S^1$ in \cite{Krieger:2012vj}.
For the sake of uniformity in notation we do not do this in our series of papers. 
\end{remark}
 
\medskip

In addition to $Y^1$ and $S^1$, in this paper we also need function spaces 
to describe the class of gauge transformations we use in order to assemble 
local solutions to (MKG). The main  space we use for this is
 $\CG:= \ell^{1} Y^{2}(\bbR^{1+4})$, with norm 
\begin{equation} \label{eq:CGnorm}
  \begin{aligned}
    \nrm{\varphi}_{\CG(\bbR^{1+4})}
    =  & \sum_{k} \sum_{N=0}^{2} \bb( 2^{(\frac{5}{2} - N) k}
    \nrm{\rd_{t}^{N} P_{k} \varphi}_{L^{2}_{t,x}(\bbR^{1+4})} +
    2^{(2-N) k} \nrm{\rd_{t}^{N} P_{k} \varphi}_{L^{\infty}_{t}
      L^{2}_{x}(\bbR^{1+4})} \bb).
  \end{aligned}
\end{equation}
For technical reasons we will also consider a variant of $\CG$, namely
the $\wCG$ space. Its norm is defined as
\begin{equation*}
  \nrm{\eta}_{\wCG(\bbR^{1+4})} := \nrm{\eta}_{Y^{2,2}(\bbR^{1+4})} + \sum_{k} 2^{2k} \nrm{P_{k} \eta}_{L^{\infty}_{t} L^{2}_{x}(\bbR^{1+4})}. 
\end{equation*}
It is easy to see that $\wCG(\bbR^{1+4})$ is weaker than
$\CG(\bbR^{1+4})$, i.e.,
\begin{equation} \label{eq:embedding4CG2wCG}
  \nrm{\chi}_{\wCG(\bbR^{1+4})} \aleq \nrm{\chi}_{\CG(\bbR^{1+4})}.
\end{equation}
\medskip

Insofar, we have defined our function spaces on the whole space
$\bbR^{1+4}$.  Here we also need to use them on on compact time
intervals $I \times \bbR^4$ or more generally on open sets.
For this it suffices to take the easy way out 
and use the method of restrictions. Precisely, 
Let $X$ be any one of $Y^{1}$, $S^{1}$, $\CG$, $\wCG$ or
$\dot{B}^{\frac{5}{2}, 2}_{1}$, etc. For an open subset $\0 \neq \calO
\subseteq \bbR^{1+4}$, we define the space $X(\calO)$ to consist of
restrictions of elements in $X(\bbR^{1+4})$ to $\calO$, with the norm
given by
\begin{equation*}
  \nrm{\phi}_{X(\calO)} := \inf \set{ \nrm{\widetilde{\phi}}_{X(\bbR^{1+4})} : 
    \widetilde{\phi} \in X(\bbR^{1+4}), \,
    \widetilde{\phi} = \chi \hbox{ on } \calO}.
\end{equation*}
Given two non-empty open sets $\calO_{1} \supseteq \calO_{2}$, the
restriction map $\calY(\calO_{1}) \to \calY(\calO_{2})$ is a bounded
surjection.

In particular, for $X$ as above and a time interval $I$ 
we will denote by $X[I]$  the restrictions to $I\times \bbR^4$
of $X$ functions.  We refer the reader to the second paper
in our series \cite{OT2} for further discussion of the $S^1[I]$ and $Y[I]$ spaces. 
 
We remark that, in view of the above definition, all algebraic
estimates involving our spaces in $\bbR^{1+4}$ easily carry over to
any nonempty open subsets. In particular this applies to all of the
estimates below in this subsection.

\medskip

The space $\calY$ (more precisely, its local version defined below)
will be the main function space that contains the local gauge
transformations in the proof of Theorem \ref{thm:lwp4MKG}. It has the
desirable property that if $\chi \in \calY$ and $(A, \phi)$ is a
solution to \eqref{eq:MKG} such that $A_{0} \in Y^{1}, A_{x}, \phi \in
S^{1}$, then the gauge transformed solution $(A', \phi') = (A - \ud
\chi, e^{i \chi} \phi)$ also belong to the same functions spaces. The
following lemma justifies a half of this statement, precisely the part dealing with $A$.
The other half is in Lemma \ref{lem:gt4wCG}.

\begin{lemma} \label{lem:embedding4CG} For $\chi \in \CG(\bbR^{1+4})$
  we have
  \begin{equation} \label{eq:embedding4CG} \nrm{\rd_{t}
      \chi}_{\ell^1 Y^{1}(\bbR^{1+4})} + \nrm{\rd_{x}
      \chi}_{\ell^1 S^{1}(\bbR^{1+4})} +
    \nrm{\chi}_{L^{\infty}_{t,x}(\bbR^{1+4})} \aleq
    \nrm{\chi}_{\calY(\bbR^{1+4})}.
  \end{equation}
\end{lemma}
\begin{proof}
 Due to the $\ell^1$ dyadic summation in the $\CG$ norm, we can assume without 
loss of generality that $\chi$ has dyadic frequency localization at frequency $2^k$.
Then the estimate for $\nrm{\rd_{t} \chi}_{Y^{1}} \aleq 1$ is straightforward, while the 
$L^\infty$ bound  is a consequence of Bernstein's inequality. 

To prove the bound for $\nrm{\rd_{x} \chi}_{\ell^1 S^{1}}$, it suffices to
verify the following two bounds for functions $\chi$ at frequency $2^k$:
  \begin{align}
    2^{k} \nrm{Q_{\leq k + 10} \chi}_{X^{1, \frac{1}{2}}_{1}} \aleq& 
\|\chi\|_{\CG}, \label{eq:embedding4CG:pf:1}\\
    2^{k} \nrm{Q_{> k + 10} \chi}_{ \underline{X}} \aleq &\|\chi\|_{\CG}.
 \label{eq:embedding4CG:pf:2}
  \end{align}
  Indeed, thanks to the spatial frequency localization $\chi =
  P_{[k-1, k+1]} \chi$, it follows that $\nrm{\rd_{x} \chi}_{S^{1}}$
  is bounded by the sum of the left-hand sides of the preceding two
  inequalities. The first bound \eqref{eq:embedding4CG:pf:1} is obtained as follows:
  \begin{align*}
    2^{k} \nrm{Q_{\leq k + 10} \chi}_{X^{1, \frac{1}{2}}_{1}} \aleq
    \sum_{j \leq k+10} 2^{2k} 2^{\frac{1}{2} j} \nrm{Q_{j}
      \chi}_{L^{2}_{t,x}} \aleq \sum_{j \leq k+10} 2^{\frac{1}{2}(j -
      k)} (2^{\frac{5}{2}} \nrm{\chi}_{L^{2}_{t,x}}) \aleq \|\chi\|_{\CG}.
  \end{align*}
  The second bound \eqref{eq:embedding4CG:pf:2} follows from the time
  regularity of $\chi$:
  \begin{align*}
    2^{k} \nrm{Q_{> k + 10} \chi}_{\underline{X}} \aleq
    2^{\frac{1}{2} k} \nrm{\Box \chi}_{L^{2}_{t,x}} \aleq
    2^{\frac{1}{2} k} \nrm{\rd_{t}^{2} \chi}_{L^{2}_{t,x}} +
    2^{\frac{5}{2} k} \nrm{\chi}_{L^{2}_{t,x}} \aleq \|\chi\|_{\CG}.
  \end{align*}
  This completes the proof of \eqref{eq:embedding4CG}. \qedhere
\end{proof}

In order to estimate the action of a gauge transformation $\chi$ on
the scalar field $\phi$ in the space $S^{1}$ it suffices to use the weaker norm $\wCG$:

\begin{lemma} \label{lem:gt4wCG} For $\chi^{1}, \chi^{2} \in
  \wCG(\bbR^{1+4})$, we have
  \begin{equation} \label{eq:first:alg4Y} \nrm{\chi^{1}
      \chi^{2}}_{\wCG(\bbR^{1+4})} \aleq
    \nrm{\chi^{1}}_{\wCG(\bbR^{1+4})}
    \nrm{\chi^{2}}_{\wCG(\bbR^{1+4})}.
  \end{equation}

  Moreover, there exist functions $\Gmm_{1} : [0, \infty) \to [1,
  \infty)$ and $\Gmm_{2}: [0, \infty)^{2} \to [1, \infty)$, which grow
  at most polynomially, such that the following estimates hold for
  every $\chi, \chi' \in \wCG(\bbR^{1+4})$ and $\phi, \phi' \in
  S^{1}(\bbR^{1+4})$:
  \begin{align} \label{eq:first:gt4S1}
    \nrm{e^{i \chi} \phi}_{S^{1}} \aleq & \Gmm_{1}(\nrm{\chi}_{\wCG}) \nrm{\phi}_{S^{1}} , \\
    \nrm{e^{i \chi} \phi - e^{i \chi'} \phi'}_{S^{1}} \aleq &
    \Gmm_{1}(\nrm{\chi}_{\wCG}) \nrm{\phi - \phi'}_{S^{1}} +
    \Gmm_{2}(\nrm{\chi}_{\wCG}, \nrm{\chi'}_{\wCG}) \nrm{\chi -
      \chi'}_{\wCG} \nrm{\phi'}_{S^{1}} .
  \end{align}
  Here, all norms are defined on the whole space-time $\bbR^{1+4}$.
\end{lemma}

The proof of this lemma requires further knowledge of the space
$S^{1}$; we will defer this proof until Section \ref{sec:gtCutoff}.

\medskip

The following simple lemma will be useful for patching up local
solutions which satisfy certain compatibility conditions; see
Proposition \ref{prop:patch} and the first two steps in
Section~\ref{subsec:lwp4MKG:pf}.

\begin{lemma} \label{lem:Xcutoff} Let $\eta \in \dot{B}^{\frac{5}{2},
    2}_{1}(\bbR^{1+4})$.  Then for $X = Y^{1}, S^{1}, \CG$ or $\wCG$, we
  have $\eta X \subseteq X$. Furthermore, the following estimate
  holds:
  \begin{align}
    \nrm{\eta \phi}_{X} \aleq \nrm{\eta}_{\dot{B}^{\frac{5}{2},
        2}_{1}} \nrm{\chi}_{X}. \label{eq:Xcutoff}
  \end{align}
\end{lemma}
The proof of this lemma will also be deferred until Section
\ref{sec:gtCutoff}.
  The lemma should be interpreted as saying that the space $X$ is
  stable under multiplication by a smooth rapidly decaying space-time
  cutoff $\eta$.  In this sense, the choice of the space
  $\dot{B}^{\frac{5}{2}, 2}_{1}$ is not essential; it is simply a
  convenient space with a scale-invariant norm in which
  $\calS(\bbR^{1+4})$ is dense.

  \begin{remark} \label{rem:Xcutoff} In order to apply this lemma in
    an open set $\calO$, we need to ensure that $\eta \in
    \dot{B}^{\frac{5}{2}, 2}_{1}(\calO)$, i.e., $\eta$ is the
    restriction to $\calO$ of an element in $\dot{B}^{\frac{5}{2},
      2}_{1}(\bbR^{1+4})$. A simple sufficient condition, which will
    be enough for almost all of our usage below, is if $\eta$ is
    \emph{smooth} on $\calO$ and $\calO$ is a \emph{bounded open set
      with piecewise smooth boundary}.
\end{remark}

We end this subsection with two lemmas, which will be useful for our
proof below of the existence and continuous dependence statements of
Theorem \ref{thm:lwp4MKG}. The first lemma provides a criterion for a
time-independent function $\chi$ to belong to $\CG[I]$ for a compact
time interval $I$.  The same will apply in sets of the form $\calO = I
\times O$, with $ O \subset \bbR^4$, open.

\begin{lemma} \label{lem:CG4t-indep} Let $\underline{\chi} \in
  \dot{B}^{2,2}_{x;1} \cap \dot{B}^{\frac{5}{2},2}_{x; 1} (\bbR^4)$,
  and $I$ be a compact time interval containing $0$.  Extend $\underline{\chi}$ to $I
  \times \bbR^4$ by imposing $\rd_{t} \chi = 0$ and $\chi
  \rst_{\set{0} \times \bbR^4} = \underline{\chi}$. Then $\chi \in
  \CG[I]$ and we have
  \begin{equation} \label{eq:CG4t-indep} 
\nrm{\chi}_{\CG[I]} \aleq \nrm{\underline{\chi}}_{\dot{B}^{2,
        2}_{x,1}} + |I|^{\frac{1}{2}}
    \nrm{\underline{\chi}}_{\dot{B}^{\frac{5}{2}, 2}_{ x,1}} \,.
  \end{equation}
\end{lemma}
\begin{proof}
By scaling and translation we can assume that $I =[0,1]$.
  Due to the $\ell^1$ dyadic summation in the spaces
  $\dot{B}^{2,2}_{x;1} \cap \dot{B}^{\frac{5}{2},2}_{x; 1}(\bbR^{4})$,
  we may assume that $\widetilde{\underline{\chi}}$ has dyadic
  frequency localization, i.e., $\widetilde{\underline{\chi}} =
  P_{[k-1,k+1]} \widetilde{\underline{\chi}}$ for some $k \in \bbZ$.
  To prove the lemma, it suffices to show that there exists an
  extension $\widetilde{\chi}$ of $\widetilde{\underline{\chi}}$ to
  $\bbR^{1+4}$ such that $\widetilde{\chi} \in \CG(\bbR^{1+4})$,
  $\rd_{t} \widetilde{\chi} = 0$ on $I \times \bbR^{4}$,
  $\widetilde{\chi} \rst_{\set{t=0}} = \widetilde{\underline{\chi}}$
  and satisfies
  \begin{equation} \label{eq:CG4t-indep:pf}
    \nrm{\widetilde{\chi}}_{\CG[I]} \aleq
    \nrm{\widetilde{\underline{\chi}}}_{\dot{B}^{2, 2}_{x; 1}
      (\bbR^{4})} +
    \nrm{\widetilde{\underline{\chi}}}_{\dot{B}^{\frac{5}{2}, 2}_{x;
        1}(\bbR^{4})} \,.
  \end{equation}
  Let $\eta \in C^{\infty}_{0}(\bbR)$ be a smooth compactly supported
  function such that $\eta = 1$ on $I$, and take
  $\widetilde{\chi}(t,x) = \eta_{k}(t)
  \widetilde{\underline{\chi}}(x)$, where
  \begin{equation*}
    \eta_{k}(t) = \left\{
      \begin{array}{cc}
        \eta(C^{-1} 2^{k} t) & \hbox{ for } k \leq 0, \\
        \eta(C^{-1} t) & \hbox{ for } k \geq 0.
      \end{array}
    \right.
  \end{equation*}
  In Fourier space, $\widehat{\eta_{k}}$ decays rapidly away from
  $\set{\abs{\tau} \aleq C^{-1} \min \set{2^{k}, 1}}$,
  $\nrm{\widehat{\eta}_{k}}_{L^{1}_{\tau}} \aleq 1$ and
  $\nrm{\widehat{\eta_{k}}}_{L^{2}_{\tau}} \aleq C^{\frac{1}{2}}
  2^{\frac{1}{2} \min\set{k,0}}$. Combining these facts with the
  assumption that $\widetilde{\underline{\chi}}= P_{k}
  \widetilde{\underline{\chi}}$ is frequency localized,
  \eqref{eq:CG4t-indep:pf} follows for $C$ sufficiently large
  (independent of $k$). \qedhere
\end{proof}

The second lemma concerns solving a certain Poisson equation in
$\wCG[I]$, which arises when we attempt to gauge
transform the solution obtained by patching to the global Coulomb
gauge.
\begin{lemma} \label{lem:ellipticEst4wCG} Let $I \subseteq \bbR$ be a time
  interval. Let $\eta \in \dot{B}^{\frac{5}{2}, 2}_{1}[I]$ and $\phi \in \wCG [I]$. Consider the
  Poisson equation
  \begin{equation*}
    - \lap \chi = \eta \lap \phi.
  \end{equation*}
  Then the following statements hold.
  \begin{enumerate}
  \item The right-hand side belongs to $C_{t} \dot{B}^{0, 2}_{x; 1}$,
    and therefore we may define $\chi(t)$ for each $t \in I$
    unambiguously as the convolution of $\eta \lap \phi(t, x)$ with
    the Newton potential, i.e.,
    \begin{equation*}
      \chi(t,x) = \frac{3}{4 \pi^{2}}  \int_{\bbR^{4}} \frac{1}{\abs{x-y}^{2}} \eta(t, y) \lap \phi(t, y) \, \ud y.
    \end{equation*}

  \item Moreover, $\chi \in \wCG[I]$ and satisfies the
    estimate
    \begin{equation} \label{eq:ellipticEst4wCG} \nrm{\chi}_{\wCG[I]} \aleq \nrm{\eta}_{\dot{B}^{\frac{5}{2},
          2}_{1}[I]} \nrm{\phi}_{\wCG [I]} \, .
    \end{equation}
  \end{enumerate}
\end{lemma}

The proof of Lemma \ref{lem:ellipticEst4wCG} will be similar to that
of Lemma \ref{lem:Xcutoff}. Hence it will be given in Section
\ref{sec:gtCutoff} as well.

\subsection{Patching compatible pairs} \label{subsec:patch} In this
subsection, we present a technical tool that will be used to
quantitatively patch together local solutions, which are given by the
small energy theorem (Theorem \ref{thm:KST}), to obtain a global
solution with the desired properties.



We now introduce the notion of \emph{compatible pairs}.

\begin{definition}[Compatible $C_{t} \calH^{1}$
  pairs] \label{def:compatiblePairs} Let $\calO \subseteq \bbR^{1+4}$ be
  an open set and $\calQ = \set{Q_{\alp}}$ be a finite covering of
  $\calO$. For each index $\alp$, consider a pair $(A_{[\alp]},
  \phi_{[\alp]}) \in C_{t} \calH^{1}(Q_{\alp})$ of a real-valued
  1-form $A_{[\alp]}$ and a $\bbC$-valued function $\phi_{[\alp]}$ on
  $Q_{\alp}$.  We say that the pairs $(A_{[\alp]}, \phi_{[\alp]})$ are
  \emph{compatible} if for every $\alp, \bt$ there exists a gauge
  transformation $\chi_{[\alp \bt]} \in C_{t} \calG^{2} \cap
  C^{0}_{t,x}(Q_{\alp} \cap Q_{\bt})$ such that the following
  properties hold:
  \begin{enumerate}
  \item For every $\alp$, we have $\chi_{[\alp \alp]} = 0$.
  \item For every $\alp, \bt$, we have
    \begin{equation} \label{eq:gaugetransform4chi} A_{[\bt]} =
      A_{[\alp]} - \ud \chi_{[\alp \bt]}, \quad \phi_{[\bt]} = e^{i
        \chi_{[\alp \bt]}} \phi_{[\alp]} \quad \hbox{ on } Q_{\alp}
      \cap Q_{\bt},
    \end{equation}
  \item For every $\alp, \bt, \gmm$, the following \emph{cocycle
      condition} is satisfied:
    \begin{equation} \label{eq:cocycle4chi} \chi_{[\alp \bt]} +
      \chi_{[\bt \gmm]} + \chi_{[\gmm \alp]} \in 2 \pi \bbZ \quad
      \hbox{ on } Q_{\alp} \cap Q_{\bt} \cap Q_{\gmm}.
    \end{equation}
  \end{enumerate}
\end{definition}

The main result of this subsection is Proposition \ref{prop:patch}
below, whose formulation and proof were motivated by the classical
result of Uhlenbeck \cite{Uhlenbeck:1982vna} on weak compactness of
connections with curvature bounded in $L^{p}$.

In order to state our result we need to specify the set $\calO$ and the covering
$\calQ$. For this, we begin with the partition 
\[
\bbR^4 = \cup_\alpha R_\alpha \cup R_0^c
\]
given in Proposition~\ref{p:cubes}.  Taking $I = [0,1]$ and $r_\delta
= 1$, (which suffices by scaling), we define 
\[
\calO = I \times \bbR^4, \qquad Q_0 = I \times R_0^c, \qquad Q_\alpha = I \times 1.5 R_\alpha.
\]
The factor $1.5$ above is what guarantees, in view of condition (4) in
Proposition~\ref{p:cubes}, that this covering is locally finite.

We also consider a smaller, subordinated subcovering $\calP =
\set{P_{\alp}}$ given by
\[
P_{\alp} = I \times 1.25 R_{\alp}, \qquad P_0 = I \times (1.001 R_0)^c, 
\qquad \calO = \cup_{\alp} P_{\alp}
\]
This is also locally finite. Using this notations we have:

\begin{proposition}[Patching compatible pairs] \label{prop:patch} Let
  $(A_{[\alp]}, \phi_{[\alp]})$ on $Q_{\alp}$ be \emph{compatible
    pairs} associated to the above covering $\calQ$ of
  $\calO$. Suppose furthermore that for every $\alp, \bt$, the gauge
  transformation $\chi_{[\alp \bt]}$ belongs to $\CG(Q_{\alp} \cap
  Q_{\bt})$ (defined in Section~\ref{subsec:ftnspace4patching}), which
  embeds into $C_{t} \calG^{2} \cap C^{0}_{t,x}(Q_{\alp} \cap
  Q_{\bt})$.

  Let $\set{\underline{\chi}_{[\alp \bt]}}$ be another collection of
  gauge transformations such that $\underline{\chi}_{[\alp \bt]} \in
  \CG(Q_{\alp} \cap Q_{\bt})$ for every $\alp, \bt$, and satisfies the
  cocycle condition \eqref{eq:cocycle4chi}. Assume moreover that
  $\set{\chi_{[\alp \bt]}}$ is $C^{0}$ close to
  $\set{\underline{\chi}_{[\alp \bt]}}$, in the sense that
  \begin{equation} \label{eq:close4chi} \sup_{Q_{\alp} \cap Q_{\bt}}
    \abs{\chi_{[\alp \bt]} - \underline{\chi}_{[\alp \bt]}} <
    \eps_{\ast \ast},
  \end{equation}
  where $\eps_{\ast \ast} > 0$ is a universal constant to be specified
  below.

  Then there exists gauge transformations $\chi_{[\alp]} \in
  \CG(P_{\alp})$ on each $P_{\alp}$, depending linearly on $\chi_{[\alp \bt]}$
and $\underline{\chi}_{[\alp \bt]}$, which satisfy
    \begin{equation*}
      - \chi_{[\alp]} + \chi_{[\alp \bt]} + \chi_{[\bt]} = \underline{\chi}_{[\alp \bt]}	\hbox{ on } P_{\alp} \cap P_{\bt}.
    \end{equation*}
    Moreover, $\chi_{[\alp]}$ obey the following bounds with a universal implicit constant:
    \begin{equation}
      \sup_{\alp} \nrm{\chi_{[\alp]}}_{\CG(P_{\alp})} 
      \aleq
      \sup_{\alp, \bt} \, \bb( \nrm{\chi_{[\alp \bt]}}_{\CG(Q_{\alp} \cap Q_{\bt})} + \nrm{\underline{\chi}_{[\alp \bt]}}_{\CG(Q_{\alp} \cap Q_{\bt})} \bb).
    \end{equation}

\end{proposition}

\begin{remark} 
  The role of the $C^{0}$ closeness condition \eqref{eq:close4chi} is
  to remove the $2 \pi \bbZ$ ambiguity in the cocycle condition
  \eqref{eq:cocycle4chi}.  More precisely, since both $\chi_{[\alp
    \bt]}$ and $\underline{\chi}_{[\alp \bt]}$ satisfy
  \eqref{eq:cocycle4chi}, we have
  \begin{equation*}
    (\chi_{[\alp \bt]} - \underline{\chi}_{[\alp \bt]}) 
    + (\chi_{[\bt \gmm]} - \underline{\chi}_{[\bt \gmm]}) 
    + (\chi_{[\gmm \alp]} - \underline{\chi}_{[\gmm \alp]}) \in 2 \pi \bbZ. 
  \end{equation*}
  For a sufficiently small $\eps_{\ast \ast}$ (say $\eps_{\ast \ast} <
  \frac{2 \pi}{3}$), the $C^{0}$ closeness condition
  \eqref{eq:close4chi} then implies that the absolute value of the
  left-hand side is bounded by $< 2 \pi$; therefore, it follows that
  \begin{equation} \label{eq:cocycle4dltchi} (\chi_{[\alp \bt]} -
    \underline{\chi}_{[\alp \bt]}) + (\chi_{[\bt \gmm]} -
    \underline{\chi}_{[\bt \gmm]}) + (\chi_{[\gmm \alp]} -
    \underline{\chi}_{[\gmm \alp]}) = 0.
  \end{equation}
\end{remark}
\begin{proof}
Our $\{Q_\alpha\}$  covering is locally finite, so let $N_0 = N_0(d)$ (which we can take $4^4$
in dimension $d=4$) be so that each  $Q_\alpha$ intersects at most $N_0$ neighbors. 
Then we define a reduction map $\mathfrak{R}$ which decreases the cube size by a fixed factor,
so that $\mathfrak{R}^{N_0}(Q_\alpha) = P_\alpha$ for $\alpha \neq 0$, with the obvious adjustment 
  $\mathfrak{R}^{-N_0}(Q_0^c) = P_0^c$ for $\alpha = 0$. For uniformity of notation, we write $\mathfrak{R} Q_{0} := (\mathfrak{R}^{-1} (Q_{0}^{c}) )^{c}$, so that $\mathfrak{R}^{N_{0}} Q_{0} = P_{0}$.

  Consider an enumeration of the elements in $\calQ$ by positive
  integers $0, 1, \ldots, K$, in nonincreasing order of size, where we
  take $Q_{0}$ to be the first element.  We proceed by induction on
  this enumeration. 

For the induction step, suppose that we have
  constructed an open covering $\calQ_{k-1} = \set{Q_{\alp, k-1}}$, with 
$P_\alpha \subseteq Q_{\alp, k-1} \subseteq Q_{\alp}$,
  $\calO = \cup_{\alp} Q_{\alp,k-1}$ and gauge transforms
  $\chi_{[\alp]}$ on $Q_{\alp, k-1}$ with $\alp= 1,
  \ldots, k-1$ such that
  \begin{enumerate}
  \item $Q_{\alp, k-1} = \mathfrak{R}^{n(\alpha,k)} Q_{\alp}$ where $n(\alpha,k)$ is between $0$ 
and $N_0$, and is zero for $\alp \geq k-1$,
\label{item:ind4patching:1}
  \item $- \chi_{[\alp]} + \chi_{[\alp \bt]} + \chi_{[\bt]} =
    \underline{\chi}_{[\alp \bt]}$ for $1 \leq \alp, \bt \leq k-1$
    provided $Q_{\alp, k-1} \cap Q_{\bt, k-1} \neq
    \0$, \label{item:ind4patching:2}
  \item $\nrm{\chi_{[\alp]}}_{\CG(Q_{\alp, k-1})} \lesssim X_\alpha$ for
    $1 \leq \alp \leq k-1$, \label{item:ind4patching:3}
  \end{enumerate}
  where 
  \begin{align*}
    X_\alpha & = \sup_{Q_\alpha \cap Q_\beta \neq \emptyset} \, 
\bb( \nrm{\chi_{[\alp \bt]}}_{\CG(Q_{\alp} \cap Q_{\bt})} + \nrm{\underline{\chi}_{[\alp \bt]}}_{\CG(Q_{\alp} \cap Q_{\bt})} \bb).
  \end{align*}

  Define the open covering $\calQ_{k}$ so that $Q_{\alpha,k}=
  \mathfrak{R} Q_{\alpha,k-1}$ if $\alpha \leq k-1$ and $Q_\alpha$ is a neighbor
  of $Q_k$, and $Q_{\alpha,k}= Q_{\alpha,k-1}$ otherwise. We shall
  then construct a gauge transformation $\chi_{[k]}$ on $Q_{k, k} =
  Q_{k}$ such that the above properties hold with $k-1$ replaced by
  $k$, where $\chi_{[\alp]}$ for $\alp \leq k-1$ are
  defined by simply restricting to $Q_{\alp, k} \subseteq Q_{\alp,
    k-1}$.  From this statement, the proposition will follow by
  induction, starting with $Q_{\alp, 0} = Q_{\alp}$ and $\chi_{[0]}=
  0$.

  We remark that the uniformity in the estimate (3) is due to the fact that our
  covering of $\calO$ is locally finite, and also that $\calQ$ is slowly varying.  Indeed, it is obvious in the
  proof below that the construction in the induction step only
  involves $Q_k$ and its neighbors, whose side length is comparable to that of $Q_{k}$. 
  Thus, for each $\alpha$ the sets $Q_{\alpha,k}$ are reduced in size only finitely many times,
  and the cutoff functions $\zeta_{[k]}$ below can be taken to be
  uniformly smooth with respect to the scale of $Q_k$.

  We now proceed with the proof of the induction step.  We begin by defining
  $\widetilde{\chi}_{[k]}$ on $Q_{k} \cap (\cup_{\alp \leq k-1}
  Q_{\alp, k-1})$ to be
  \begin{equation} \label{eq:chitilde} \widetilde{\chi}_{[k]} =
    \chi_{[k \alp]} + \chi_{[\alp]} + \underline{\chi}_{[\alp k]}
    \quad \hbox{ on } Q_{k} \cap Q_{\alp, k-1} \hbox{ if it is
      nonempty.}
  \end{equation}
  Observe that this definition is consistent on $Q_{k} \cap
  (\cup_{\alp \leq k-1} Q_{\alp, k-1})$ thanks to property
  (\ref{item:ind4patching:2}) in the induction hypothesis and the
  exact cocycle condition \eqref{eq:cocycle4dltchi} for $\chi_{[\alp
    \bt]} - \underline{\chi}_{[\alp \bt]}$.
  Moreover, by considering a partition of unity subordinate to
  $\set{Q_{k} \cap Q_{\alp, k-1}}_{\alp=1, \ldots, k-1}$ and using the
  induction hypothesis (\ref{item:ind4patching:3})  and Lemma \ref{lem:Xcutoff}, we can
  derive the estimate
  \begin{align}
    \nrm{\widetilde{\chi}_{[k]}}_{\CG(Q_{k} \cap Q_{\alp, k})} & \aleq_{C_{k-1}} X_k \label{eq:patching:est4chitilde}
  \end{align}

  Now let $\zt_{[k]} : \calO \to [0,1]$ be a smooth function that
  satisfies the following properties:
  \begin{gather} \label{}
    \zt_{[k]} = 0 \quad \hbox{ on } Q_{k} \setminus (\cup_{\alp \leq k-1} Q_{\alp, k-1}), \\
    \zt_{[k]} = 1 \quad \hbox{ on } \cup_{\alp \leq k-1} Q_{\alp, k}.
  \end{gather}
  We remark that such a $\zeta$ exists because by construction the two
  sets $Q_{k} \setminus (\cup_{\alp \leq k-1} Q_{\alp, k-1})$ and $
  \cup_{\alp \leq k-1} Q_{\alp, k}$ are separated by a distance which
  is proportional to the size of $Q_k$. This also allows us to choose 
the functions $\zeta_{[k]}$ uniformly smooth on $Q_k$.

  Now we define
  \begin{equation} \label{eq:patching:chi} 
\chi_{[k]} := \zt_{[k]}
    \widetilde{\chi}_{[k]} \quad \hbox{ on } Q_{k, k} = Q_{k}.
  \end{equation}
  Note that Properties (\ref{item:ind4patching:1}) and
  (\ref{item:ind4patching:2}) are immediately consequences of the
  construction. 

For the property (iii), we observe that 
  $\zt_{[k]} \rst_{Q_{k}}$ can be extended as an element in
  $C^{\infty}_{0}(\bbR^{1+4}) \subseteq \dot{B}^{\frac{5}{2},
    2}_{1}(\bbR^{1+4})$. Thus Property (\ref{item:ind4patching:3})
   follows from Lemma \ref{lem:Xcutoff}
  (in particular, stability of $\CG$ by cutoffs in
  $\dot{B}^{\frac{5}{2},2}_{1}$), Lemma \ref{lem:Xcutoff},
  \eqref{eq:chitilde}, \eqref{eq:patching:est4chitilde} and \eqref{eq:patching:chi}.
  \qedhere
\end{proof}


\subsection{Proof of existence and continuous
  dependence} \label{subsec:lwp4MKG:pf} \setcounter{claim}{0} Using
the tools developed in the previous subsections, we are ready to prove
the existence and continuous dependence statements of Theorem
\ref{thm:lwp4MKG}.  In what follows, we will often use the shorthand
$\calE := \calE[a, e, f, g]$.

\pfstep{Step 0. Preliminaries} Let $(a, e, f, g)$ be an $\calH^{1}$
initial data set satisfying the global Coulomb condition $\rd^{\ell}
a_{\ell} = 0$ and $\calE[a,e,f,g] < \En$. It suffices to assume $\ecs[a, e, f, g] < \infty$, since otherwise the small data result (Theorem~\ref{thm:KST}) is applicable. By scaling, we may take
\begin{equation} \label{eq:lwp4MKG:ecs4aefg} 
\ecs[a, e, f, g] = 1.
\end{equation}
By time reversal symmetry, it suffices to restrict to $t \geq 0$ and consider the unit time interval $I = [0, 1]$.
Let $\set{R_{0}^{c}} \cup \set{R_{\alp}}$ be the covering of $\bbR^{4}$ introduced in Section~\ref{sec:split},
such that the local small energy condition \eqref{eq:loc-enR} holds. 

In what follows, we will construct 
a local-in-time solution $(A, \phi)$ in $I \times \bbR^4$, which obeys the $S^{1}$ a-priori regularity property 
\eqref{eq:lwp4MKG:aprioriEst}. Moreover, we will show that our construction below also has the following two properties:
%
%
%
\begin{itemize}
\item Continuous dependence: the data-to-solution map is continuous as follows:
\[
 \calH^1(\bbR^4)  \ni (a, e, f, g)  \to (A_0,A_x,\phi) \in Y^1(I \times
  \bbR^4) \times S^1(I \times \bbR^4) \times S^1(I \times \bbR^4).
\]
\item Regularity: If in addition $(a, e, f, g) \in
  \calH^\infty(\bbR^4)$ then the solution $(A,\phi)$ belongs to
  $C_{t}^{\infty} \calH^{\infty}(\bbR^4)$.
\end{itemize}

Theorem \ref{thm:lwp4MKG} will then follow by combining these statements with the uniqueness statement proved in Section~\ref{subsec:unique}.


\pfstep{Step 1. Construction of local Coulomb solutions} The goal of
this step is to show that corresponding to the $\calQ = \{ Q_{\alp} \}$ covering of  
$I \times \bbR^{4}$, introduced in  Section~\ref{subsec:patch},
we can produce a compatible local solution
$(A_{[\alp]}, \phi_{[\alp]})$ on each $Q_{\alp}$, each of which is the
restriction of a small energy global Coulomb solution to
\eqref{eq:MKG} given by Theorem \ref{thm:KST}. We will in effect construct these 
solutions on the larger sets $I \times 3R_\alpha$, and then 
simply restrict them to $Q_\alpha$.

\begin{claim} \label{claim:locSol}  The following  hold 
for each Coulomb initial data $(a, e, f, g)$ satisfying \eqref{eq:loc-enR}:
  \begin{enumerate}
  \item \label{item:locSol:locSol} On each set $I \times 3R_{\alp}$ 
there exists an admissible $C_{t}  \calH^{1}(I \times 3R_{\alp})$ solution $(A_{[\alp]},
    \phi_{[\alp]})$ and a gauge transformation
    $\underline{\chi}_{[\alp]} \in
    \calG^{2}(3R_{\alp})$ that satisfy the Coulomb gauge
    condition
    \begin{equation} \label{eq:Coulomb4locSol} \rd^{\ell}
      A_{[\alp] \ell} = 0 \quad \hbox{ on } I \times 3R_\alpha,
    \end{equation}
    and the initial condition
    \begin{equation} \label{eq:id4locSol} (A_{[\alp] j},
      F_{[\alp] 0j}, \phi_{[\alp]}, \covD_{[\alp] t}
      \phi_{[\alp]})  = (a_{j}
      - \rd_{j} \underline{\chi}_{[\alp]}, e_{j}, e^{i
        \underline{\chi}_{[\alp]}} f, e^{i
        \underline{\chi}_{[\alp]}} g) \quad \hbox{ on }  3R_\alpha
      \,.
    \end{equation}
    Moreover, $A_{[\alp] x}, \phi_{[\alp]} \in
    S^{1}(I \times 3R_{\alp})$, $A_{[\alp] 0} \in Y^{1}(I \times 3R_{\alp})$ 
depend continuously on the initial data in $\calH^1$,
and we have the smallness bound
    \begin{gather}
      \nrm{A_{[\alp] 0}}_{Y^{1}(I \times 3R_{\alp})} +
      \nrm{A_{[\alp] x}}_{S^{1}(I \times 3R_{\alp})} +
      \nrm{\phi_{[\alp]}}_{S^{1}(I \times 3R_{\alp})}
      \aleq \thE, \label{eq:existence:est4locSol} 
    \end{gather}

  \item \label{item:locSol:chiub} Extend
    $\underline{\chi}_{[\alp]}$ to $I \times 3R_\alpha$ by requiring
    $\rd_{t} \underline{\chi}_{[\alp]} = 0$; abusing the
    notation slightly, we shall denote the extension by
    $\underline{\chi}_{[\alp]}$. Then
    \begin{equation} \label{eq:Coulomb4chiub} \lap
      \underline{\chi}_{[\alp]} = 0 \quad \hbox{ on } I \times 3R_{\alp}.
    \end{equation}
    Moreover, $\underline{\chi}_{[\alp]} \in \CG(I \times 3R_{\alp})$, depending 
continuously on the initial data,  and obeys the estimate
    \begin{align}
      \nrm{\underline{\chi}_{[\alp]}}_{\CG(I \times 3R_{\alp})}
       \aleq_{\En} 1, \label{eq:existence:est4chiub}
    \end{align}

  \item \label{item:locSol:chiab} For every $\alp$ and $\bt$, there
    exists $\chi_{[\alp \bt]} \in \CG(I \times (3R_{\alp} \cap 3R_{\bt}))$
    that connects $(A_{[\alp]}, \phi_{[\alp]})$ and
    $(A_{[\bt]}, \phi_{[\bt]})$ in the sense of Definition
    \ref{def:compatiblePairs} and satisfies
    \begin{equation} \label{eq:Coulomb4chiab} \lap \chi^{(n)}_{[\alp
        \bt]} = 0 \quad \hbox{ on } I \times (3R_{\alp} \cap 3R_{\bt}).
    \end{equation}
    Moreover, $\chi_{[\alp \bt]}$ depends continuously on the
    initial data and obeys the estimate
    \begin{align}
      \nrm{\chi_{[\alp \bt]}}_{\CG(I \times (3R_{\alp} \cap 3R_{\bt}))}
      &\aleq_{\En} 1 \label{eq:existence:est4chiab} 
    \end{align}
    Finally, the following $C^{0}$ closeness condition holds:
    \begin{equation} \label{eq:close4chiabub} \sup_{I \times 3R_{\alp} \cap
        Q_{\bt}} \abs{\chi_{[\alp \bt]} -
        (\underline{\chi}_{[\alp]} -
        \underline{\chi}_{[\bt]})} < \eps_{\ast\ast},
    \end{equation}
    where $\eps_{\ast \ast} > 0$ is the universal small constant that
    appeared in Proposition \ref{prop:patch}.

  \item \label{item:locSol:higherReg} Higher regularity: if in addition $(a, e, f, g) \in
    \calH^{\infty}$, then for each $\alp, \bt$ we have
    \begin{equation*}
      (A_{[\alp]}, \phi_{[\alp]}) \in C_{t}^{\infty} \calH^{\infty} (I \times 3R_{\alp}), \quad
      \underline{\chi}_{[\alp]} \in \calG^{\infty}(I \times 3R_{\alp}), \quad
      \chi_{[\alp \bt]} \in C_{t}^{\infty} \calG^{\infty}(I \times (3R_{\alp}\cap 3R_{\bt})).
    \end{equation*}
  \end{enumerate}
\end{claim}

We proceed to the proof of this claim.

\pfstep{Step 1.1. Construction of $(A_{[\alp]}, \phi_{[\alp]})$ and
  $\underline{\chi}_{[\alp]}$ for $\alp \geq 1$} 
Our starting point here is the estimate \eqref{eq:loc-enR}. We insert a ball
$3R_\alpha \subset B \subset 2B \subset 18R_\alpha$, which has radius $r_{\alp} \aeq \ell(R_{\alp})$. 
Applying Proposition \ref{prop:gluing}
with $\rgExt = 4/3$ and $\rglue = 2$ to $(a, e, f, g)$ with respect to  the ball $B$,  
 we obtain an initial data set $(\widetilde{a}_{[\alp]},
\widetilde{e}_{[\alp]}, \widetilde{f}_{[\alp]},
\widetilde{g}_{[\alp]}) \in \calH^{1}(\bbR^{4})$,
depending continuously on $(a, e, f, g)$ in $\calH^1$, such that we have the matching 
condition
\begin{align}
(\widetilde{a}_{[\alp]},
\widetilde{e}_{[\alp]}, \widetilde{f}_{[\alp]},
\widetilde{g}_{[\alp]}) =&
  (a, e, f, g) \hbox{ on }
  4 R_\alpha
\end{align}
and small energy
\begin{align}
 \calE[ \widetilde{a}_{[\alp]},
\widetilde{e}_{[\alp]}, \widetilde{f}_{[\alp]},
\widetilde{g}_{[\alp]} ] \ll    \thE^{2}.
\end{align}
However, our small localized data $(\widetilde{a}_{[\alp]},
\widetilde{e}_{[\alp]}, \widetilde{f}_{[\alp]},\widetilde{g}_{[\alp]}) $
is no longer in the Coulomb gauge. To rectify this we use 
a gauge transformation defined by 
\begin{equation*}
  \underline{\chi}_{[\alp]} := - (-\lap)^{-1} \rd^{\ell} \widetilde{a}_{[\alp] \ell},
\end{equation*}
where $(-\lap)^{-1}$ on the right-hand side is defined as convolution
with the Newtonian potential. In general, this expression may not be
uniquely determined if we only knew $\rd^{\ell} \widetilde{a}_{[\alp] \ell}
\in L^{2}_{x}$. However, note that we have the support condition
\begin{equation} \label{eq:locSol:supp4divAk} \supp (\rd^{\ell}
  \widetilde{a}_{[\alp] \ell}) \subseteq 
  9R_\alp\setminus 4R_{\alp},
\end{equation}
since $\widetilde{a} \equiv 0$ outside $9R_\alp$. It follows that $\rd^{\ell}
\widetilde{a}_{[\alp] \ell} \in L^{1}_{x} \cap
L^{2}_{x}(\bbR^{4})$ and therefore the right-hand side is
well-defined.  The gauge transformed data set
\begin{equation} \label{} (\check{a}_{[\alp]},
  \check{e}_{[\alp]}, \check{f}_{[\alp]},
  \check{g}_{[\alp]}) := (\widetilde{a}_{[\alp]} -
  \ud \underline{\chi}_{[\alp]},
  \widetilde{e}_{[\alp]}, e^{i
    \underline{\chi}_{[\alp]}}
  \widetilde{f}_{[\alp]}, e^{i
    \underline{\chi}_{[\alp]}}
  \widetilde{g}_{[\alp]}),
\end{equation}
 is a small energy $\calH^{1}(\bbR^{4})$ Coulomb
initial data set;
hence Theorem \ref{thm:KST} is applicable. Let $(A_{[\alp]}, \phi_{[\alp]})$ be
the unique global small energy Coulomb solution to \eqref{eq:MKG}
given by Theorem \ref{thm:KST}. By construction,
\eqref{eq:Coulomb4locSol} and \eqref{eq:id4locSol} hold; moreover,
\eqref{eq:existence:est4locSol} and the continuous dependence property are
consequences of Theorem \ref{thm:KST}.

We now verify \eqref{eq:Coulomb4chiub} and
\eqref{eq:existence:est4chiub} for $\underline{\chi}_{[\alp]}$. Indeed, by the
support condition \eqref{eq:locSol:supp4divAk} we directly get \eqref{eq:Coulomb4chiub},
as well as the uniform bounds 
\begin{equation*}
  \nrm{\rd_{x}^{(N)} \underline{\chi}_{[\alp]}}_{L^{\infty}_{x}(3.5R_\alp)} \aleq_{ N} r_\alp^{-N} 
\nrm{ \rd^{\ell} \widetilde{a}_{[\alp] \ell}}_{L^2_x}\aleq  r_\alp^{-N} E^\frac12
  \quad \hbox{ for every } N \geq 0.
\end{equation*}
By  Lemma \ref{lem:CG4t-indep} 
(see also Remark~\ref{rem:Xcutoff}) this directly leads to \eqref{eq:existence:est4chiub}.
The continuous dependence similarly follows.

\pfstep{Step 1.3. Construction of $(A_{[\alp]}, \phi_{[\alp]})$ and
  $\underline{\chi}_{[\alp]}$ for $\alp = 0$}  Again we start with
\eqref{eq:loc-enR} but with $\alpha = 0$. This time we insert the ball
$\frac{1}{18} R_0 \subset B \subset 2B \subset \frac13R_0$, which has radius $r_{0} \aeq \ell(R_{0})$,
 and apply Proposition \ref{prop:intGluing} with $\rgExt = \frac43$ and $\rglue = 2$.
We  obtain an initial data set
$(\widetilde{a}_{[0]}, \widetilde{e}_{[0]},
\widetilde{f}_{[0]}, \widetilde{g}_{[0]}) \in
\calH^{1}(\bbR^{4})$ such that
\begin{align}
  (\widetilde{a}_{[0]}, \widetilde{e}_{[0]},
  \widetilde{f}_{[0]}, \widetilde{g}_{[0]})
  =& (a, e, f, g) \hbox{ on } (\tfrac14 R_0)^c , \\
  \calE[\widetilde{a}_{[0]},
  \widetilde{e}_{[0]}, \widetilde{f}_{[0]},
  \widetilde{g}_{[0]}] \ll  &  \eps_{\ast}^{2},
\end{align}
where the last line follows from \eqref{eq:intGluing:energy} and our
choice of $R_{0}$. As before, we define
\begin{equation*}
  \underline{\chi}_{[0]} := - (-\lap)^{-1} \rd^{\ell} \widetilde{a}_{[0] \ell}, 
\end{equation*}
which is unambiguously defined due to 
 the support condition
\begin{equation} \label{eq:locSol:supp4divAinfty} \supp (\rd^{\ell}
  \widetilde{a}_{[0] \ell}) \subseteq 1.5 B \setminus B\,
\end{equation}
as $\widetilde{a}_{[0]} = a$ on $(1.5 B)^c$ is divergence-free.  Again the 
gauge corrected data
\begin{equation} \label{} (\check{a}_{[0]},
  \check{e}_{[0]}, \check{f}_{[0]},
  \check{g}_{[0]}) := (\widetilde{a}_{[0]} - \ud
  \underline{\chi}_{[0]}, \widetilde{e}_{[0]},
  e^{i \underline{\chi}_{[0]}}
  \widetilde{f}_{[0]}, e^{i
    \underline{\chi}_{[0]}}
  \widetilde{g}_{[0]})
\end{equation}
is an $\calH^{1}(\bbR^{4})$ Coulomb initial data set
with energy $\ll \eps_{\ast}^{2}$.  Hence we can apply Theorem
\ref{thm:KST} to define $(A_{[0]}, \phi_{[0]})$ as the unique global
small energy Coulomb solution to \eqref{eq:MKG} given by Theorem
\ref{thm:KST}. Then \eqref{eq:Coulomb4locSol}, \eqref{eq:id4locSol},
\eqref{eq:existence:est4locSol} as well as the the continuous
dependence property and the regularity property follow easily from
construction.

As  \eqref{eq:Coulomb4chiub} is a direct consequence of \eqref{eq:locSol:supp4divAinfty},
it remains to establish the bound \eqref{eq:existence:est4chiub}  for
$\underline{\chi}_{[0]}$. Using again the support condition \eqref{eq:locSol:supp4divAinfty}
and the decay of the Newton potential we obtain 
\[
| \partial_x^N \underline{\chi}_{[0]}(x)| \lesssim r_0^{-N} (1+
r_0^{-1} |x|)^{-2} E^\frac12, \qquad N \geq 0, \quad x \in (2B)^c
\]
which suffices for \eqref{eq:existence:est4chiub}.

\pfstep{Step 1.4. Properties of $\chi_{[\alp \bt]}$} We now proceed to
prove Statement (\ref{item:locSol:chiab}). The existence of
$\chi_{[\alp \bt]}$ will be a consequence of Proposition
\ref{prop:locGeom} (local geometric uniqueness); the estimate
\eqref{eq:existence:est4chiab} and the corresponding continuous
dependence, on the other hand, will follow from the global Coulomb
condition satisfied by each solution $(A_{[\alp]}, \phi_{[\alp]})$.


In what follows, we explain the details in the case $\alp, \bt \neq 0$; the case $\alp = 0$ 
is handled by an obvious modification.
By construction, the initial data for $(A_{[\alp]},\phi_{[\alp]})$ and 
$(A_{[\beta]},\phi_{[\beta]})$ are gauge equivalent on $3R_{\alp} \cap 3 R_{\beta}$, with the 
gauge transformation given by $\underline{\chi}_{[\alp]} - \underline{\chi}_{[\beta]}$.
By scaling, each of these cubes has side length larger than $1$, so their domains
of dependence satisfy
\[
I \times (2R_{\alp} \cap  2R_{\beta}) \subset \calD^+(3R_{\alp}) \cap \calD^+(3R_{\beta})
\]
Hence, Proposition \ref{prop:locGeom} shows that the two solutions are
gauge equivalent in $I \times (2R_{\alp} \cap 2R_{\beta})$.  We
denote by $\chi_{[\alp\beta]} \in C_{t} \calG^{2}(I \times
(2R_{\alp} \cap 2R_{\beta}))$ the transition map.  A-priori this
is only determined modulo $2 \pi$, but this ambiguity is easily fixed
by requiring that
\begin{equation*}
  \chi_{[\alp \bt]}  
  = \underline{\chi}_{[\alp]} - \underline{\chi}_{[\bt]} \qquad 
\text{ on } \{0\} \times   (2R_{\alp} \cap  2R_{\beta}). 
\end{equation*}
 Moreover, this satisfies
\begin{equation}
  \lap \chi_{[\alp \bt]} = 0 \quad \hbox{ on } I \times (2R_{\alp} \cap  2R_{\beta} ),
\end{equation}
thanks to the fact that $\lap \chi_{[\alp \bt]} = \rd^{\ell}
A_{[\alp] \ell} - \rd^{\ell} A_{[\bt] \ell} =
0$. Therefore, by the mean value property of harmonic functions,
\begin{equation} \label{eq:MVP4chiab} \chi_{[\alp \bt]}(t, x) =
  \int \chi_{[\alp \bt]}(t, x-y) r_\alpha^{-4} \varphi(y/r_\alpha) \, \ud y \quad
  \hbox{ for } (t, x) \in Q_{\alp} \cap Q_{\bt},
\end{equation}
where we recall that $\varphi$ is a smooth radial function on
$\bbR^{4}$ with $\int \varphi = 1$ and $\supp \, \varphi \subseteq
\set{\abs{x} \leq 1}$.  Here we have also used the fact that $r_\alpha
\approx r_\beta$, and that an $O(r_\alpha)$ spatial neighborhood of
$Q_\alpha \cap Q_\beta$ is contained in $I \times (2R_{\alp} \cap
2R_{\beta}) $.

It remains to prove \eqref{eq:existence:est4chiab} and \eqref{eq:close4chiabub}. We begin
with the following bounds for $\rd_{t} \chi_{[\alp \bt]}$ and
$\rd_{t}^{2} \chi_{[\alp \bt]}$: Differentiating
\eqref{eq:MVP4chiab} (in $t, x$), using H\"older's inequality and
recalling the identity $\rd_{\mu} \chi_{[\alp \bt]} =
A_{[\alp] \mu} - A_{[\bt] \mu}$, we have for $N \geq 0$
\begin{align}
  \nrm{\rd_{x}^{(N)} \rd_{t, x} \chi_{[\alp
      \bt]}}_{L^{\infty}_{t,x}(Q_{\alp} \cap Q_{\bt})}
  \aleq_{N} & r_\alpha^{-1-N} \nrm{A_{[\alp]} - A_{[\bt]}
  }_{L^{\infty}_{t} L^{4}_{x} (I \times (2R_{\alp} \cap
2R_{\beta}) )}
  \aleq  r_\alpha^{-1-N}  \thE, \label{eq:chiab:sup4dtchi} \\
  \nrm{\rd_{x}^{(N)} \rd_{t}^{2} \chi_{[\alp
      \bt]}}_{L^{\infty}_{t,x}(Q_{\alp} \cap Q_{\bt})}
  \aleq_{N} & r_\alpha^{-2-N} \nrm{\rd_{t} A_{[\alp] 0} - \rd_{t} A_{[\bt]
      0} }_{L^{\infty}_{t} L^{2}_{x} (I \times (2R_{\alp} \cap
2R_{\beta}) )} \aleq
 r_\alpha^{-2-N}  \thE. \label{eq:chiab:sup4dtdtchi}
\end{align}
Taking $N = 0$ and integrating \eqref{eq:chiab:sup4dtchi}, the $C^{0}$
closeness statement \eqref{eq:close4chiabub} follows.
Moreover, we have 
\begin{equation} \label{eq:chiab:sup4chi} \nrm{\chi_{[\alp
      \bt]}}_{L^{\infty}_{t,x}(Q_{\alp} \cap Q_{\bt})} \aleq
  \thE +
  \nrm{\underline{\chi}_{[\alp]}}_{L^{\infty}_{x}(1.5 R_{\alp})}
  +
  \nrm{\underline{\chi}_{[\bt]}}_{L^{\infty}_{x}(1.5 R_{\bt})}
  \aleq_{\En} 1
\end{equation}
thanks to \eqref{eq:existence:est4chiub}. Finally, observe that
$Q_{\alp} \cap Q_{\bt}$ is pre-compact for any pair $\alp$, $\bt$ such
that $\alp \neq \bt$, since there is only one unbounded element in
$\calQ$, namely $Q_{0}$. From the bounds
\eqref{eq:chiab:sup4dtchi}, \eqref{eq:chiab:sup4dtdtchi} and
\eqref{eq:chiab:sup4chi}, and the fact that $Q_{\alp} \cap Q_{\bt}$ is
pre-compact, we may easily construct an extension
$\widetilde{\chi}_{[\alp \bt]}$ of $\chi_{[\alp \bt]}$
such that
\begin{equation} \label{eq:chiab:chitilde}
  \nrm{\widetilde{\chi}_{[\alp \bt]}}_{\CG(\bbR^{1+4})}
  \aleq_{\En} 1
\end{equation}
Finally, we note that $\chi_{[\alp\bt]}$ constructed above depend continuously
on $(A_{[\alpha]},\phi_{[\alpha]})$ and thus on the initial data $(a,e,f,g)$ in $\calH^1$.


\pfstep{Step 1.5. Completion of proof of Claim \ref{claim:locSol}}
Restricting $(A_{[\alp]}, \phi_{[\alp]})$ and
$\underline{\chi}_{[\alp]}$ to $Q_{\alp}$, and
$\chi_{[\alp \bt]}$ to $Q_{\alp} \cap Q_{\bt}$, Statements
(\ref{item:locSol:locSol})--(\ref{item:locSol:chiab}) follow from the
previous steps. On the other hand, Statement
\ref{item:locSol:higherReg} (persistence of regularity) can be quickly
read off from the above construction, using the corresponding
statements in Propositions \ref{prop:gluing}, \ref{prop:intGluing} and
Theorem \ref{thm:KST}. We omit the details.


\pfstep{Step 2. Construction of global almost Coulomb solution} We now
construct a global solution $(A^{\prime }, \phi^{^{\prime }})$
on $I \times \bbR^{4}$ such that $A^{\prime }_{x},
\phi^{\prime } \in S^{1}[I]$ and $A^{\prime }_{0} \in Y^{1}[I]$ by
patching together the compatible pairs obtained in the previous
step. This solution will \emph{not} satisfy the global Coulomb
condition \eqref{eq:globalCoulomb} in general. Nevertheless, it will
have the redeeming feature that the spatial divergence $\rd^{\ell}
A^{\prime }_{\ell}$ obeys an improved bound compared to a general
derivative of a $A^{\prime }$. This feature will be a consequence
of the fact that $(A^{\prime }, \phi^{\prime })$ will be
constructed by patching together local \emph{Coulomb} solutions
$(A_{[\alp]}, \phi_{[\alp]})$.

The above statements are made precise in the following claim.
\begin{claim} \label{claim:almostCoulombSol} For any initial data
  $(a,e,f,g)$ of energy at most $E$, with $r_c \geq 1$ and satisfying\footnote{The only reason for this requirement is to ensure a uniform construction of $(A',\phi')$,
which guarantees its continuous dependence on the initial data.}
  \eqref{eq:loc-enR} there exists an admissible $C_{t} \calH^{1}$
  solution $(A^{\prime }, \phi^{\prime })$ to \eqref{eq:MKG} on $I
  \times \bbR^{4}$ such that the following statements hold.

  \begin{enumerate}
  \item The data for $(A^{\prime }, \phi^{\prime })$ on
    $\set{t=0}$ coincide with $(a, e, f, g)$,
    i.e.,
    \begin{equation} \label{} (A^{\prime }_{j}, F^{\prime
        }_{0j}, \phi^{\prime }, \covD^{\prime }_{t}
      \phi^{\prime }) \rst_{\set{t=0}} = (a_{j}, e_{j},
      f, g).
    \end{equation}

  \item The solution $(A^{\prime }, \phi^{\prime })$ satisfies
    $A^{\prime }_{x}, \phi^{\prime } \in S^{1}[I]$, $A^{\prime }_{0} \in Y^{1}[I]$, depends continuously on the initial data, and obeys
    \begin{align}
      & \nrm{A^{\prime }_{0}}_{Y^{1}[I]} +
      \nrm{A^{\prime }_{x}}_{S^{1}[I]} +
      \nrm{\phi^{\prime }}_{S^{1}[I]} \lesssim_{E, K} 1
      \label{eq:existence:A'} 
    \end{align}
    where $K$ is the total number of cubes in the set $\set{R_{\alp}}$ constructed in Section~\ref{sec:split}. In our case, $K \aleq (r_{0} / \ecs)^{4}$.

  \item \label{item:existence:divA} The spatial divergence of
    $A^{\prime }$ satisfies $\rd^{\ell} A^{\prime }_{\ell} \in
    C^{0}_{t} \dot{B}^{0,2}_{x; 1} (I \times \bbR^{4})$.
    Therefore, the convolution with the Newtonian potential
    \begin{equation*}
      \chi := - (-\lap)^{-1} \rd^{\ell} A^{\prime }_{\ell} = - \frac{3}{4 \pi^{2}} \int_{\bbR^{4}} \frac{1}{\abs{x-y}^{2}} \rd^{\ell} A^{\prime }_{\ell}(t, y) \, \ud y
    \end{equation*} 
    is unambiguously defined and belongs to $C^{0}_{t} \dot{B}^{2,2}_{x; 1}
    \subseteq C^{0}_{t,x}$. Moreover, it satisfies the additional
    estimates
    \begin{align}
      \nrm{\chi}_{\wCG [I]}
      & \aleq_{E, K} 1 \label{eq:existence:divA:wCG} \\
       \nrm{\rd_{x} \chi}_{S^{1}[I]}
      & \aleq_{E, K} 1  \label{eq:existence:divA:S1}
    \end{align}

  \item \label{item:almostCoulombSol:higherReg}  If additionally $(a, e, f, g) \in
    \calH^{\infty}$, then we have
    \begin{equation*}
      (A^{\prime }, \phi^{\prime }) \in C_{t}^{\infty} \calH^{\infty} (I \times \bbR^{4})  \hbox{ and }
      \chi \in C_{t}^{\infty} \calG^{\infty}(I \times \bbR^{4}).
    \end{equation*}
  \end{enumerate}
\end{claim}

To prove the claim, we begin by applying Proposition \ref{prop:patch}
to the covering $\calQ$ of $I \times \bbR^4$, the compatible pairs $(A_{[\alp]},
\phi_{[\alp]})$ and the gauge transformations $\chi_{[\alp
  \bt]}$ and $\underline{\chi}_{[\alp \bt]} :=
\underline{\chi}_{[\alp]} - \underline{\chi}_{[\bt]}$;
note that the $C^{0}$ closeness condition has been established in
\eqref{eq:close4chiabub}. Then for the sub-covering $\calP =
\set{P_{\alp}}$, we obtain  gauge transformations $\chi_{[\alp]} \in
\CG(P_{\alp})$ such that
\begin{gather}
  \nrm{\chi_{[\alp]}}_{\CG(P_{\alp})} \aleq_E 1 \label{eq:existence:est4chi} 
\end{gather}
\begin{equation} \label{eq:almostCoulombSol:compatible}
  \chi_{[\alp \bt]} = \chi_{[\alp]} +
  \underline{\chi}_{[\alp]} - \chi_{[\bt]} -
  \underline{\chi}_{[\bt]}.
\end{equation}
This identity motivates the following definition of the desired global
solution $(A^{\prime }, \phi^{\prime })$. Let $\eta_{\alp}$ be a
smooth partition of unity adapted to the covering $\set{P_{\alp}}$.
Since $\calP$ is a locally finite covering where intersecting cubes
have comparable sizes, we can choose this partition of unity so that
the $\eta_{\alp}$'s are uniformly smooth on the scale of their
respective cubes.
We define the global solution $(A^{\prime }, \phi^{\prime })$ as
follows:
\begin{equation} \label{eq:almostCoulombSol:construction}
  \begin{aligned}
    A^{\prime }_{\mu}
    :=& \sum_{\alp} \eta_{\alp} (A_{[\alp] \mu} - \rd_{\mu} \chi_{[\alp]} - \rd_{\mu} \underline{\chi}_{[\alp]}), \\
    \phi^{\prime } := &\sum_{\alp} \eta_{\alp} e^{i
      (\chi_{[\alp]} + \underline{\chi}_{[\alp]})}
    \phi_{[\alp]}.
  \end{aligned}
\end{equation}
Such a definition makes sense, since
\eqref{eq:almostCoulombSol:compatible} implies that on every $P_{\alp}
\cap P_{\bt} \neq \0$, we have
\begin{align}
  A_{[\alp] \mu} - \rd_{\mu} \chi_{[\alp]} - \rd_{\mu}
  \underline{\chi}_{[\alp]}
  =& A_{[\bt] \mu} - \rd_{\mu} \chi_{[\bt]} - \rd_{\mu} \underline{\chi}_{[\bt]}, \label{eq:almostCoulombSol:compatibility4A} \\
  e^{i (\chi_{[\alp]} + \underline{\chi}_{[\alp]})}
  \phi_{[\alp]} =& e^{i (\chi_{[\bt]} +
    \underline{\chi}_{[\bt]})}
  \phi_{[\bt]}. \label{eq:almostCoulombSol:compatibility4phi}
\end{align}
For every $\alp \neq 0$, $\eta_{\alp} \in \dot{B}^{\frac{5}{2},
  2}_{1}(P_{\alp})$ since $\eta_{\alp}$ is smooth and $P_{\alp}$ is
pre-compact. On the other hand  for $\alp = 0$ we have $1 -
\eta_{0} \in \dot{B}^{\frac{5}{2}, 2}_{1}(P_{0})$.  By
Lemmas \ref{lem:embedding4CG}, \ref{lem:gt4wCG}, \ref{lem:Xcutoff} and
estimates \eqref{eq:existence:est4locSol},
\eqref{eq:existence:est4chiub}, \eqref{eq:existence:est4chi}, we have
\begin{align*}
  \nrm{\eta_{\alp} (A_{[\alp] 0} - \rd_{t} \chi_{[\alp]} -
    \rd_{t} \underline{\chi}_{[\alp]})}_{Y^{1}[I]}
  \aleq_E & 1 \\
  \nrm{\eta_{\alp} (A_{[\alp] x} - \rd_{x} \chi_{[\alp]} -
    \rd_{x} \underline{\chi}_{[\alp]})}_{S^{1}[I]}
  \aleq_E & 1 \\
  \nrm{\eta_{\alp} e^{i (\chi_{[\alp]} +
      \underline{\chi}_{[\alp]})} \phi_{[\alp]}
  }_{S^{1}[I]} \aleq_E & 1
\end{align*}
 Adding up the preceding estimates, \eqref{eq:existence:A'} follows. 
The continuous dependence on the initial data and the persistence of regularity 
also follow directly from our construction.

It remains to establish Statement (\ref{item:existence:divA}) and the
bounds \eqref{eq:existence:divA:wCG},
\eqref{eq:existence:divA:S1}. This part depends crucially on the
special cancellation that occurs only for $\rd^{\ell} A^{\prime
}_{\ell}$. Indeed, thanks to \eqref{eq:Coulomb4locSol},
\eqref{eq:Coulomb4chiub} and
\eqref{eq:almostCoulombSol:compatibility4A} on each $P_{\alp} \cap
P_{\bt} \neq \0$, we have
\begin{gather*}
  \rd^{\ell} A^{\prime }_{\ell} = \rd^{\ell} \sum_{\alp}
  \eta_{\alp} (A_{[\alp] \ell} - \rd_{\ell} \chi_{[\alp]}
  - \rd_{\ell} \underline{\chi}_{[\alp]})
  = - \sum_{\alp} \eta_{\alp} \lap \chi_{[\alp]}, \\
  \rd^{\ell} (A^{\prime}_{\ell} - A^{\prime }_{\ell} ) = -
  \sum_{\alp} \eta_{\alp} \lap (\chi_{[\alp]} - \chi_{[\alp]}).
\end{gather*}
Equipped with these formulae, we are ready to establish
\eqref{eq:existence:divA:wCG} and \eqref{eq:existence:divA:S1}. Since $\eta_{\alp}$ extends naturally to
$\dot{B}^{\frac{5}{2}, 2}_{1}(I \times \bbR^{4})$ and
$\chi_{[\alp]} \in \CG [I] \subseteq
\wCG[I]$, we are in position to apply Lemma
\ref{lem:ellipticEst4wCG} to each summand $\eta_{\alp} \lap
\chi_{[\alp]}$. Then \eqref{eq:existence:divA:wCG} follows.  To
estimate the $S^{1}[I]$ norm of $(- \lap)^{-1} \rd_{j} \rd^{\ell}
A^{\prime }_{\ell}$, simply observe that
\begin{equation*}
  \nrm{(- \lap)^{-1} \rd_{j} \rd^{\ell} A^{\prime }_{\ell}}_{S^{1}[I]}
  \aleq \nrm{A^{\prime }_{x}}_{S^{1}[I]} 
  \lesssim_{E, K} 1
\end{equation*}
Thus \eqref{eq:existence:divA:S1} follows.

\pfstep{Step 3. Gauge transformation to Coulomb solution} In this
final step of the proof of existence and continuous dependence, we
perform a gauge transformation to $(A^{\prime }, \phi^{\prime
  })$ in order to impose the global Coulomb condition $\rd^{\ell}
A_{\ell} = 0$. The gauge transformation cannot be put directly
into $\CG[I]$, but this difficulty can be
circumvented using the elliptic equations of \eqref{eq:MKG} in the
global Coulomb gauge.

 From the previous step, recall
the definition
\begin{equation*}
  \chi = - (-\lap)^{-1} \rd^{\ell} A^{\prime }_{\ell} \quad \hbox{ on } I \times \bbR^{4},
\end{equation*}
where the first term on the right-hand side is defined as in Statement
(\ref{item:existence:divA}) in Claim \ref{claim:almostCoulombSol}. As
$\rd^{\ell} A^{\prime }_{\ell} \rst_{\set{t=0}} = 0$, it follows
that
\begin{equation} \label{eq:existence:id4chi} \chi
  \rst_{\set{t=0}} = 0.
\end{equation}
Directly taking the $\rd_{t}$ derivative of $\chi$ twice and
using the fact that $(A^{\prime }, \phi^{\prime })$ satisfies
\eqref{eq:MKG}, we see that $\rd_{t} \chi$ and $\rd_{t}^{2}
\chi$ are given by
\begin{align*}
  \rd_{t} \chi =& - (-\lap)^{-1} \rd^{\ell} \rd_{t} A^{\prime
    }_{\ell}
  = - (-\lap)^{-1} \bb( \Im[\phi^{\prime } \overline{\covD^{\prime }_{t} \phi^{\prime }}] + \lap A^{\prime }_{0} \bb), \\
  \rd_{t}^{2} \chi = & - (-\lap)^{-1} (\rd_{t} \rd^{\ell}
  F_{0 \ell} + \rd_{t} A^{\prime }_{0} )
  = - (-\lap)^{-1} \bb( \rd^{\ell} \Im[\phi^{\prime }
  \overline{\covD^{\prime }_{\ell} \phi^{\prime }}] + \lap
  \rd_{t} A^{\prime }_{0} \bb).
\end{align*}

Since $\phi^{\prime }, A^{\prime }_{0} \in C^{0}_{t}
\dot{H}^{1}_{x}$ and $\covD^{\prime }_{t,x} \phi^{\prime },
\rd_{t} A^{\prime }_{0} \in C^{0}_{t} L^{2}_{x}$, we have
$\Im[\phi^{\prime } \overline{\covD^{\prime }_{t,x} \phi^{\prime
    }}] \in C^{0}_{t} \dot{H}^{-1}_{x}$. Therefore, $(-\lap)^{-1}
\Im[\phi^{\prime } \overline{\covD^{\prime }_{t} \phi}]$ and
$(-\lap)^{-1} \rd^{\ell} \Im[\phi^{\prime } \overline{\covD^{\prime
    }_{\ell} \phi}]$ are well-defined as convolution with the
Newtonian potential. By the non-existence of non-trivial entire
harmonic functions in $L^{2}_{x}$ and $\dot{H}^{1}_{x} \subseteq
L^{4}_{x}$, it follows that
\begin{align}
  \rd_{t} \chi =& - (-\lap)^{-1} \Im[\phi^{\prime } \overline{\covD^{\prime }_{t} \phi^{\prime }}] + A^{\prime }_{0} \in C^{0}_{t} \dot{H}^{1}_{x} \label{eq:existence:reg4dtchi}\\
  \rd_{t}^{2} \chi =& - (-\lap)^{-1} \rd^{\ell} \Im[\phi^{\prime
    } \overline{\covD^{\prime }_{\ell} \phi^{\prime }}] +
  \rd_{t} A^{\prime }_{0} \in C^{0}_{t}
  L^{2}_{x}. \label{eq:existence:reg4dtdtchi}
\end{align}

%

Let $(A, \phi)$ be defined by applying the gauge
transformation $\chi$ to $(A^{\prime }, \phi^{\prime })$,
i.e., 
\[
(A, \phi) = (A^{\prime } - \ud \chi, e^{i
  \chi} \phi^{\prime }).
\]
 By \eqref{eq:existence:id4chi}, we
have
\begin{equation*}
  (A_{j}, F_{0j}, \phi, \covD_{t} \phi) \rst_{t=0} 
  = (A^{\prime }_{j}, F^{\prime }_{0j}, \phi^{\prime }, \covD^{\prime }_{t} \phi^{\prime }) \rst_{t=0} 
  = (a_{j}, e_{j}, f, g).
\end{equation*} 
Furthermore, thanks to the equation $\lap \chi = \rd^{\ell}
A^{\prime}_{\ell}$, it follows that $(A, \phi)$
satisfies the global Coulomb condition \eqref{eq:globalCoulomb} on
$I \times \bbR^{4}$. By \eqref{eq:existence:divA:wCG},
\eqref{eq:existence:divA:S1}, \eqref{eq:existence:reg4dtchi},
\eqref{eq:existence:reg4dtdtchi} and Lemma \ref{lem:gt4wCG}, we have
$A_{0} \in Y^{1}[I]$ and $A_{x},
\phi \in S^{1}[I]$ with
\begin{equation*}
  \nrm{A_{0}}_{Y^{1}[I]} + \nrm{A_{x}}_{S^{1}[I]} + \nrm{\phi}_{S^{1}[I]} \lesssim_{E, K} 1
\end{equation*}
Combining these statements, we conclude that $(A, \phi)$ is an
admissible $C_{t} \calH^{1}$ solution to \eqref{eq:MKG} in the global
Coulomb gauge on $I \times \bbR^{4}$ with the initial data $(a, e, f,
g)$, which satisfies the conditions in Theorem \ref{thm:lwp4MKG}. 
Further,  from the previous step, it follows that $(A, \phi)$ is
uniformly approximated by $\calH^\infty$ solutions, 
 thereby finishing the proof of Theorem
\ref{thm:lwp4MKG}. We conclude the proof with two remarks:

\begin{remark}
Our construction yields a solution operator that depends continuously on the initial
data for a class of $\calH^1$ data which satisfy the uniform bounds  \eqref{eq:loc-enR}.
However the final result does not depend on the choice of the partition $\set{R_\alpha}$.
\end{remark}

\begin{remark} \label{rem:apriori-bnd}
	Our proof gives an a-priori bound on the $S^{1}$ norm of $(A_{x}, \phi)$ (as well as the $Y^{1}$ norm of $A_{0}$)
	of the form $\aleq (r_{0} / \ecs)^{4} \, C_{E}$, where the dependence on the energy $E$ of $C_{E}$ is polynomial. 
	By comparison with the gauge-free nonlinear wave equation, one would conjecture that the bound should be 
	independent of $r_{0} / \ecs$, and that $C_{E} \aeq E^{1/2} + E$ by \eqref{eq:id-en-H1}. 
	However, our present argument is very far from that.
\end{remark}

\section{Proof of gauge transformation and cutoff
  estimates} \label{sec:gtCutoff} The purpose of this section is to
provide proofs of Lemmas \ref{lem:gt4wCG}, \ref{lem:Xcutoff} and
\ref{lem:ellipticEst4wCG}, which were used in Section \ref{sec:lwp} in
the proof of Theorem \ref{thm:lwp4MKG}. In Section~\ref{subsec:S1}, we
recall some properties of the space $S^{1}$ needed for establishing
these statements. In Section~\ref{subsec:gt}, we give a proof of Lemma
\ref{lem:gt4wCG} concerning gauge transformation with $\chi \in
\wCG$. Finally, in Section~\ref{subsec:cutoff}, we prove Lemmas
\ref{lem:Xcutoff} and \ref{lem:ellipticEst4wCG}.

In this section, when we omit writing the domain on which a norm is
defined, it is to be understood that the norm is defined globally on
$\bbR^{1+4}$. All functions considered in this section will be assumed
to be $\calS(\bbR^{1+4})$, unless otherwise stated. Furthermore, we
will follow the common abuse of terminology and refer to semi-norms as
simply \emph{norms}.

\subsection{Further structure of $S^{1}$} \label{subsec:S1}

We recall the structure of the $S^{1}$ norm from
\cite{Krieger:2012vj}. The $S^{1}$ norm takes the form (see also Remark~\ref{rem:s1}) 
\begin{equation*}
  \nrm{\varphi}_{S^{1}} := \bb( \sum_{k} \nrm{\rd_{t,x} P_{k} \varphi}_{S_{k}}^{2} \bb)^{\frac{1}{2}} + \nrm{\varphi}_{\underline{X}}.
\end{equation*}
The $\underline{X}$ norm was defined in \eqref{eq:uX-def}. For every
$k \in \bbZ$, we define the $S_{k}$ norm as
\begin{equation*}
  \nrm{\varphi}_{S_{k}} := \nrm{\varphi}_{S^{\mathrm{str}}_{k}}  + \nrm{\varphi}_{X^{0, \frac{1}{2}}_{\infty}} + \nrm{\varphi}_{S^{\mathrm{ang}}_{k}}	
\end{equation*}
where the $X^{0, \frac{1}{2}}_{\infty}$ norm was defined in
\eqref{eq:Xsbkr-def}, \eqref{eq:Xsbr-def}, and we define
\begin{gather*}
  \nrm{\varphi}_{S^{\mathrm{str}}_{k}} := \sup_{(q, r) : \frac{1}{q} +
    \frac{3}{2} \frac{1}{r} \leq \frac{3}{4}} 2^{\frac{1}{q} +
    \frac{4}{r} - 2} \nrm{\varphi}_{L^{q}_{t} L^{r}_{x}}, \quad
  \nrm{\varphi}_{S^{\mathrm{ang}}_{k}}
  :=  \sup_{\ell < 0} \nrm{\varphi}_{S^{\mathrm{ang}}_{k, k+2\ell}}, \\
  \nrm{\varphi}_{S^{\mathrm{ang}}_{k, j}} := \bb( \sum_{\omg \in
    \Omg_{\ell}} \nrm{P^{\omg}_{\ell} Q_{<k + 2\ell}
    \varphi}^{2}_{S^{\omg}_{k}(\ell)} \bb)^{\frac{1}{2}}, \quad \hbox{
    where } \ell = \lceil \frac{j-k}{2} \rceil.
\end{gather*}
The preceding square sum runs over $\Omg_{\ell} := \set{\omg}$
consisting of finitely overlapping covering of $\bbS^{3}$ by caps
$\omg$ of diameter $2^{\ell}$, and the symbols of the multipliers
$P^{\omg}_{\ell}$ form a smooth partition of unity associated to this
covering.  The \emph{angular sector norm} $S^{\omg}_{k}(\ell)$
contains the square-summed $L^{2}_{t} L^{\infty}_{x}$ norm with gain
in the radial dimension in Fourier space (essentially as in
\cite{Klainerman:1999do}) and the null frame space (first introduced
in the wave map context \cite{MR1827277, Tao:2001gb}). Fortunately,
for most of our argument, we need not use the fine structure of this
norm. Hence we omit the precise definition, and refer the reader to
\cite[Eq.~(8)]{Krieger:2012vj}. The following stability property for
$S^{\mathrm{ang}}_{k_{0}, j_{0}}$ is our only necessity.
\begin{lemma} \label{lem::stability4Sang} Let $k_{0}, j_{0}, k_{2} \in
  \bbZ$ be such that $j_{0} < k_{0}$. Then for $\eta, \varphi \in
  H^{\infty}_{t,x}(\bbR^{1+4})$, we have
  \begin{equation} \label{eq::stability4Sang} \nrm{ P_{k_{0}} (S_{\leq
        j_{0} - 30} \eta \, P_{k_{2}}
      \varphi)}_{S^{\mathrm{ang}}_{k_{0}, j_{0}}} \aleq
    \nrm{\eta}_{L^{\infty}_{t,x}} \nrm{P_{k_{2}} \varphi}_{S_{k_{2}}}
  \end{equation}
  Moreover, the left-hand side is vacuous unless $k_{2} \in [k_{0} -
  5, k_{0} + 5]$.
\end{lemma}
\begin{proof}
  This lemma is essentially \cite[Section 16:Case
  2(b).3.(b).2(b)]{Tao:2001gb} and \cite[Lemma~9.1]{MR2130618}. We
  sketch the proof, following the notation in \cite[Section
  3]{Krieger:2012vj}.

  We may assume that $k_{2} \in [k_{0} - 5, k_{0}+5]$, as the
  left-hand side is clearly vacuous otherwise. Moreover, using the
  embedding $X^{0, \frac{1}{2}}_{1} \subseteq S^{\mathrm{ang}}_{k_{0},
    j_{0}}$, the case $j_{0} \leq k_{0} - C$ for any constant $C > 0$
  is easy to handle. Hence we may assume that $j_{0} \leq k_{0} - 20$,
  and in particular $j_{0} < k_{2}$.

  Let $\ell_{0} = \lceil \frac{j_{0} - k_{0}}{2} \rceil$ and fix $\omg
  \in \Omg_{\ell_{0}}$. Thanks to the small space-time Fourier support
  of $S_{\leq j_{0} - 30} \eta$, we have
  \begin{align*}
    & \hskip-2em
    P_{k_{0}} P^{\omg}_{\ell_{0}} Q_{<k_{0} + 2 \ell_{0}} (S_{\leq j_{0} - 30} \eta P_{k_{2}} \varphi) \\
    =& P_{k_{0}} P^{\omg}_{\ell_{0}} Q_{<k_{0} + 2 \ell_{0}} \bb(
    S_{\leq j_{0} - 30} \eta \sum_{\omg' \subseteq \omg} P_{k_{2}}
    P^{\omg'}_{\ell_{0}-5} Q_{<k_{2}+2\ell_{0}+10}\varphi \bb)
  \end{align*}
  where we sum over caps $\omg' \in \Omg_{\ell_{0}-5}$ such that
  $\omg' \subseteq \omg$.
  Similarly, given a radially directed rectangular block
  $\calC_{k}(\ell) \subseteq \set{2^{k_{0}-5} \leq \abs{\xi} \leq
    2^{k_{0}+5}}$ of dimensions $2^{k} \times (2^{k+\ell})^{3}$ with
  $k \leq k_{0}$, $\ell \leq 0$ and $k + \ell \geq k_{0} + 2\ell_{0}$,
  we have
  \begin{align*}
    & \hskip-2em
    P_{\calC_{k}(\ell)} P_{k_{0}} P^{\omg}_{\ell_{0}} Q_{<k_{0} + 2 \ell} (S_{\leq j_{0} - 30} \eta P_{k_{2}} \varphi) \\
    =& P_{\calC_{k}(\ell)} P_{k_{0}} P^{\omg}_{\ell_{0}} Q_{<k_{0} + 2
      \ell_{0}} \bb( S_{\leq j_{0} - 30} \eta \sum_{\omg' \subseteq
      \omg} \sum_{\calC'_{k}(\ell)} P_{\calC'_{k}(\ell)} P_{k_{2}}
    P^{\omg'}_{\ell_{0}-5} Q_{<k_{2}+2\ell_{0}+10}\varphi \bb)
  \end{align*}
  where $\omg'$ is summed over the same set and we sum over
  $\calC'_{k}(\ell)$ which is either equal to or adjacent to
  $\calC_{k}(\ell)$.  The projections $P_{\calC_{k}(\ell)}$,
  $P_{\calC'_{k}(\ell)}$ and $P_{k_{0}} P^{\omg}_{\ell_{0}} Q_{<k_{0}
    + 2 \ell}$ are disposable (i.e., has a Schwarz kernel of
  $L^{1}_{t,x}$ norm $\aleq 1$), hence they are bounded in all
  functions spaces under consideration. Moreover, from the definitions
  in \cite[Section 3]{Krieger:2012vj}, it is clear that
  \begin{equation*}
    \nrm{\eta \varphi}_{X} \leq \nrm{\eta}_{L^{\infty}_{t,x}} \nrm{\varphi}_{X}, \quad
  \end{equation*}
  for $X = S^{\mathrm{str}}_{k}$, $L^{2}_{t} L^{\infty}_{x}$, $NE$,
  and $PW^{\pm}_{\omg}(\ell)$. Moreover, for every sign $\pm$ and cap
  $\omg' \in \Omg_{\ell_{0}-5}$ with $\omg' \subseteq \omg$, we have
  \begin{equation*}
    \nrm{\varphi}_{PW^{\pm}_{\omg}(\ell_{0})} \leq \nrm{\varphi}_{PW^{\pm}_{\omg'}(\ell_{0}-5)}.
  \end{equation*}
  Recalling the definition of the $S^{\omg}_{k}(\ell)$ norm
  \cite[Eq.~(8)]{Krieger:2012vj}, we see that
  \begin{equation*}
    \nrm{P_{k_{0}} P^{\omg}_{\ell} Q_{<k + 2 \ell} (S_{\leq j_{0} - 30} \eta P_{k_{2}} \varphi)}_{S^{\omg}_{k_{0}}(\ell_{0})}
    \aleq \nrm{\eta}_{L^{\infty}_{t,x}} \sum_{\omg' \subseteq \omg} \nrm{P_{k_{2}} P^{\omg'}_{\ell_{0}-5} Q_{<k_{2}+2\ell_{0}+10} \varphi}_{S^{\omg'}_{k_{2}}(\ell_{0}-5)}.
  \end{equation*}
  We square sum this bound in $\omg \in \Omg_{\ell_{0}}$. Note that if
  we replace $Q_{<k_{2}+2\ell_{0}+10}$ by $Q_{<k_{2}+2\ell_{0}-10}$,
  then the last factor is controlled by the $S^{\mathrm{ang}}_{k_{2},
    k_{2}+\ell_{0}-5}$ norm of $P_{k_{2}} \varphi$. For the resulting
  error, we use the embedding $X^{0, \frac{1}{2}}_{1} \subseteq
  S^{\omg'}_{k}(\ell)$ and estimate
  \begin{equation*}
    \bb( \sum_{\omg \in \Omg_{\ell}} \nrm{P_{k} P^{\omg}_{\ell} Q_{k+2 \ell -C \leq \cdot <k+2\ell+C} \varphi}_{S^{\omg'}_{k}(\ell)}^{2} \bb)^{\frac{1}{2}}
    \aleq_{C} \nrm{P_{k} Q_{k+2 \ell -C \leq \cdot <k+2\ell+C} \varphi}_{X^{0, \frac{1}{2}}_{1}}
    \aleq_{C} \nrm{P_{k} \varphi}_{X^{0, \frac{1}{2}}_{\infty}},
  \end{equation*}
  and apply this inequality to $k = k_{2}$, $\ell = \ell_{0}-5$ and $C
  = 10$.  The lemma follows. \qedhere
\end{proof}

\subsection{Gauge transformation estimate} \label{subsec:gt}
Here we establish Lemma \ref{lem:gt4wCG}. This is carried out in two steps.
The first one deals with the algebra type property for the space $\wCG$:

\begin{lemma} \label{lem:gt4S1:restate}
The space $\wCG$ is an algebra,
\begin{align} 
	\nrm{\chi^{1} \chi^{2}}_{\wCG} 
	& \aleq \nrm{\chi^{1}}_{\wCG} \nrm{\chi^{2}}_{\wCG}, \label{eq:alg4Y} 
\end{align}
Further, for any $F$ of class $C^6(\bbR)$ with $F(0) = 0$ we have the
Moser type estimate
\begin{align} 
  \nrm{F(\chi)}_{\wCG} & \aleq (\nrm{\chi}_{\wCG}
  +\nrm{\chi}_{\wCG}^2) (1+
  \nrm{\chi}_{L^\infty_{t,x}}^4), \label{eq:nonlin4Y}
\end{align}
\end{lemma}
\begin{proof}
The main step of the proof is to establish the result for a component of the $\wCG$ 
norm, namely the $\ell^1 L^\infty_{t} \dot H^2_{x}$ norm. We begin with a simple observation,
namely that by Bernstein's inequality we have
\[
\nrm{\chi}_{L^\infty_{t,x}} \lesssim \nrm{\chi}_{\ell^1 L^\infty_{t} \dot H^2_{x}}
\]
This is the only place where the $\ell^1$ summation is used.  The
bound \eqref{eq:alg4Y} for the $\ell^1 L^\infty_{t} \dot H^2_{x}$ norm is now an
application of the standard Littlewood-Paley trichotomy, which in
effect yields the stronger bound
\[
	\nrm{\chi^{1} \chi^{2}}_{\ell^{1} L^{\infty}_{t} \dot{H}^{2}_{x}} 
 \aleq \nrm{\chi^{1}}_{\ell^1 L^\infty_{t} \dot H^2_{x}} \nrm{\chi^{2}}_{L^\infty_{t,x}}+ 
 \nrm{\chi^{1}}_{L^\infty_{t,x}}  \nrm{\chi^{2}}_{\ell^1 L^\infty_{t} \dot H^2_{x}}
\]
A similar bound can be proved for the $Y^{2, 2}$ norm in an analogous manner.

To estimate $F(\chi)$ we use a continuous Littlewood-Paley theory 
decomposition, 
\[
1 = \int_{-\infty}^\infty P_k \, \ud k, \qquad P_{<j} = \int_{-\infty}^j P_k \, \ud k
\]
where $\chi$ is a continuous dyadic frequency parameter.
See e.g. \cite{MR2130618} for a similar argument. Representing
$\chi$ as 
\[
\chi = \int_{-\infty}^\infty P_k \chi \, \ud k,
\]
for $F(\chi)$ we have the similar representation
\[
F(\chi) = F(0) +  \int_{-\infty}^\infty F'(P_{<k} \chi ) P_k \chi \, \ud k
\]
which is easily seen to converge in $L^\infty_{t,x}$. Now it suffices to estimate
the nonlinear term in $L^\infty_{t,x}$,
\[
\| \partial_x^N F'(P_{<k} \chi )\|_{L^\infty_{t,x}} \lesssim 2^{-Nk} (1+\|
\chi\|_{L^\infty_{t,x}}^3), \qquad N = 0,1,2,3.
\]
Then the integrand satisfies the bound
\[
\| \partial_x^N F'(P_{<k} \chi ) P_k \chi \|_{L^\infty_{t} L^2_{x}} \lesssim
2^{(2-N)k} \|P_k \chi \|_{L^\infty_{t} \dot H^2_{x}}
\]
After dyadic integration in $k$ this yields the bound
\begin{equation}\label{moser}
\| F(\chi)\|_{\ell^1 L^\infty_{t} \dot H^2_{x}} \lesssim   \| \chi\|_{\ell^1 L^\infty_{t} \dot H^2_{x}}   (1+\|
\chi\|_{L^\infty_{t,x}}^3)
\end{equation}
which is the $\ell^1 L^\infty_{t} \dot H^2_{x}$ counterpart of \eqref{eq:nonlin4Y}.

To also estimate the $Y^{2,2}$ norm of $F(\chi)$ we differentiate twice,
\begin{equation}\label{moser-dec}
\partial_{x,t}^2  F(\chi) = \partial_{x,t}^2 \chi  F'(\chi) +  
\partial_{x,t} \chi\partial_{x,t} \chi  F''(\chi) 
\end{equation}
We need to estimate the terms on the right in $L^2_{t} \dot H^\frac12_{x}$.
We have the Bernstein type bounds
\[
\|  F'(\chi) - F'(0)\|_{L^\infty_{t,x} \cap L^\infty_{t} \dot W^{1,4}} \lesssim \| F(\chi)\|_{\ell^1 L^\infty_{t} \dot H^2_{x}}
\]
and similarly for $F''(\chi)$, where the norm on the right is further estimated as in 
\eqref{moser}.  Also we control $ \partial_{x,t}^2 \chi $ in $L^2_{t} \dot H^\frac12_{x}$,
as well as 
\[
\| \partial_{x,t} \chi\|_{L^4 \dot W^{\frac34,4}_{x}} \lesssim 
 \| \partial^2_{x,t} \chi\|_{L^2_{t} \dot H^{\frac12}_{x}}^\frac12 \|\chi\|_{L^\infty_{t, x}} 
\]
Hence for the first term on the right in \eqref{moser-dec} it remains to establish the bound
\[
\| f G\|_{L^2_{t} \dot H^s_{x}} \lesssim \|f \|_{L^2_{t} \dot H^s_{x}} \|G\|_{L^\infty_{t,x}
  \cap L^\infty_{t} \dot W^{1,4}_{x}}, \qquad s = \frac12.
\]
But this follows by interpolation from the $s=0$ and $s=1$ cases, which are
straightforward.

Similarly, for the second term on the right in \eqref{moser-dec} we
need to establish the bound
\[
\| f_1 f_2 G\|_{L^2_{t} \dot H^\frac12_{x}} \lesssim \|f_1 \|_{L^4 \dot W^{\frac34,4}_{x} } 
 \|f_1 \|_{L^4 \dot W^{\frac34,4}_{x} } 
\|G\|_{L^\infty_{t, x}  \cap L^\infty_{t} \dot W^{1,4}_{x}}.
\]
which is again a simple exercise which is left for the reader. \qedhere

\end{proof}

The second step deals with the stability of the $S^1$ space with respect to multiplication
by $\wCG$. Before we state it, we begin with a dyadic decomposition of the $Y^{n,2}$
norms which will be used repeatedly in the sequel. Precisely, for $N = 0, 1, 2, \ldots$, the following square summability estimate holds:
\begin{align} \label{eq:sqsum4Y}
\bb( \sum_{\ell} 2^{2 N \ell} \nrm{S_{\ell} \chi}_{Y^{0,2}}^{2} \bb)^{\frac{1}{2}}
+ \bb( \sum_{k, j} (2^{N k} + 2^{N j})^{2} \nrm{P_{k} T_{j} \chi}_{Y^{0,2}}^{2} \bb)^{\frac{1}{2}} \aleq & \nrm{\chi}_{Y^{N,2}}.
\end{align}
We have:

\begin{lemma} \label{lem:gt4S1:restate+}
The following estimate holds:
\begin{align} 
	\nrm{\chi \varphi}_{S^{1}} 
	& \aleq \nrm{\chi}_{\wCG} \nrm{\varphi}_{S^{1}}  . \label{eq:gt4S1}
\end{align}
\end{lemma}
\begin{proof} 
 We begin by splitting
 \begin{equation} \label{eq::gt4S1:pf:1}
	\chi \varphi 
	= \sum_{k_{0}} P_{k_{0}} Q_{\leq k_{0} + 25} (\chi \varphi) + \sum_{k_{0}} P_{k_{0}} Q_{> k_{0} + 25} (\chi \varphi)
\end{equation}

\pfstep{Step 1. Contribution of $\sum_{k_{0}} P_{k_{0}} Q_{\leq k_{0} + 25} (\chi \varphi)$}
In this step, we will show
\begin{equation} \label{eq::gt4S1:pf:2}
	\nrm{\sum_{k_{0}} P_{k_{0}} Q_{\leq k_{0} + 25} (\chi \varphi)}_{S^{1}} \aleq \nrm{\chi}_{Y^{2,2} \cap L^{\infty}_{t,x}} \nrm{\varphi}_{S^{1}}.
\end{equation}

We need different arguments for different parts of the $S^{1}$ norm. The common strategy, however, is to divide into two cases, one in which $\chi$ has a high space-time frequency and the other in which $\chi$ has very low space-time frequency.

In the former case, we will rely on the following simple lemma:

\begin{lemma} \label{lem::nearCone}
Let $j_{0} \leq k_{0}+30$. 
\begin{enumerate}
\item For $\ell > k_{0} - 5$, we have
\begin{equation} \label{eq::nearCone:1}
\begin{aligned}
& \hskip-2em
2^{k_{0}} 2^{\frac{1}{2} j_{0}} \nrm{P_{k_{0}} Q_{j_{0}} (S_{\ell} \chi \, P_{k_{2}} \varphi)}_{L^{2}_{t,x}} \\
\aleq & 2^{\frac{1}{2} (j_{0}-k_{0})} 2^{\frac{3}{2} (k_{0}-\ell)} 2^{\frac{1}{2} (k_{2}-\ell)} (2^{2\ell} \nrm{S_{\ell} \chi}_{Y^{0,2}}) (2^{k_{2}} \nrm{P_{k_{2}} \varphi}_{S_{k_{2}}}),
\end{aligned}
\end{equation}
and the left-hand side of \eqref{eq::nearCone:2} is vacuous unless $k_{2} \leq \ell + 10$.

\item If $\ell \leq k_{0} - 5$, we have instead
\begin{equation} \label{eq::nearCone:2}
2^{k_{0}} 2^{\frac{1}{2} j_{0}} \nrm{P_{k_{0}} Q_{j_{0}} (S_{\ell} \chi \, P_{k_{2}} \varphi)}_{L^{2}_{t,x}} \\
\aleq 2^{\frac{1}{2} (j_{0}-\ell)} (2^{2\ell} \nrm{S_{\ell} \chi}_{Y^{0,2}}) (2^{k_{2}} \nrm{P_{k_{2}} \varphi}_{S_{k_{2}}}).
\end{equation}
Moreover, the left-hand side of \eqref{eq::nearCone:2} is vacuous unless $k_{2} \in [k_{0}-5, k_{0}+5]$.
\end{enumerate}
\end{lemma}

\begin{proof} 
The claims regarding the range of $k_{2}$ are clear. We estimate the left-hand side of \eqref{eq::nearCone:1} by
\begin{align*}
	\aleq 2^{k_{0}} 2^{\frac{1}{2} j_{0}} \nrm{S_{\ell} \chi}_{L^{2}_{t} L^{\frac{8}{3}}_{x}} \nrm{P_{k_{2}} \varphi}_{L^{\infty}_{t} L^{8}_{x}}
	\aleq 2^{\frac{1}{2} (k_{2}-\ell)} 2^{\frac{1}{2} (j_{0}-\ell)} 2^{k_{0}-\ell} (2^{2\ell} \nrm{S_{\ell} \chi}_{Y^{0,2}}) (2^{k_{2}} \nrm{P_{k_{2}} \varphi}_{S_{k_{2}}}).
\end{align*}
For \eqref{eq::nearCone:2}, we estimate
\begin{align*}
	\aleq 2^{k_{0}} 2^{\frac{1}{2} j_{0}} \nrm{S_{\ell} \chi}_{L^{2}_{t} L^{\infty}_{x}} \nrm{P_{k_{2}} \varphi}_{L^{\infty}_{t} L^{2}_{x}} 
	\aleq & 2^{\frac{1}{2} (j_{0}-\ell)}  (2^{2\ell} \nrm{S_{\ell} \chi}_{Y^{0,2}}) (2^{k_{2}} \nrm{P_{k_{2}} \varphi}_{S_{k_{2}}}). \qedhere
\end{align*}
\end{proof}

We now proceed to treat each constituent of the $S^{1}$ norm.

\pfstep{Case 1.1. $S^{\mathrm{str}}_{k}$ part of $S^{1}$}
Here we prove
\begin{equation} \label{eq::nearCone:Sstr}
	\bb( \sum_{k_{0}} 2^{2k_{0}} \nrm{  P_{k_{0}} Q_{\leq k_{0} + 25} (\chi \varphi)}_{S^{\mathrm{str}}_{k_{0}}}^{2} \bb)^{\frac{1}{2}}
	\aleq \nrm{\chi}_{Y^{2,2} \cap L^{\infty}_{t,x}} \nrm{\varphi}_{S^{1}} \, .
\end{equation}

We split the summand on the left-hand side as follows:
\begin{equation} \label{eq::nearCone:Sstr1}
\begin{aligned}
2^{k_{0}} \nrm{  P_{k_{0}} Q_{\leq k_{0} + 25} (\chi \varphi)}_{S^{\mathrm{str}}_{k_{0}}} 
\aleq &  \sum_{\ell > k_{0} - 5} 2^{k_{0}} \nrm{  P_{k_{0}} Q_{\leq k_{0} + 25} (S_{\ell} \chi \varphi)}_{S^{\mathrm{str}}_{k_{0}}}  \\
	& +2^{k_{0}} \nrm{  P_{k_{0}} Q_{\leq k_{0} + 25} (S_{\leq k_{0} - 5} \chi \varphi)}_{S^{\mathrm{str}}_{k_{0}}} \, .
\end{aligned}
\end{equation}

For the first term on the right-hand side, we use the embedding $P_{k_{0}} (X^{0, \frac{1}{2}}_{1}) \subseteq S^{\mathrm{str}}_{k_{0}}$ and Lemma \ref{lem::nearCone} to estimate
\begin{align*}
	\aleq &  \sum_{\ell > k_{0} - 5} \sum_{j_{0} \leq k_{0} + 25} 2^{k_{0}} 2^{\frac{1}{2} j_{0}} \nrm{  P_{k_{0}} Q_{j_{0}} (S_{\ell} \chi \varphi)}_{L^{2}_{t,x}}  \\
	\aleq &  \sum_{\ell > k_{0} - 5} \sum_{j_{0} \leq k_{0} + 25} \sum_{k_{2} \leq \ell+10}
			2^{\frac{1}{2}(j_{0} - k_{0})} 2^{\frac{3}{2} (k_{0} - \ell)} 2^{\frac{1}{2} (k_{2} - \ell)}
			(2^{2 \ell} \nrm{S_{\ell} \chi}_{Y^{0, 2}} ) (2^{k_{2}} \nrm{P_{k_{2}} \varphi}_{S_{k_{2}}}) \\
	\aleq &  \nrm{\varphi}_{S^{1}} \sum_{\ell > k_{0} - 5} 
			2^{\frac{3}{2} (k_{0} - \ell)} 
			\, 2^{2 \ell} \nrm{S_{\ell} \chi}_{Y^{0, 2}}
\end{align*}
which is square summable in $k_{0}$, thanks to \eqref{eq:sqsum4Y}.

For the second term in \eqref{eq::nearCone:Sstr1}, we can freely replace $\varphi$ by $P_{[k_{0}-5, k_{0}+5]} \varphi$. Then removing $P_{k_{0}} Q_{\leq k_{0} + 25}$, which is disposable, and using H\"older with $S_{\leq k_{0} - 5} \chi \in L^{\infty}_{t,x}$, we see that
\begin{equation*}
2^{k_{0}} \nrm{  P_{k_{0}} Q_{\leq k_{0} + 25} (S_{\leq k_{0} - 5} \chi \varphi)}_{S^{\mathrm{str}}_{k_{0}}}
\aleq \sum_{k_{2} \in [k_{0} - 5, k_{0} + 5]} \nrm{\chi}_{L^{\infty}_{t,x}} 2^{k_{2}} \nrm{P_{k_{2}} \varphi}_{S^{\mathrm{str}}_{k_{2}}},
\end{equation*}
which is acceptable.

\pfstep{Case 1.2. $X^{0, \frac{1}{2}}_{\infty}$ part of $S^{1}$}
We prove
\begin{equation} \label{eq::nearCone:Xsb}
	\bb( \sum_{k_{0}} 2^{2k_{0}} \nrm{P_{k_{0}} Q_{\leq k_{0}+25} (\chi \varphi)}_{X^{0, \frac{1}{2}}_{\infty}}^{2} \bb)^{\frac{1}{2}}
	\aleq \nrm{\chi}_{Y^{2,2} \cap L^{\infty}_{t,x}} \nrm{\varphi}_{S^{1}} \, .
\end{equation}
The summand on the left-hand side is bounded by
\begin{align}
	\sup_{j_{0} \leq k_{0} + 25} 2^{k_{0}} 2^{\frac{1}{2} j_{0}} \nrm{P_{k_{0}} Q_{j_{0}} (\chi \varphi)}_{L^{2}_{t,x}}
	\aleq & \sup_{j_{0} \leq k_{0} + 25} 2^{k_{0}} \sum_{\ell > k_{0} - 5} 2^{\frac{1}{2} j_{0}} \nrm{P_{k_{0}} Q_{j_{0}} (S_{\ell} \chi \varphi)}_{L^{2}_{t,x}} \notag \\
		&+ \sup_{j_{0} \leq k_{0} + 25} 2^{k_{0}}\sum_{\ell \in [j_{0}-30, k_{0} - 5]} 2^{\frac{1}{2} j_{0}} \nrm{P_{k_{0}} Q_{j_{0}} (S_{\ell} \chi \varphi)}_{L^{2}_{t,x}} \label{eq::nearCone:Xsb1} \\
		&+ \sup_{j_{0} \leq k_{0} + 25} 2^{k_{0}} 2^{\frac{1}{2} j_{0}} \nrm{P_{k_{0}} Q_{j_{0}} (S_{\leq j_{0} - 30}\chi \varphi)}_{L^{2}_{t,x}} \notag
\end{align}

Let $j_{0} \leq k_{0} + 25$. Using Lemma \ref{lem::nearCone} and proceeding as in Case 1.1, the first term can be bounded by
\begin{align*} 
& \hskip-2em
2^{k_{0}} \sum_{\ell > k_{0} - 5} 2^{\frac{1}{2} j_{0}} \nrm{P_{k_{0}} Q_{j_{0}} (S_{\ell} \chi \varphi)}_{L^{2}_{t,x}} \\
	\aleq & \sum_{\ell > k_{0} - 5} \sum_{k_{2} \leq \ell  +10} 
		2^{\frac{1}{2}(j_{0} - k_{0})} 2^{\frac{3}{2}(k_{0}- \ell)} 2^{\frac{1}{2}(k_{2} - \ell)} 
		(2^{2 \ell} \nrm{S_{\ell} \chi}_{Y^{0,2}}) (2^{k_{2}} \nrm{P_{k_{2}} \varphi}_{S_{k_{2}}}) \\
	\aleq & 2^{\frac{1}{2} (j_{0} - k_{0})} \nrm{\varphi}_{S^{1}} \sum_{\ell > k_{0} - 5} 2^{\frac{3}{2}(k_{0} - \ell)} \, 2^{2 \ell} \nrm{S_{\ell} \chi}_{Y^{0,2}},
\end{align*}
which is $\ell^{2}$ summable in $k_{0}$ thanks to \eqref{eq:sqsum4Y}.

For the second term in \eqref{eq::nearCone:Xsb1}, we can replace $\varphi$ by $P_{[k_{0} - 5, k_{0} - 5]} \varphi$. Then we estimate
\begin{align*}
& \hskip-2em
2^{k_{0}}\sum_{\ell \in [j_{0}-30, k_{0} - 5]} 2^{\frac{1}{2} j_{0}} \nrm{P_{k_{0}} Q_{j_{0}} (S_{\ell} \chi \varphi)}_{L^{2}_{t,x}} \\
	\aleq & \sum_{\ell \in [j_{0}-30, k_{0} - 5]} \sum_{k_{2} \in [k_{0} -5, k_{0}+5]}
		2^{\frac{1}{2} (j_{0} - \ell)} 
		(2^{2 \ell} \nrm{S_{\ell} \chi}_{Y^{0, 2}}) (2^{k_{2}} \nrm{P_{k_{2}} \varphi}_{S_{k_{2}}}) \\
	\aleq & \nrm{\chi}_{Y^{2,2}} \sum_{k_{2} \in [k_{0} -5, k_{0}+5]}
		 (2^{k_{2}} \nrm{P_{k_{2}} \varphi}_{S_{k_{2}}}),
\end{align*}
which is acceptable.

For the third term in \eqref{eq::nearCone:Xsb1}, we can replace $\varphi$ by $P_{[k_{0} - 5, k_{0} - 5]} Q_{[j_{0} - 5, j_{0}+5]} \varphi$. Therefore
\begin{equation*}
	2^{k_{0}} 2^{\frac{1}{2} j_{0}} \nrm{P_{k_{0}} Q_{j_{0}} (S_{\leq j_{0} - 5}\chi \varphi)}_{L^{2}_{t,x}}
	\aleq \nrm{\chi}_{L^{\infty}_{t,x}}
		\sum_{k_{2} \in [k_{0} - 5, k_{0} + 5]} \sum_{j_{2} \in [j_{0} - 5, j_{0} + 5]}
		 2^{k_{2}} 2^{\frac{1}{2} j_{2}} \nrm{P_{k_{2}} Q_{j_{2}} \varphi}_{L^{2}_{t,x}}
\end{equation*}
which is acceptable.

\pfstep{Case 1.3. $S^{\mathrm{ang}}_{k, j}$ part of $S^{1}$}
Here we prove
\begin{equation} \label{eq::nearCone:Sang}
	\bb( \sum_{k_{0}} 2^{2 k_{0}} \sup_{j_{0} < k_{0}} \nrm{P_{k_{0}} Q_{< j_{0}} (\chi \varphi)}_{S^{\mathrm{ang}}_{k_{0}, j_{0}}}^{2} \bb)^{\frac{1}{2}}
	\aleq \nrm{\chi}_{Y^{2,2} \cap L^{\infty}_{t,x}} \nrm{\varphi}_{S^{1}} \, .
\end{equation}

Fix $k_{0}$ and $j_{0} < k_{0}$. As before, we split
\begin{align*}
2^{k_{0}} \nrm{P_{k_{0}} Q_{< j_{0}} (\chi \varphi)}_{S^{\mathrm{ang}}_{k_{0}, j_{0}}}
\aleq &\sum_{\ell > k_{0} + 5} 2^{k_{0}} \nrm{P_{k_{0}} Q_{< j_{0}} (S_{\ell} \chi \varphi)}_{S^{\mathrm{ang}}_{k_{0}, j_{0}}} \\
	& + \sum_{\ell \in [j_{0} - 30, k_{0} - 5]} 2^{k_{0}} \nrm{P_{k_{0}} Q_{< j_{0}} (S_{\ell} \chi \varphi)}_{S^{\mathrm{ang}}_{k_{0}, j_{0}}} \\
	& + 2^{k_{0}} \nrm{P_{k_{0}} Q_{< j_{0}} (S_{< j_{0} - 30} \chi \varphi)}_{S^{\mathrm{ang}}_{k_{0}, j_{0}}}
\end{align*}

Using the embedding $P_{k_{0}} Q_{< j_{0}} (X^{0, \frac{1}{2}}_{1}) \subseteq S^{\mathrm{ang}}_{k_{0}, k_{0}}$, the first two terms can be treated by proceeding as in Case 1.2. On the other hand, for the third term, we use Lemma \ref{lem::stability4Sang} to estimate
\begin{equation*}
	\sup_{j_{0}} 2^{k_{0}} \nrm{P_{k_{0}} Q_{< j_{0}}(S_{<j_{0} - 20} \chi \varphi)}_{S^{\mathrm{ang}}_{k_{0}, j_{0}}}
	\aleq \nrm{\eta}_{L^{\infty}_{t,x}} \sum_{k_{2} \in [k_{0}-5, k_{0}+5]} 2^{k_{2}} \nrm{P_{k_{2}} \varphi}_{S_{k_{2}}},
\end{equation*}
which is square summable in $k_{0}$, proving \eqref{eq::nearCone:Sang}.

\pfstep{Step 2. Contribution of $\sum_{k_{0}} P_{k_{0}} Q_{> k_{0} + 25} (\chi \varphi)$}
When the output is away from the cone, the $\underline{X}$ norm dominates the whole $S^{1}$ norm. To see this, let $k_{0} \in \bbZ$. As $P_{k_{0}} (X^{0, \frac{1}{2}}_{1}) \subseteq S_{k_{0}}$, we have
\begin{align*}
\nrm{\rd_{t,x} P_{k_{0}} Q_{> k_{0} + 25} (\eta \varphi)}_{S_{k_{0}}}
\aleq & \sum_{j_{0} > k_{0} + 25} 2^{\frac{3}{2} j_{0}} \nrm{P_{k_{0}} Q_{j_{0}} (\eta \varphi)}_{L^{2}_{t,x}} \\
\aleq & \sum_{j_{0} > k_{0} + 25} 2^{\frac{1}{2} (k_{0} - j_{0})} \nrm{P_{k_{0}} Q_{j_{0}} (\eta \varphi)}_{\underline{X}} \\
\aleq & \nrm{P_{k_{0}} Q_{> k_{0} + 20}(\eta \varphi)}_{\underline{X}}.
\end{align*}
Thus by $L^{2}$ almost orthogonality,
\begin{equation} \label{eq::gt4S1:pf:farCone}
	\nrm{\sum_{k_{0}} P_{k_{0}} Q_{> k_{0} + 25} (\eta \varphi)}_{S^{1}}^{2} \aleq \sum_{k_{0}} \nrm{P_{k_{0}} Q_{> k_{0} + 20}(\eta \varphi)}_{\underline{X}}^{2}.
\end{equation}

To conclude the proof of \eqref{eq:gt4S1}, it remains to estimate the right-hand side of \eqref{eq::gt4S1:pf:farCone}. This is the content of Lemma \ref{lem::farCone} below. \qedhere
\end{proof}

\begin{lemma} \label{lem::farCone}
The following estimate holds.
\begin{align} 
	\bb( \sum_{k_{0}} \nrm{P_{k_{0}} Q_{>k_{0}+20} (\chi \varphi)}_{\underline{X}}^{2} \bb)^{\frac{1}{2}}
	\aleq & \nrm{\chi}_{\wCG} \nrm{\varphi}_{S^{1}}. \label{eq::farCone:S1} 
\end{align}
\end{lemma}
\begin{proof}
Since the spaces have different regularity in space and time, we will need to divide into cases depending on both the space and time frequency configurations. 
We begin with the standard Littlewood-Paley trichotomy in the spatial Fourier variable:
\begin{align*}
	P_{k_{0}} Q_{>k_{0}+20} (\chi \varphi)
	= & 	P_{k_{0}} Q_{> k_{0}+20} (\chi_{< k_{0}+10} \varphi_{[k_{0}-5, k_{0}+5]}) 
		+P_{k_{0}} Q_{>k_{0}+20} (\chi_{[k_{0}-5, k_{0}+5]} \varphi_{< k_{0}-5})\\
	&	+ \sum_{k_{1} \geq k_{0}+10} \sum_{k_{2} \in [k_{1}-5, k_{1}+5]} P_{k_{0}} Q_{> k_{0}+20} (\chi_{k_{1}} \varphi_{k_{2}}).
\end{align*}
In each case we will further divide into cases, which will essentially correspond to doing another round of Littlewood-Paley trichotomy in the temporal Fourier variable.
 
 \pfstep{Case 1. $(\mathrm{LH})$ interaction} Here we treat the contribution of 
 \begin{equation*}
P_{k_{0}} Q_{j_{0}} (\chi_{\leq k_{0} + 10} \varphi_{[k_{0}-5,k_{0}+5]})
\end{equation*}
We divide further into two sub-cases, depending on the temporal frequency of $\chi_{\leq k_{0}+10}$.

 \pfstep{Case 1.1 $\chi$ has high temporal frequency, $j_{1} > j_{0}-20$}
Recalling that $\underline{X}$ is a $L^{2}_{t,x}$ based norm, by orthogonality it suffices to estimate
\begin{equation} \label{eq::farCone:1:1}
\bnrm{\sum_{k_{2} \in [k_{0}-5, k_{0}+5]} \nrm{P_{k_{0}} Q_{j_{0}} (T_{> j_{0}-20} \chi_{\leq k_{0} +10} \varphi_{k_{2}})}_{\underline{X}}}_{\ell^{2}_{k_{0}, j_{0}} (j_{0} > k_{0} + 20)}
\end{equation}
We estimate each summand as follows:
\begin{align*}
& \hskip-2em	
\nrm{P_{k_{0}} Q_{j_{0}} (T_{> j_{0}-20} \chi_{\leq k_{0} + 10} \, \varphi_{k_{2}})}_{\underline{X}} \\
\aleq&	\sum_{k_{1} \leq k_{0} + 10} \sum_{j_{1} > j_{0} - 20} 
	2^{2 j_{0}} 2^{-\frac{1}{2} k_{0}} \nrm{T_{j_{1}} \chi_{k_{1}}}_{L^{2}_{t} L^{\infty}_{x}} \nrm{\varphi_{k_{2}}}_{L^{\infty}_{t} L^{2}_{x}} \\
\aleq&	\sum_{k_{1} \leq k_{0} + 10} \sum_{j_{1} > j_{0} - 20} 2^{2 (j_{0}-j_{1})} 2^{\frac{3}{2} (k_{1}-k_{0})}
	(2^{2 j_{1}} \nrm{T_{j_{1}} \chi_{k_{1}}}_{Y^{0,2}}) (2^{k_{1}} \nrm{\varphi_{k_{2}}}_{S_{k_{2}}}) 
\end{align*} 
We now sum up $k_{2} \in [k_{0}- 5, k_{0}+5]$ and take the $\ell^{2}_{k_{0}, j_{0}}(j_{0} > k_{0}+20)$ summation. Then \eqref{eq::farCone:1:1} is estimated by
\begin{equation*}
	\aleq \nrm{\varphi}_{S^{1}} \bnrm{\sum_{k_{1} \leq k_{0} + 10} \sum_{j_{1} \geq j_{0} - 20} 2^{2 (j_{0}-j_{1})} 2^{\frac{3}{2} (k_{1}-k_{0})}
	(2^{2 j_{1}} \nrm{T_{j_{1}} \chi_{k_{1}}}_{Y^{0,2}})}_{\ell^{2}_{k_{0}, j_{0}} (j_{0} > k_{0} + 20)}
\end{equation*}
which in turn is bounded by $\aleq \nrm{\varphi}_{S^{1}} \nrm{\chi}_{Y^{2,2}}$ thanks to \eqref{eq:sqsum4Y}.

 \pfstep{Case 1.2 $\chi$ has low temporal frequency, $j_{1} \leq j_{0}-20$}
It suffices to bound
\begin{equation} \label{eq::farCone:1:2}
	\bnrm{\sum_{k_{2} \in [k_{0}-5, k_{0}+5]} \nrm{P_{k_{0}} Q_{j_{0}} (T_{\leq j_{0} - 20} \chi_{\leq k_{0} +10} \varphi_{k_{2}})}_{\underline{X}}}_{\ell^{2}_{k_{0}, j_{0}} (j_{0} > k_{0} + 20)}
\end{equation}
By the restrictions on the Fourier supports of inputs and outputs, we can freely replace $\varphi_{k_{2}}$ by $\sum_{j_{2} \in[j_{0} - C, j_{0}+C]} Q_{j_{2}} \varphi_{k_{2}}$.
Thus throwing away $P_{k_{0}} Q_{j_{0}}$, estimating $T_{\leq j_{0} - 20} \chi_{\leq k_{0} + 10}$ in $L^{\infty}_{t,x}$ and $Q_{j_{2}} \varphi_{k_{2}}$ in $L^{2}_{t,x}$, we can estimate the summand in \eqref{eq::farCone:1:2} by
\begin{align*}
	\nrm{P_{k_{0}} Q_{j_{0}} (T_{\leq j_{0} - 20} \chi_{\leq k_{0} +10} \varphi_{k_{2}})}_{\underline{X}}
	\aleq \sum_{j_{2} \in [j_{0}-C, j_{0}+C]} \nrm{\chi}_{L^{\infty}_{t,x}} \nrm{Q_{j_{2}} \varphi_{k_{2}}}_{\underline{X}}
\end{align*}
Summing it up, we obtain $\eqref{eq::farCone:1:2} \aleq \nrm{\chi}_{L^{\infty}_{t,x}} \nrm{\varphi}_{\ell^{1} \underline{X}}$ as desired.

\pfstep{Case 2. $(\mathrm{HL})$ interaction}
Here we treat the contribution of
\begin{equation*}
P_{k_{0}} Q_{j_{0}} (\chi_{[k_{0} - 5, k_{0}+5]} \varphi_{< k_{0}-5})
\end{equation*}
As in the previous case, we divide into two sub-cases.

 \pfstep{Case 2.1 $\chi$ has high temporal frequency, $j_{1} > j_{0}-20$}
As in the previous case, we need to consider
\begin{equation} \label{eq::farCone:2:1}
	\bnrm{ \sum_{k_{1} \in [k_{0}- 5, k_{0}+5]} \sum_{k_{2} < k_{0} - 5} 2^{(s-1) k_{0}}
	\nrm{P_{k_{0}} Q_{j_{0}} (T_{> j_{0}-20} \chi_{k_{1}} \varphi_{k_{2}})}_{\underline{X}}
	}_{\ell^{2}_{k_{0}, j_{0}}(j_{0} > k_{0} + 20)}
\end{equation}
Disposing $P_{k_{0}} Q_{j_{0}}$ and using H\"older, we estimate each summand as 
\begin{align*}
\nrm{P_{k_{0}} Q_{j_{0}} (T_{> j_{0}-20} \chi_{k_{1}} \varphi_{k_{2}})}_{\underline{X}}
\aleq & \sum_{j_{1} > j_{0}-20}  2^{2 j_{0}} 2^{-\frac{1}{2} k_{0}} \nrm{T_{j_{1}} \chi_{k_{1}}}_{L^{2}_{t,x}} \nrm{\varphi_{k_{2}}}_{L^{\infty}_{t,x}} \\
\aleq & \sum_{j_{1} > j_{0}-20} 2^{2 (j_{0} - j_{1})} 2^{k_{2} - k_{0}} (2^{2 j_{1}} \nrm{T_{j_{1}} \chi_{k_{1}}}_{Y^{0,2}}) (2^{k_{2}} \nrm{\varphi_{k_{2}}}_{S_{k_{2}}}) \\
\end{align*}
Thanks to the high-low gain $2^{k_{2} - k_{0}}$, this can be summed up in $\ell^{2}_{k_{0}, j_{0}} (j_{0}> k_{0} + 20)$ using \eqref{eq:sqsum4Y}. We conclude
$\eqref{eq::farCone:2:1} \aleq \nrm{\chi}_{Y^{2,2}} \nrm{\varphi}_{S^{1}}$, as desired.

 \pfstep{Case 2.2 $\chi$ has low temporal frequency, $j_{1} \leq j_{0}-20$}
We consider
\begin{equation} \label{eq::farCone:2:2}
	\bnrm{ \sum_{k_{1} \in [k_{0}- 5, k_{0}+5]} \sum_{k_{2} < k_{0} - 5} 
		\nrm{P_{k_{0}} Q_{j_{0}} (T_{\leq j_{0}-20} \chi_{k_{1}} \varphi_{k_{2}})}_{\underline{X}}}_{\ell^{2}_{k_{0}, j_{0}}(j_{0} > k_{0} +20)}
\end{equation}
In this case, we can replace $\varphi_{k_{2}}$ by $\sum_{j_{2} \in [j_{0}-C, j_{0}+C]}Q_{j_{2}} \varphi_{k_{2}}$, thanks to the restrictions on the Fourier supports. Then as before, we estimate
\begin{align*}
\nrm{P_{k_{0}} Q_{j_{0}} (T_{\leq j_{0}-20} \chi_{k_{1}} \varphi_{k_{2}})}_{\underline{X}}
\aleq & \sum_{j_{2} \in [j_{0}-C, j_{0}+C]} 2^{2j_{0}} 2^{-\frac{1}{2} k_{0}} \nrm{T_{\leq j_{0}-20}\chi_{k_{1}}}_{L^{\infty}_{t} L^{2}_{x}} \nrm{Q_{j_{2}} \varphi_{k_{2}}}_{L^{2}_{t} L^{\infty}_{x}} \\
\aleq & \sum_{j_{2} \in [j_{0}-C, j_{0}+C]} 2^{\frac{5}{2} (k_{2}-k_{0})} (2^{2 k_{1}} \nrm{\chi_{k_{1}}}_{L^{\infty}_{t} L^{2}_{x}}) \nrm{Q_{j_{2}} \varphi_{k_{2}}}_{\underline{X}}
\end{align*}
Thanks again to the high-low gain $2^{\frac{5}{2} (k_{2}-k_{0})}$ this is again summable, and we obtain $\eqref{eq::farCone:2:2} \aleq \nrm{\chi}_{\ell^{1} L^{\infty}_{t} \dot{H}^{2}_{x}} \nrm{\varphi}_{\underline{X}}$.

\pfstep{Case 3. $(\mathrm{HH})$ interaction}
Here we treat the contribution of
\begin{equation*}
P_{k_{0}} Q_{j_{0}} (\chi_{k_{1}} \varphi_{k_{2}})
\end{equation*}
where $\abs{k_{1} - k_{2}} \leq 5$, $k_{1} \geq k_{0} + 10$.

 \pfstep{Case 3.1 $\chi$ has high spatial frequency, $k_{1} > j_{0} - 20$}
 We first consider
 \begin{equation} \label{eq::farCone:3:1}
	\bnrm{ \sum_{k_{1} > j_{0}- 20} \sum_{k_{2} \in [k_{1}-5, k_{1}+5]} \nrm{P_{k_{0}} Q_{j_{0}} (\chi_{k_{1}} \varphi_{k_{2}})}_{\underline{X}}}_{\ell^{2}_{k_{0}, j_{0}}(j_{0} > k_{0} +20)}
\end{equation}
Throwing away $Q_{j_{0}}$, applying Bernstein in space and using H\"older, we estimate each summand by
 \begin{align*}
\nrm{P_{k_{0}} Q_{j_{0}} (\chi_{k_{1}} \varphi_{k_{2}})}_{\underline{X}}
\aleq &	2^{2 j_{0}} 2^{\frac{3}{2} k_{0}} \nrm{\chi_{k_{1}} \varphi_{k_{2}}}_{L^{2}_{t} L^{1}_{x}} \\
\aleq &	2^{2 (j_{0}-k_{1})} 2^{\frac{3}{2} (k_{0}-k_{1})}  (2^{2 k_{1}} \nrm{\chi_{k_{1}}}_{Y^{0,2}}) (2^{k_{2}} \nrm{\varphi_{k_{2}}}_{S_{k_{2}}})
\end{align*}
Using \eqref{eq:sqsum4Y} and the square summability of $2^{k_{2}} \nrm{\varphi_{k_{2}}}_{S_{k_{2}}}$, the last expression can be summed up in the $\ell^{1}$ sense over $\set{(k_{0}, j_{0}, k_{1}, k_{2}) : j_{0} > k_{0}+20, k_{1} > j_{0} - 20, \abs{k_{1} - k_{2}} \leq 5}$ and be estimated by $\aleq \nrm{\chi}_{Y^{2,2}} \nrm{\varphi}_{S^{1}}$. 

  \pfstep{Case 3.2 $\chi$ has high temporal frequency, $k_{1} \leq j_{0}-20$, $j_{1} > j_{0}-20$}
Next, we estimate
  \begin{equation} \label{eq::farCone:3:2}
	\bnrm{ \sum_{k_{1} \in [k_{0}- 5, j_{0}-20]} \sum_{k_{2} \in [k_{1}-5, k_{1}+5]} \nrm{P_{k_{0}} Q_{j_{0}} (T_{> j_{0}-20} \chi_{k_{1}} \varphi_{k_{2}})}_{\underline{X}}}_{\ell^{2}_{k_{0}, j_{0}}(j_{0} > k_{0} + 20)}
\end{equation}
Throwing away $Q_{j_{0}}$, applying Bernstein in space and using H\"older, we have
 \begin{align*}
\nrm{P_{k_{0}} Q_{j_{0}} (T_{> j_{0}-20} \chi_{k_{1}} \varphi_{k_{2}})}_{\underline{X}}
\aleq &	\sum_{j_{1} > j_{0} - 20} 2^{2 j_{0}} 2^{\frac{3}{2} k_{0}} \nrm{T_{j_{1}} \chi_{k_{1}} \varphi_{k_{2}}}_{L^{2}_{t} L^{1}_{x}} \\
\aleq &	\sum_{j_{1} > j_{0}-20} 2^{2 (j_{0}-j_{1})} 2^{\frac{3}{2} (k_{0}-k_{1})}  (2^{2 j_{1}} \nrm{T_{j_{1}}\chi_{k_{1}}}_{Y^{0,2}}) (2^{k_{2}} \nrm{\varphi_{k_{2}}}_{S_{k_{2}}})
\end{align*}
Using the triangle inequality to pull out $k_{1}$, $k_{2}$ summations out of $\ell^{2}_{k_{0}, j_{0}}(j_{0} > k_{0}+20)$ and performing the latter summation, we estimate \eqref{eq::farCone:3:2} by
\begin{equation*}
	\aleq \sum_{k_{1}} \sum_{k_{2} \in [k_{1}-5, k_{1}+5]} 
		(\sum_{j_{1} > k_{1}-20} 2^{4 j_{1}} \nrm{T_{j_{1}}\chi_{k_{1}}}_{Y^{0,2}}^{2})^{\frac{1}{2}}
			( 2^{k_{2}} \nrm{\varphi_{k_{2}}}_{S_{k_{2}}}),
\end{equation*}
which is estimated by $\aleq \nrm{\chi}_{Y^{2,2}} \nrm{\varphi}_{S^{1}}$ using \eqref{eq:sqsum4Y} and the square summability of $2^{k_{2}} \nrm{\varphi_{k_{2}}}_{S_{k_{2}}}$.

\pfstep{Case 3.3 $\chi$ is close to frequency origin, $k_{1} \leq j_{0}-20$, $j_{1} \leq j_{0}-20$}
In this case, we estimate
  \begin{equation} \label{eq::farCone:3:3}
	\bnrm{ \sum_{k_{1} \in [k_{0}- 5, j_{0}-20]} \sum_{k_{2} \in [k_{1}-5, k_{1}+5]} 
		\nrm{P_{k_{0}} Q_{j_{0}} (T_{\leq j_{0}-20} \chi_{k_{1}} \varphi_{k_{2}})}_{\underline{X}}
			}_{\ell^{2}_{k_{0}, j_{0}}(j_{0} > k_{0} + 20)}
\end{equation}
As before, the restrictions on Fourier supports allow us to replace $\varphi_{k_{2}}$ by the expression $\sum_{j_{2} \in [j_{0} - C, j_{0} + C]} Q_{j_{2}} \varphi_{k_{2}}$. 
Throwing away $Q_{j_{0}}$, applying Bernstein and using H\"older (and furthermore the fact that $T_{\leq j_{0} - 20}$ is bounded in $L^{\infty}_{t} L^{2}_{x}$), the summand in \eqref{eq::farCone:3:3} is estimated by
\begin{equation*}
	\nrm{P_{k_{0}} Q_{j_{0}} (T_{\leq j_{0}-20} \chi_{k_{1}} \varphi_{k_{2}})}_{\underline{X}}
	\aleq \sum_{j_{2} \in [j_{0} - C, j_{0}+C]} 
		2^{\frac{3}{2} (k_{0} - k_{1})}  (2^{2 k_{1}}\nrm{\chi_{k_{1}}}_{L^{\infty}_{t} L^{2}_{x}}) \nrm{Q_{j_{2}} \varphi_{k_{2}}}_{\underline{X}}.
\end{equation*}
The last expression can be summed up using \eqref{eq:sqsum4Y} and $\ell^{2}_{k_{2}, j_{2}}$ summability of $\nrm{Q_{j_{2}} \varphi_{k_{2}}}_{\underline{X}}$, leading to $\eqref{eq::farCone:3:3} \aleq \nrm{\chi}_{L^{\infty}_{t} \dot{H}^{2}_{x}} \nrm{\varphi}_{\underline{X}}$ as desired.
\end{proof}

\subsection{Cutoff estimates} \label{subsec:cutoff}
In this subsection, we prove Lemmas \ref{lem:Xcutoff} and \ref{lem:ellipticEst4wCG}. 

We begin with a brief discussion on $\dot{B}^{\frac{5}{2}, 2}_{1}$,
which basically plays the role of the space of smooth cutoffs. Recall
that $\dot{B}^{\frac{5}{2},2}_{1}$ is an atomic space, whose atoms
satisfy $\eta = S_{[\ell-1,\ell+1]} \eta$ and $2^{\frac{5}{2} \ell}
\nrm{\eta}_{L^{2}_{t,x}} \leq 1$ for some $\ell \in \bbZ$. Note that
the following $\ell^{1}$ summability estimate holds:
\begin{equation} \label{eq::besovEst}
	\sum_{\ell} 2^{\frac{5}{2} \ell} \nrm{S_{\ell} \eta}_{L^{2}_{t, x}} + \sum_{\ell} 2^{2 \ell} \nrm{S_{\ell} \eta}_{L^{\infty}_{t} L^{2}_{x}}+ \nrm{\eta}_{L^{\infty}_{t,x}} \aleq \nrm{\eta}_{\dot{B}^{\frac{5}{2}, 2}_{1}}
\end{equation}
Note furthermore that 
\begin{equation} \label{eq:embedding4besov}
\dot{B}^{\frac{5}{2}, 2}_{1} 
\subseteq \dot{H}^{\frac{5}{2}}_{t,x} \cap \ell^{1} C^{0}_{t} \dot{H}^{2}_{x}
\subseteq \wCG
\end{equation}
which follows easily from Bernstein's inequality.


We first establish Lemma \ref{lem:Xcutoff}. By the definition of restriction spaces, it suffices to prove the following global statement.
\begin{lemma} \label{lem:Xcutoff:restate}
The following estimate holds for $X = Y^{1}, S^{1}, \wCG$ or $\CG$.
\begin{equation} \label{eq:Xcutoff:global}
	\nrm{\eta \varphi}_{X(\bbR^{1+4})} 
	\aleq \nrm{\eta}_{\dot{B}^{\frac{5}{2}, 2}_{1}(\bbR^{1+4})} \nrm{\varphi}_{X(\bbR^{1+4})}.
\end{equation}
\end{lemma}

\begin{proof} 
Before we begin, note that the following cutoff estimates hold:
\begin{align} 
	\nrm{\eta \varphi}_{Y^{1,2}} 
	\aleq & \nrm{\eta}_{\dot{B}^{\frac{5}{2}, 2}_{1}} \nrm{\varphi}_{Y^{1,2}}, \label{eq:cutoff4Y12} \\
	\nrm{\eta \varphi}_{\ell^{1} Y^{2,2}} 
	\aleq & \nrm{\eta}_{\dot{B}^{\frac{5}{2}, 2}_{1}} \nrm{\varphi}_{\ell^{1} Y^{2,2}}. \label{eq:cutoff4l1Y22}
\end{align}
Indeed, both estimates can be proved in a similar manner as \eqref{eq:alg4Y}; we omit the details.
With \eqref{eq:cutoff4Y12} and \eqref{eq:cutoff4l1Y22} in our hand, we proceed to the proof of \eqref{eq:Xcutoff:global}.

\pfstep{Case 1: $X = Y^{1}$}
Recall that $Y^{1} = Y^{1,2} \cap Y^{1, \infty}$. The desired estimate for the $Y^{1,2}$ norm of $\eta \varphi$ follows from \eqref{eq:cutoff4Y12}; thus it remains to bound $\nrm{\eta \varphi}_{Y^{1, \infty}}$. By the Leibniz rule, H\"older, $\dot{H}^{1}_{x} \subseteq L^{4}_{x}$ Sobolev and \eqref{eq::besovEst}, we have
\begin{align*}
	\nrm{\rd_{t,x} (\eta \varphi)}_{Y^{0,\infty}}
	\aleq & \nrm{\eta \rd_{t,x}\varphi}_{L^{\infty}_{t} L^{2}_{x}}
		+ \nrm{\rd_{t,x} \eta \varphi}_{L^{\infty}_{t} L^{2}_{x}} \\
	\aleq & (\nrm{\rd_{t,x} \eta}_{L^{\infty}_{t} L^{4}_{x}} + \nrm{\eta}_{L^{\infty}_{t,x}})
		(\nrm{\rd_{t,x} \varphi}_{L^{\infty}_{t} L^{2}_{x}} + \nrm{\varphi}_{L^{\infty}_{t} L^{4}_{x}}) \\
	\aleq & \nrm{\eta}_{B^{\frac{5}{2},2}_{1}} \nrm{\varphi}_{Y^{1}},
\end{align*}
which completes the proof in this case.

\pfstep{Cases 2 \& 3: $X = S^{1}$ or $\wCG$}
These cases are immediate consequences of \eqref{eq:alg4Y}, \eqref{eq:gt4S1} and the embedding \eqref{eq:embedding4besov}. 

\pfstep{Case 4: $X = \CG$}
Recall that $\CG = \ell^{1} Y^{2, 2} \cap \ell^{1} Y^{2, \infty}$. For the $\ell^{1} Y^{2, 2}$ norm of $\eta \varphi$, we use \eqref{eq:cutoff4l1Y22}. In order to bound the $\ell^{1} Y^{2, \infty}$ norm of $\eta \varphi$, we first use the Leibniz rule to compute
\begin{equation*}
	\rd_{t} (\eta \varphi) = \rd_{t} \eta \varphi + \eta \rd_{t} \varphi, \quad
	\rd_{t}^{2} (\eta \varphi) = \rd_{t}^{2} \eta \varphi + 2 \rd_{t} \eta \rd_{t} \varphi + \eta \rd_{t}^{2} \varphi.
\end{equation*}
By the embedding $\dot{B}^{N+\frac{1}{2}, 2}_{1} \subseteq \ell^{1}
C^{0} \dot{H}^{N}_{x}$ and the definition of the space $\ell^{1} Y^{2,
  \infty}$, we have $\rd_{t}^{(N)} \eta, \rd_{t}^{(N)} \varphi \in
\ell^{1} C^{0}_{t} \dot{H}^{2-N}_{x}$ for $N = 0, 1, 2$. Thus the
desired estimate is easily obtained using the standard
Littlewood-Paley trichotomy; we leave the details to the
reader. \qedhere
\end{proof}

Finally, we give a proof of Lemma \ref{lem:ellipticEst4wCG}. Extending $\eta$ and $\varphi$ to the whole space in such a way that $\eta \in \dot{B}^{\frac{5}{2}, 2}_{1}(\bbR \times \bbR^{4})$ and $\varphi \in \wCG(\bbR \times \bbR^{4})$, it suffices to consider the case $I = \bbR$. Thus Lemma \ref{lem:ellipticEst4wCG} would follow once we establish the following statement.
\begin{lemma} \label{lem:ellipticEst4wCG:restate}
Let $\eta \in \dot{B}^{\frac{5}{2}, 2}_{1}(\bbR^{1+4})$ and $\varphi \in \wCG (\bbR^{1+4})$. Let $\chi := (-\lap)^{-1} (\eta \lap \varphi)(t)$ be given as convolution with the Newton potential. Then we have
\begin{equation} \label{eq:ellipticEst4wCG:restate}
	\nrm{\chi}_{\wCG}
	\aleq  \nrm{\eta}_{\dot{B}^{\frac{5}{2}, 2}_{1}}  \nrm{\varphi}_{\wCG} \, .
\end{equation}
\end{lemma}

\begin{proof} 

From the embedding \eqref{eq:embedding4besov}, it easily follows that $\eta \lap \varphi \in \ell^{1} C^{0}_{t} L^{2}_{x}(\bbR \times \bbR^{4})$ with
\begin{equation*}
	\nrm{\eta \lap \varphi}_{\ell^{1} L^{\infty}_{t} L^{2}_{x}}
	\aleq \nrm{\eta}_{\ell^{1} L^{\infty}_{t} \dot{H}^{2}_{x}} \nrm{\lap \varphi}_{\ell^{1} L^{\infty}_{t} \dot{H}^{2}_{x}}
	\aleq \nrm{\eta}_{\dot{B}^{\frac{5}{2}, 2}_{1}} \nrm{\varphi}_{\wCG}.
\end{equation*}
Therefore, the estimate for $\nrm{\chi}_{\ell^{1} L^{\infty}_{t}
  \dot{H}^{2}_{x}}$ in the $\wCG$ norm in
\eqref{eq:ellipticEst4wCG:restate} follows. It remains to establish
the estimate for the $Y^{2,2}$ norm in
\eqref{eq:ellipticEst4wCG:restate}; for this we will show that
\begin{equation}
\| \chi\|_{Y^{2,2}} \lesssim \| \eta\|_{\dot{B}^{\frac{5}{2}, 2}_{1}}  \nrm{\varphi}_{Y^{2,2}} 
\end{equation}
The left-hand side is equivalent to $\| \partial_{x,t}^2 (\eta \Delta
\varphi) \|_{L^2_{t} \dot H^{-\frac32}_{x}}$. We apply the Leibniz rule to write
\[
\partial_{x,t}^2 (\eta \Delta \varphi) = \partial_{x,t}^2 \eta \Delta \varphi +
\partial_{x,t} \eta \partial_{x,t}\Delta \varphi +  \eta \partial_{x,t}^2 \Delta \varphi
\]
We can estimate 
\[
\|  \Delta \varphi \|_{L^2_{t} \dot H^{\frac12}_{x}} + \| \partial_{x,t}\Delta \varphi\|_{L^2_{t} \dot H^{-\frac12}_{x}} 
+  \| \partial_{x,t}^{2} \Delta \varphi\|_{L^2_{t} \dot H^{-\frac32}_{x}}  \lesssim  \nrm{\varphi}_{Y^{2,2}} 
\]
and, by the trace theorem, 
\[
\|  \partial_{x,t}^2 \eta\|_{L^\infty_{t} L^2_{x}} + \|  \partial_{x,t} \eta\|_{L^\infty_{t} \dot H^1} 
+ \| \eta\|_{L^\infty_{t} \ell^1 \dot H^2_{x}} \lesssim \| \eta\|_{\dot{B}^{\frac{5}{2}, 2}_{1}}
\]
Hence it remains to establish the fixed time multiplicative estimates
\[
\dot H^\frac12 \times L^2_{t} \to \dot H^{-\frac32}_{x} , \qquad \dot H^{-\frac12} \times \dot H^1 \to 
\dot H^{-\frac32}, \qquad \dot H^{-\frac32} \times  \ell^1 \dot H^2_{x} \to \dot H^{-\frac32} 
\]
These in turn are easily obtained using the standard Littlewood-Paley trichotomy. \qedhere
\end{proof}

\bibliographystyle{amsplain}
\providecommand{\bysame}{\leavevmode\hbox to3em{\hrulefill}\thinspace}
\providecommand{\MR}{\relax\ifhmode\unskip\space\fi MR }
\providecommand{\MRhref}[2]{%
  \href{http://www.ams.org/mathscinet-getitem?mr=#1}{#2}
}
\providecommand{\href}[2]{#2}

\end{document}